\documentclass[12pt]{amsart}
\usepackage{amsmath}
\usepackage{mathdots}
\usepackage{amssymb}
\usepackage[all]{xy}
\usepackage{tikz}
\usepackage{adjustbox}

\numberwithin{equation}{section}

\newtheorem{thm}{Theorem}[section]

\newtheorem{lem}[thm]{Lemma}
\newtheorem{prop}[thm]{Proposition}
\newtheorem{cor}[thm]{Corollary}
\newtheorem{rem}[thm]{Remark}
\newtheorem{exam}[thm]{Example}
\newtheorem{exam-nota}[thm]{Example-Notation}
\newtheorem{rem-nota}[thm]{Remark-Notation}
\newtheorem{nota}[thm]{Notation}
\newtheorem{dfn}[thm]{Definition}

\newtheorem{dfn-nota}[thm]{Definition-Notation}

\newtheorem{dfn-lem}[thm]{Lemma-Definition}
\newtheorem{dfn-prop}[thm]{Proposition-Definition}

\newcommand{\beqa}{\begin{eqnarray*}}
\newcommand{\eeqa}{\end{eqnarray*}}

\newcommand{\lspan}{\mbox{${\rm span}$}}

\newcommand{\tsigma}{\mbox{$\tilde{\sigma}$}}

\newcommand{\fk}{\mbox{${\mathfrak k}$}}
\newcommand{\fg}{\mbox{${\mathfrak g}$}}

\newcommand{\fl}{\mbox{${\mathfrak l}$}}
\newcommand{\fs}{\mbox{${\mathfrak s}$}}
\newcommand{\fsl}{\mbox{${\fs\fl}$}}

\newcommand{\fh}{\mbox{${\mathfrak h}$}}

\newcommand{\fp}{\mbox{${\mathfrak p}$}}
\newcommand{\fr}{\mbox{${\mathfrak r}$}}
\newcommand{\fb}{\mbox{${\mathfrak b}$}}
\newcommand{\fz}{\mbox{${\mathfrak z}$}}
\newcommand{\fm}{\mbox{${\mathfrak m}$}}

\newcommand{\Gr}{{\rm Gr}}

\newcommand{\supp}{{\rm supp}}

\newcommand{\eps}{\epsilon}

\newcommand{\C}{\mbox{${\mathbb C}$}}

\newcommand{\Ad}{{\rm Ad}}

\newcommand{\fgl}{\mathfrak{gl}}

\newcommand{\B}{\mathcal{B}}

\newcommand{\fso}{\mathfrak{so}}

\newcommand{\he}{\hat{e}}
\newcommand{\orbitbundle}{\mathcal{O}(Q_{\fr},Q_{\fl})}
\newcommand{\F}{\mathcal{F}}

\newcommand{\te}{\widetilde{e}}
\newcommand{\OGr}{\mbox{OGr}}
\newcommand{\SPILzero}{SPIL(\ell \cup \{ 0 \})}
\newcommand{\SPILell}{SPIL(\ell)}
\newcommand{\Morbitspace}{M\backslash \B}
\newcommand{\Borbitspace}{B_{n-1}\backslash\B_{n}}
\title{$B_{n-1}$-bundles on the Flag Variety, II}

\author[M. Colarusso]{Mark Colarusso}
\address{Department of Math and Stats, University of South Alabama, Mobile, AL, 36608}
\email{mcolarusso@southalabama.edu}

\author[S. Evens]{Sam Evens}
\address{Department of Mathematics, University of Notre Dame, Notre Dame, IN, 46556}
\email{sevens@nd.edu}

\begin{document}
\maketitle

\begin{abstract}
This paper is the sequel to ``$B_{n-1}$-bundles on the flag variety, I".  
We continue our study of the orbits of a Borel subgroup $B_{n-1}$ of $G_{n-1}=GL(n-1)$  (resp.  $SO(n-1)$) acting on the flag variety $\B_{n}$ of $G=GL(n)$ (resp. $SO(n)$).  We begin by using the results of the first paper to obtain a complete combinatorial model of the $B_{n-1}$-orbits on $\B_{n}$ in terms of partitions into lists.  The model allows us to obtain explicit formulas for the number of orbits as well as the exponential generating functions for the sequences $\{|\Borbitspace|\}_{n\geq 1}$ .  We then use the combinatorial description of the orbits to construct a canonical set of representatives of the orbits in terms of flags.  These representatives allow us to understand an extended monoid action on $\Borbitspace$ using simple roots of both $\fg_{n-1}$ and $\fg$ and show that the closure ordering on $\Borbitspace$ is
the standard ordering of Richardson and Springer.



\end{abstract}

\section{Introduction}\label{s:intro}

Let $G=G_{n}=GL(n)$ or $ SO(n)$ and let $B_{n-1}$ be the standard upper triangular Borel subgroup of 
$G_{n-1}\subset G$.  In this paper, we continue our study of the $B_{n-1}$-orbits on the flag variety $\B_{n}$ of 
$G$ using the results of the accompanying paper ``$B_{n-1}$-bundles on the flag variety, I".  We obtain a combinatorial model for the set of orbits $B_{n-1}\backslash\B_{n}$ and use this model to develop a canonical set of representatives for the orbits in terms of flags.   We then use these representatives 
along with our results from \cite{CE21I} to show that every $B_{n-1}$-orbit arises from a zero dimensional orbit using an extended monoid action by simple roots of $\fk:=\fg_{n-1}$ and $\fg$.  Given this result, we show using a theorem of Richardson and Springer in \cite{RS} that the closure ordering on $B_{n-1}\backslash\B_{n}$ is given by the standard order (Corollary \ref{c:bruhatisstandard}).  In the general linear case, $B_{n-1}$-orbits are described in \cite{Hashi} and in \cite{MWZ}.  In a subsequent paper \cite{Magyar}, Magyar gives a combinatorial description of the closure relations.  Our approach is more geometric and has the advantage that both the general linear and orthogonal cases can be described in an essentially uniform manner.  We also establish a complete picture of the extended monoid action.  It would be interesting to understand the relations between the different descriptions of the orbits in the general linear case.  We note that the orthogonal case does not seem to be accessible using the previous approaches.    



In the future, we plan to use the closure relations for these orbits to understand the category of $(\fg, B_{n-1})$-modules, i.e., $U(\fg)$-modules with a locally finite action of the group $B_{n-1}$.  Such modules can be constructed using local systems on $B_{n-1}$-orbits on $\B_{n}$ and the Beilinson-Bernstein correspondence.  
The bundle description of orbits from \cite{CE21I} and our extended monoid action should both be essential ingredients in describing the $(\fg, B_{n-1})$-modules.  The category of $(\fg, B_{n-1})$-modules is closely related to the category of Gelfand-Zeitlin (GZ) modules which quantize the complex GZ integrable systems.  GZ modules were first introduced by Drozd, Futorny, and Ovsienko in \cite{DFO} and are increasingly becoming a broad area of research.

In more detail, in the case $\fg=\fgl(n)$, we develop a
bijection between $\Borbitspace$ and the set of 
all partitions of $\{1,\dots, n\}$ into ordered subsets (Theorem \ref{thm:counttypeA}).  
We refer to any such partition as a \emph{PIL} (partition into lists) of the set $\{1,\dots, n\}$.  
In Theorem \ref{thm:orthocount}, we extend this result to the orthogonal case, associating to each $B_{n-1}$-orbit on $\B_{n}$ a partition into certain signed lists (\emph{SPIL}) of the set $\{1,\dots, \ell\}$ where $\ell=\mbox{rank}(\fso(n))$. These combinatorial models allow us to compute explicit
formulas and exponential generating functions for the sequence $\{|\Borbitspace|\}_{n\in\mathbb{N}}$ in the various cases.  When $\fg$ is a type A or D Lie algebra these formulas can be expressed succinctly in terms of the classical Lah numbers, which play an important role in enumerative and applied combinatorics (Propositions \ref{p:formulaA} and \ref{p:formulaD}).  In type B a variant of Lah numbers is required.  

We then use the combinatorial models to develop explicit representatives for the $B_{n-1}$-orbits 
in terms of flags.  We call these representatives flags in standard form and note that they depend on the type of $\fg$.  We show that any $B_{n-1}$-orbit contains a unique flag in standard form by first using the parameterization of
$\Borbitspace$ by PILs and SPILs to show that $|\Borbitspace|$ is the same as the cardinality of 
the set of all flags in standard form.  We then show that any two flags in standard form contained in a given 
$B_{n-1}$-orbit must coincide.  The key ingredient in this last step is the classification of the $B_{n-1}$-orbits on the Grassmannian of (isotropic) $i$-planes in $\C^{n}$.

The representatives of $\Borbitspace$ in standard form facilitate easy computation of certain instances of the monoid action that are not covered by the results of \cite{CE21I} (Section \ref{s:monoid}).  These computations along with the results of \emph{loc.cit.} allow us to prove that the closure ordering on $\Borbitspace$ is given by the standard order (Corollary \ref{c:bruhatisstandard}).  

The paper is organized as follows.  In Section \ref{s:Prelim}, we review some standard facts about the Lie
algebras $\fgl(n)$ and $\fso(n)$ and discuss the parameterization of $K:=G_{n-1}$-orbits on $\B_{n}$ using flags and isotropic flags.  In Section \ref{s:PILS}, we develop the combinatorial models for $\Borbitspace$ by PILs and SPILs and prove that our models classify the orbits.  In \ref{ss:formulae}, we give explicit formulas for the sequence $\{|\Borbitspace|\}_{n\in\mathbb{N}}$ in the various cases and compute the corresponding exponential generating functions.  Section \ref{s:std} is devoted to proving that every element of $\Borbitspace$ contains a unique flag in standard from.  The various types are discussed separately.  In Section \ref{s:monoid}, after reviewing the monoid action on both $K\backslash \B_{n}$ and $\Borbitspace$ via simple roots of $\fg$, we use the representatives of $B_{n-1}$-orbits in standard form to compute certain monoid actions on $\Borbitspace$ that are needed later.  The results of Section \ref{s:monoid} along with the results of \emph{loc.cit.} are then used in Section \ref{s:closure} to prove that closure ordering on $\Borbitspace$ coincides with the standard ordering.  These results use crucially the extended monoid action and provide an answer to a conjecture posed by Hashimoto in the Introduction and Section 4 of \cite{Hashi}.   Section \ref{s:closure} begins with a review of the work of Richardson and Springer and then we prove that the minimal elements for the extended monoid action are the zero dimensional orbits.  Finally, in Section \ref{s:graphs}, we give examples of Bruhat graphs for the set $\Borbitspace$ in all types in low rank, labelling the different monoid actions and relations in the standard order that cannot be obtained from the weak order.

We would like to thank Jeb Willenbring and David Galvin for helpful discussions on the material in Section \ref{s:PILS}.  We would also like to thank Jacopo Gandini, Guido Pezzini, and Friedrich Knop for discussions relevant to the material in this paper, and thank Michael Finkelberg for bringing the work of Magyar in \cite{Magyar} to our attention.  

During the preparation of this paper, the first author was supported in part by the National Security Agency grant H98230-16-1-0002 and the second author was supported in part by the Simons Foundation grant 359424. 

\section{Preliminaries}\label{s:Prelim}

In this paper, we study the complex general linear and orthogonal Lie algebras.  
We let $G=G_{n}=GL(n)$  (resp. $SO(n)$) and let $\fg=\fg_{n}=\fgl(n)$ (resp. $\fso(n)$) be 
the corresponding Lie algebras.  The group $G$ acts on its Lie algebra $\fg$ via
the adjoint action $\Ad(g)X=gXg^{-1}$ for $g\in G$, $X\in \fg$.  

\subsection{Realizations of Lie algebras}\label{ss:realization}
We denote the standard basis of $\C^{n}$ by 
$\{e_{1},\dots, e_{n}\}$.  The Lie algebra $\fgl(n)$ is 
the Lie algebra of $n\times n$ complex matrices.  To realize 
the orthogonal Lie algebra, $\fso(n)$, we consider the non-degenerate, 
symmetric bilinear form $\beta$ on $\C^{n}$ given by
\begin{equation}\label{eq:beta}
\beta(x,y)=x^{T} S_{n} y, 
\end{equation}
where $x, y$ are $n\times 1$ column vectors and $S_{n}$ is the permutation matrix mapping $e_i$ to $e_{n+1-i}$ for $i=1, \dots, n$.
The special orthogonal group is 
$$
SO(n):=\{g\in SL(n): \, \beta(gx, gy)=\beta(x,y)\; \forall\; x,\, y \in \C^{n}\}
$$
and its Lie algebra is
$$
\fso(n)=\{Z\in\fgl(n):\, \beta(Zx, y)=-\beta(x,Zy)\,\forall\; x, \, y\in\C^{n}\}.  
$$
It follows that $\fso(n)$ is the Lie algebra of all $n\times n$ complex matrices
which are skew-symmetric about the skew-diagonal.  

For $\fg=\fgl(n)$ or $\fso(n)$,  
let $\fh$ be the Cartan subalgebra of diagonal matrices in $\fg$.  We denote the rank of the Lie algebra $\fg$ by 
$\mbox{rk}(\fg)=\dim\fh$.  If $\fg=\fgl(n)$, then $\mbox{rk}(\fg)=n$, and 
$\fh=\{\mbox{diag}[x_{1},\dots, x_{n}], \, x_{i}\in\C\}$.  We let 
$\eps_{i}\in\fh^{*}$ be the linear functional $\eps_{i}(\mbox{diag}[x_{1},\dots, x_{n}])=x_{i}$.  
Then the standard simple roots for $\fg$ are $\Pi_{\fg}=\{\alpha_{1},\dots, \alpha_{n-1}\}$, with 
$\alpha_{i}=\eps_{i}-\eps_{i+1}$.  

For $\fg=\fso(2\ell)$ (type D), $\mbox{rk}(\fg)=\ell$.  
In this case, it is convenient to relabel part of the basis of $\C^{2\ell}$ as $e_{-j}:=e_{2\ell+1-j}$ for 
$j=1,\dots, \ell$.  Then the Cartan subalgebra $\fh=\{\mbox{diag}[x_{1},\dots, x_{\ell}, -x_{\ell},\dots, -x_{1}],\, x_{i}\in\C\}$, 
and we let $\eps_{i}\in\fh^{*}$ be $\eps_{i}(\mbox{diag}[x_{1},\dots, x_{\ell}, -x_{\ell},\dots, -x_{1}])=x_{i}$ for 
$i=1,\dots, \ell$.  The standard simple roots of $\fg$ are 
$\Pi_{\fg}=\{\alpha_{1},\dots, \alpha_{\ell-1}, \alpha_{\ell}\}$ with 
$\alpha_{i}=\eps_{i}-\eps_{i+1}$, $i=1,\dots, \ell-1$, and $\alpha_{\ell}=\eps_{\ell-1}+\eps_{\ell}$.  

For $\fg=\fso(2\ell + 1)$ (type B), $\mbox{rk}(\fg)=\ell$.  In this case, we relabel the basis 
of $\C^{2\ell+1}$ as $e_{-j}:=e_{2\ell+2-j}$ for $j=1,\dots, \ell$ and $e_{0}:=e_{\ell+1}$.  
The Cartan subalgebra $\fh=\{\mbox{diag}[x_{1},\dots, x_{\ell}, 0, -x_{\ell},\dots, -x_{1}], x_{i}\in\C\}$ and 
$\eps_{i}\in\fh^{*}$ is given by $\eps_{i}(\mbox{diag}[x_{1},\dots, x_{\ell}, 0, -x_{\ell},\dots, -x_{1}])=x_{i}$.  
The standard simple roots for $\fg=\fso(2\ell+1)$ are 
$\Pi_{\fg}=\{\alpha_{1},\dots, \alpha_{\ell-1}, \alpha_{\ell}\}$ with $\alpha_{i}=\eps_{i}-\eps_{i+1}$ for 
$i=1,\dots, \ell-1$ and $\alpha_{\ell}=\eps_{\ell}$.  
\begin{rem}\label{r:beta}
The form $\beta$ in Equation (\ref{eq:beta}) with respect to the standard
bases of $\C^{2\ell}$ and $\C^{2\ell+1}$ is determined by the property that $\beta(e_{j}, e_{k})=\delta_{j,-k}$.  
\end{rem}

Let $H\subset G$ be the Cartan subgroup of diagonal matrices with Lie algebra $\fh$. 
Let $W=W_{G}=N_{G}(H)/H$ be the Weyl group of $G$ with respect to $H$.  

\begin{nota}\label{nota:Weyl}
 For an element $w\in W$, let $\dot{w}\in N_{G}(H)$ be a representative of $w$.   
 If $\fm\subset\fg$ is a Lie subalgebra which is normalized by $H$, then $\Ad(\dot{w})\fm$ is independent of the choice of representative of $w$,
 and we denote it by $w(\fm)$. 
\end{nota}

\subsection{Borel subalgebras and flag varieties}\label{ss:Borels}
We denote the standard Borel subalgebra of upper triangular matrices 
in $\fgl(n)$ by $\fb_{n}$.   For $\fg=\fso(n)$, we let $\fb_{n}$ denote the upper triangular matrices in $\fg$ with respect to the ordered basis $e_1, \dots, e_{\ell}, e_{-\ell}, \dots, e_{-1}$ when $n=2\ell$ is even, and with respect to the ordered basis $e_1, \dots, e_{\ell}, e_0, e_{-\ell}, \dots, e_{-1}$ when $n=2\ell + 1$ is odd.    Note that $\fh\subset\fb_{n}$ and the standard simple roots $\Pi_{\fg}$ given above are simple roots for $\fb_{n}$.  We let $\B_{n}:=\B_{\fg}$ be the variety of all 
Borel subalgebras of $\fg$.  
When $\fg=\fgl(n)$, the variety $\B_{n}$ is $G$-equivariantly isomorphic to the
variety of all full flags on $\C^{n}$, i.e., flags of the form
$$
\{\mathcal{F}=(V_{1}\subset\dots\subset V_{i}\subset\dots\subset V_{n}) |\, \dim V_{i}=i,\, V_{i}\subset \C^{n}\}.
$$
When $\fg=\fso(2\ell+1)$, the variety 
$\B_{2\ell+1}$ is $G$-equivariantly isomorphic to the variety of all maximal isotropic flags in $\C^{2\ell+1}$, i.e., partial flags of the form
$$
\mathcal{F}=( V_{1}\subset\dots\subset V_{i}\subset\dots\subset V_{\ell}),
$$
with $\dim V_i=i$ and $\beta(u,w)=0$ for all $u,w \in V_i.$ 
  In the type D case, the variety
$\B_{2\ell}$ is $G$-equivariantly isomorphic to the variety of all partial isotropic flags in $\C^{2\ell}$ of the form: 
$$
\mathcal{F}=( V_{1}\subset\dots\subset V_{i}\subset\dots\subset V_{\ell-1}).
$$
This is a different realization of $\B_{2\ell}$ than the one used in \cite{CE21I}.
We make heavy use of the following shorthand notation for flags throughout the paper.  

\begin{nota} \label{nota:standard}
  Let 
   $$
  \mathcal{F}=(V_{1}\subset V_{2}\subset\dots\subset V_{i}\subset V_{i+1}\subset \dots)
   $$
be a flag in $\C^{n}$, with $\dim V_{i}=i$ and $V_{i}=\mbox{span}\{v_{1},\dots, v_{i}\}$, with each $v_{j}\in\C^{n}$.  We will denote this flag $\mathcal{F}$ 
by
   $$
  \mathcal{F}=  (v_{1}\subset  v_{2}\subset\dots\subset v_{i}\subset v_{i+1}\subset\dots ). 
   $$
\end{nota}

\begin{nota}\label{note:uppertriangular}
Identifying $\B_{n}$ with various flag varieties as above, the standard Borel subalgebra of 
upper triangular matrices is identified with the flags
\begin{equation}\label{eq:upper}
\F_{+}:=\left\lbrace\begin{array}{ll} (e_{1}\subset e_{2}\subset \dots \subset e_{n})  &\mbox{ for } \fg=\fgl(n)\\
(e_{1}\subset\dots\subset e_{\ell}) &\mbox{ for } \fg=\fso(2\ell+1)\\ 
(e_{1}\subset \dots\subset e_{\ell-1}) &\mbox{ for } \fg=\fso(2\ell)\end{array}\right\rbrace_{.}
\end{equation}

\end{nota}

\subsection{$K$-orbits on $\B_{n}$}\label{ss:Korbits}
We recall that $\fk:=\fg_{n-1}=\fgl(n-1)$ or $\fso(n-1)$ can be embedded 
in $\fg$ as a symmetric subalgebra (up to centre in $\fgl(n)$ case).   
For more details, see Section 2.3 of \cite{CE21I}.  The 
corresponding algebraic group $K=GL(n-1)$ or $SO(n-1)$ acts on $\B_{n}$ via the adjoint action with 
finitely many orbits.  These are classified in Propositions 2.8-2.10 of \emph{loc. cit.}, and we
follow the notation used there.  For our work here, it will be useful to have 
explicit representatives of $K$-orbits in terms of flags.  
For $\fg=\fgl(n)$, these representatives are given in Equations (2.6) and (2.9) of \emph{loc. cit.}.  
We restate these here for the convenience of the reader.  There are $n$ closed $K$-orbits on $\B_{n}$ which 
we denote by $Q_{i}$ for $i=1,\dots, n$.  The orbit $Q_{i}$ contains the flag 
\begin{equation}\label{eq:typeAflagclosed}
\mathcal{F}_{i}:=(e_{1}\subset\dots \subset e_{i-1}\subset\underbrace{e_{n}}_{i}\subset e_{i}\subset \dots \subset e_{n-1}). 
\end{equation}
There are ${n\choose 2}$ non-closed $K$-orbits on $\B_{n}$ which we denote by $Q_{i,j}$ for $1\leq i<j\leq n$.
The orbit $Q_{i,j}$ contains the flag
\begin{equation}\label{eq:typeAflag}
\mathcal{F}_{i,j}:=(e_{1}\subset\dots\subset e_{i-1} \subset \underbrace{e_{j-1}+e_{n}}_{i}\subset e_{i}\subset\dots\subset e_{j-2}\subset\underbrace{e_{n}}_{j}\subset e_{j}\subset\dots\subset e_{n-1}). 
\end{equation}
For ease of notation, we sometimes denote the closed $K$-orbit $Q_{i}$ by $Q_{i,i}$, so that 
the $K$-orbits on $\B_{n}$ are denoted by $Q_{i,j}$ with $1\leq i\leq j\leq n$. 

For $\fg=\fso(2\ell+1)$, the closed $K$-orbits are $Q_{+}$ and $Q_{-}$ (see Part (2) of Proposition 2.9 of \emph{loc. cit.}) and are represented by the maximal isotropic flags
$\mathcal{F}_{+}$, and
$\mathcal{F}_{-}:=(e_{1}\subset \dots\subset e_{\ell-1}\subset e_{-\ell})$ respectively.  
We denote the non-closed $K$-orbits  by $Q_{i}$ with $i=0,\dots, \ell-1$ (see Part (3) of Proposition 2.9 of \emph{loc. cit.}) and can be represented by the maximal isotropic flags 
\begin{equation}\label{eq:typeBflag}
\mathcal{F}_{i}=(e_{1}\subset \dots \subset e_{i}\subset\underbrace{ e_{\ell}+\sqrt{2}e_{0}-e_{-\ell}}_{i+1-position}\subset e_{i+1}\subset \dots\subset e_{\ell-1}) \mbox{ for } i=0, \dots, \ell-1.
\end{equation}
This can be verified using Equation (2.11) and Remark 2.11 in \emph{loc. cit.}. 
For $\fg=\fso(2\ell)$, the unique closed $K$-orbit, $Q_{+}$ contains the 
partial isotropic flag $\mathcal{F}_{+}$ (see 
Part (2) of Proposition 2.10 of \emph{loc. cit.}).  The non-closed orbits are denoted by $Q_{i}$ with $i=1,\dots, \ell-1$ and 
can be represented by the partial isotropic flags
\begin{equation}\label{eq:typeDflag}
 \mathcal{F}_{i}=(e_{1}\subset\dots\subset e_{i-1}\subset\underbrace{ e_{\ell}}_{i-position}\subset e_{i}\subset \dots\subset e_{\ell-2}) \mbox{ for } i=1,\dots, \ell-1.
\end{equation}
This can be verified using Equation (2.13) of \emph{loc. cit.}.  

We also recall the notion of a length of a $K$-orbit.
\begin{dfn}\label{d:length}
Let $Q_{K}$ be a $K$-orbit on $\B_{n}$. We define the \emph{length} of $Q_{K}$,
$\ell(Q_{K})$, by 
\begin{equation}\label{eq:Klength}
\ell(Q_{K}):=\dim Q_{K}-\dim \B_{\fk},
\end{equation}
where $\B_{\fk}$ denotes the flag variety of $\fk$.
\end{dfn}
\begin{rem}\label{r:Korbitlength}
For $\fg=\fgl(n)$, $\ell(Q_{i,j})=j-i$.  For $\fg=\fso(n)$, $\ell(Q_{\pm})=0$ and 
$\ell(Q_{i})=\ell-i$.  These assertions can be verified using Part (3) of Propositions 2.8-2.10 of 
\emph{loc. cit.} and the well-known fact that if $Q_{K}$ is closed then $Q_{K}\cong \B_{\fk}$ ( \cite{CEexp} for example).  
\end{rem}


\begin{nota}\label{note:orbitnote}
Let $H$ be an algebraic group and $X$ an $H$-variety.  We denote the 
set of $H$-orbits on $X$ by $H\backslash X$.
\end{nota}

In this paper, we study $\Borbitspace$, where $B_{n-1}:=K\cap B_{n}$ and $B_{n}\subset G$ is the standard 
upper triangular Borel subgroup of $G$ which stabilizes the flag $\F_{+}$ in (\ref{eq:upper}).  It follows from Section 2.2 of \cite{CE21I} that $B_{n-1}\subset K$ is the standard Borel subgroup of upper triangular matrices in $K$.

\section{Combinatorics of $B_{n-1}$-orbits on $\B_{n}$}\label{s:PILS}

In this section, we introduce sets $PIL(n)$ and $SPIL(n)$ which count the number of $B_{n-1}$-orbits on $\B_n$ in the general linear and orthogonal cases.  
As a consequence, we obtain formulas counting the number of $B_{n-1}$-orbits on
$\B_n$ and study the corresponding exponential generating functions.  It will be convenient to declare $B_{n}$ to be the trivial group 
for $n<0$.  In particular, $|B_{-1}\backslash \B_{0}|=1$.


\subsection{$PIL(n)$ and the $GL(n)$ case}\label{ss:gl}
 We first consider the case where $\fg=\fgl(n)$ and $K=GL(n-1)$.

\begin{dfn}\label{dfn:lists}
\begin{enumerate}
\item We define a list on a subset $\{a_{1},\dots, a_{k}\}$ of the non-negative integers to be a $k$-tuple $\sigma=(a_{i_{1}},\dots, a_{i_{k}}) \in \mathbb{Z}_{\geq 0}^k$.  We regard a list as an ordering on the corresponding subset.  For the list $\sigma$ above, we say
$\supp(\sigma)=\{ a_{i_1}, \dots, a_{i_k} \}$,  the length $\ell(\sigma)=k$,
and $\sigma(j)=a_{i_j}$ for $1 \le j \le k$.
\item  Let $A\subset \mathbb{Z}_{\geq 0}$ be any finite subset.  We define a PIL (partition into lists) of the set $A$ to be a partition of $A$ into lists.  
We shall denote a PIL consisting of $k$ lists $\sigma_{1},\dots, \sigma_{k}$ by $\Sigma=\{\sigma_{1},\, \dots, \, \sigma_{k}\}$.  
\item For a finite subset $A\subset\mathbb{Z}_{\geq 0}$, we denote the set consisting of all PIL's of $A$ by $PIL(A)$.  
In the case where $A=\{1,\dots, n\}$, we denote $PIL(A)$ by $PIL(n)$.  
We also denote the one-element set $PIL(\{0\})$ by $PIL(0)$.
\end{enumerate}
\end{dfn} 


\begin{exam}\label{ex:lists}
If $n=2$, then $|PIL(2)|=3$.  The three elements of $PIL(2)$ are $\{(1,2)\},\, \{(2,1)\}, \,$ and $\{(1), (2)\}$. 
If $n=3$, then $|PIL(3)|=13$.  There are six PIL's of the form $\{(i_1, i_2, i_3 )\; \}$, 
six of the form $\{(i_{1} i_{2}), (i_{3})\}$, and lastly, the PIL $\{(1), \, (2),\, (3)\}$.
\end{exam}

\begin{thm}\label{thm:counttypeA} 
The number of $B_{n-1}$-orbits on the flag variety $\B_n$ of $\fgl(n)$ is $|PIL(n)|$.
\end{thm}
\begin{proof} For the proof, we adopt the convention that if $\Sigma = \{ \sigma_1, \dots, \sigma_k \} \in PIL(n)$ then $n \in \supp(\sigma_1)$, which is no loss of generality.  Recall from Section \ref{ss:Korbits} that the $K$-orbit decomposition may be written as $\B_n = \sqcup_{1 \le i \le j \le n} Q_{i,j}$.  Since $B_{n-1} \subset K$, it follows that $B_{n-1}\backslash \B_n = \sqcup_{1 \le i \le j \le n} B_{n-1} \backslash Q_{i,j}$.  The subset of $PIL(n)$ corresponding to the $B_{n-1}$-orbits in $Q_{i,j}$ is 
\begin{equation}\label{e.qij}
PIL(n)_{i,j}:= \{ \Sigma = \{ \sigma_1, \dots, \sigma_k \} \in PIL(n) : \sigma_1(i)=n \mbox{ and } \ell(\sigma_1)=n-(j-i)\}.
\end{equation}
The reader can check easily that $PIL(n)=\sqcup_{1\le i\le j \le n} PIL(n)_{i,j}.$
Hence, to prove the theorem, it suffices to verify that $|B_{n-1}\backslash Q_{i,j} |=|PIL(n)_{i,j}|.$  By Theorem 3.1  of \cite{CE21I}, for each $K$-orbit $Q_{i,j}$ in $\B_n$, there is a parabolic subgroup $R_{i,j}$ of $G$ such that
$K\cap R_{i,j}$ is a parabolic subgroup of $K$ with Levi factor
 $GL(j-i) \times (GL(1))^{n-(j-i)-1}$.  Further, by Theorem 3.5 and Notation 3.6 as well as Theorem 3.8 of \cite{CE21I}, the $B_{n-1}$-orbits on $Q_{i,j}$ are classified by
$\orbitbundle$ as $Q_{\fr}$ varies over $B_{n-1}$-orbits on $K/K\cap R_{i,j}$
and $Q_{\fl}$ varies over $B_{j-i-1}$-orbits in $\B_{j-i}$.  Thus, 
\begin{equation}\label{eq:BinKcount}
|B_{n-1} \backslash Q_{i,j} | = |B_{n-1}\backslash K /K\cap R_{i,j} | \cdot |B_{j-i-1}\backslash \B_{j-i} |.
\end{equation}

To compute $|PIL(n)_{i,j}|$, we first consider the set of lists $$\mathfrak{S}_{i,j}:= \{ \mbox { lists } \sigma_1 : \ell(\sigma_1)=n-(j-i),  \sigma_1(i)=n, \mbox { and }\supp(\sigma_1) \subset \{ 1, \dots, n \} \}.$$  
We note that $|\mathfrak{S}_{i,j}|=|B_{n-1}\backslash K/K\cap R_{i,j}|.$  Indeed, the Bruhat decomposition implies that
$|B_{n-1}\backslash K/K\cap R_{i,j}|=|W_K|/|W_{K\cap R_{i,j}}|,$ while
 $|\mathfrak{S}_{i,j}|=(n-1)!/(j-i)!$ since $\mathfrak{S}_{i,j}$ parametrizes the numbers of ways to choose an ordered subset of cardinality $n-(j-i)-1$ from a set with $n-1$ elements.  Since $W_K$ is the symmetric group on $n-1$ letters and $W_{K\cap R_{i,j}}$ is the symmetric group on $j-i$ letters, the assertion follows.  
 From the definition of $PIL(n)_{i,j}$ in (\ref{e.qij}), we have a surjective map
 $$
\pi: PIL(n)_{i,j}\to \mathfrak{S}_{i,j}\mbox{ given by } \{\sigma_{1},\sigma_{2},\dots, \sigma_{n}\}\to \sigma_{1}.
 $$
 For any $\sigma_{1}\in\mathfrak{S}_{i,j}$, the fibre $\pi^{-1}(\sigma_{1})$ is $PIL(\{1,\dots, n\}\setminus \mbox{supp}(\sigma_{1}))$.  Since 
 $\ell(\sigma_{1})=n-(j-i)$, the set $\{1,\dots, n\}\setminus \mbox{supp}(\sigma_{1})$ has cardinality $j-i$.  Thus,
 $|\pi^{-1}(\sigma_{1})|=|PIL(j-i)|$.  By induction, we may assume that 
 $|PIL(j-i)|=|B_{j-i-1}\backslash \B_{j-i}|$.  Thus,
 $$|PIL(n)_{i,j}|=|\mathfrak{S}_{i,j}||PIL(j-i)|= |B_{n-1}\backslash K /K\cap R_{i,j} | \cdot |B_{j-i-1}\backslash \B_{j-i} |,$$ 
 and the result now follows from Equation (\ref{eq:BinKcount}).

\end{proof}

\subsection{Orthogonal Case}\label{ss:orthogonal}

To parameterize the $B_{n-1}$-orbits on $\B_{n}$ in the orthogonal case, we need to introduce the concept of a signed PIL which we 
abbreviate by the acronym SPIL.  We begin with the definition of a signed list.
\begin{dfn}\label{d:signedlistsB}
 Let $A=\{a_{1},\dots, a_{k}\}\subset \mathbb{Z}_{\geq 0}$ be a finite set.  We define a signed list of the set 
$A$ to be an ordered subset $(\lambda_{1},\dots, \lambda_{r})$ of the set $\{\pm a_{1}, \dots, \pm a_{k}\}$ with the property that 
$\lambda_{m}\neq \pm \lambda_{n}$ for $m\neq n$.  
\end{dfn}


\noindent For any signed list $\sigma=(\lambda_{1},\dots, \lambda_{r})$, let $|\sigma|=(|\lambda_{1}|, \dots, |\lambda_{r}|).$
\begin{dfn}\label{d:SPILS}
For any set $A\subset\mathbb{Z}_{\geq 0}$, we define a \emph{signed} $PIL$ (abbreviated SPIL) of $A$ to be a collection
$\Sigma$ of signed lists of $A$ with $\Sigma=\{\sigma_{1},\dots, \sigma_{k}\}$ 
such that $|\Sigma|:=\{|\sigma_{1}|,\dots, |\sigma_{k}|\}\in PIL(A)$ and $\Sigma$ satisfying the following two conditions:
\begin{equation}\label{eq:Sigmaconds}
\begin{split}
& (1)\mbox{ The last entry of any signed list } \sigma_{i} \mbox{ is non-negative.}\\
& (2) \mbox{ If }0\in\sigma_{i}, \mbox{ then }0  \mbox{ is the last entry of } \sigma_{i}.
\end{split}
\end{equation}
\end{dfn}

\begin{nota}
As for PILs, we denote the set of SPILs of the set $\{1,\dots,\ell\}$ by
$SPIL(\ell)$ and the set of SPILs of $\{0, 1,\dots, \ell\}$ by 
$SPIL(\ell\cup\{0\})$.  For convenience, we denote the one-element set $SPIL(\{0\})$ by $SPIL(0)$.  
\end{nota}

\begin{exam}\label{ex:spils} 
First note that $|SPIL(1)|=1$ and $|SPIL(1 \cup \{ 0 \})|=3.$  Indeed,
$SPIL(1 \cup \{ 0 \})$ consists of $\{(\pm 1, 0)\}$ and $\{ (1), (0) \}.$
While we saw above that $|PIL(2)|=3,$
  there are 
five elements in SPIL(2).  They are $\{(\pm 1, 2)\},\, \{(\pm 2, 1)\},\, \mbox{and }\{(1), (2)\}$.  
There are seventeen elements of $SPIL(2\cup\{0\})$.  Indeed, for $a,\, b\in \{1,\, 2\}$ with $a\neq b$, we have
eight SPIL's of the form $\{(\pm a, \pm b, 0)\}$, four of the form $\{(\pm a, 0), (b)\}$, 
and five of the form $\{(0), \Sigma\}$ with $\Sigma\in SPIL(2)$.  One can also 
compute that $|SPIL(3)|=37$, $|SPIL(4)|=361$, $|SPIL(3\cup\{0\})|=139$, and
$|SPIL(4\cup\{0\})|=1473.$  
\end{exam}

To prove the main result of this section, we first require a technical lemma.   
Let $A=\{a_{1},\dots, a_{k}\}\subset \mathbb{N}$.  For any $x\in A$, we define a subset $SPIL_{x}(A)$ of $SPIL(A)$ by 
$$
SPIL_{x}(A):=\{\Sigma\in SPIL(A)|\, \pm x \mbox{ occurs at the beginning of a signed list in $\Sigma$}\}.
$$
\begin{lem}\label{l:shift}
There is a bijection between the sets
$$
SPIL_{x}(A)\leftrightarrow SPIL(A\setminus\{x\} \cup \{0\}).
$$
\end{lem}
\begin{proof}
We first define a map $\Lambda:SPIL(A\setminus\{x\} \cup \{0\})\to SPIL_{x}(A)$ by 
$$
\Lambda(\{(\lambda_{1},\dots ,\lambda_{j-1}, \lambda_{j},0), \sigma_{2},\dots, \sigma_{k}\})=\{(\mbox{sgn}(\lambda_{j})x, \lambda_{1},\dots \lambda_{j-1}, |\lambda_{j}|), \sigma_{2},\dots,  \sigma_{k}\}.
$$
To define the map $\Gamma$ in the other direction, consider $\Sigma\in SPIL_{x}(A)$ 
with\\ $\Sigma=\{(\pm x,\mu_{1},\dots, \mu_{r}), \sigma_{2},\dots, \sigma_{k}\}$.   We then define 
$$
\Gamma(\Sigma):=\{(\mu_{1},\dots, \mu_{r-1}, \pm \mu_{r}, 0), \, \sigma_{2},\dots, \sigma_{k} \}. 
$$
We leave it the reader to check that $\Lambda$ and $\Gamma$ are inverse mappings.

\end{proof}

We can now prove the main result of this section.
\begin{thm}\label{thm:orthocount}

\begin{enumerate}
\item The number of $B_{SO(2\ell-1)}$-orbits on $\B_{\fso(2\ell)}$ is $|SPIL(\ell)|$.  
\item The number of $B_{SO(2\ell)}$-orbits on $\B_{\fso(2\ell+1)}$ is $|SPIL(\ell\cup\{0\})|$.  
\end{enumerate}
\end{thm}
\begin{proof}
We prove by induction that the number of $B_{SO(n-1)}$-orbits on $\B_{\fso(n)}$ is given by $SPIL(\frac{n-1}{2} \cup \{ 0 \})$ when $n$ is odd and by $SPIL(\frac{n}{2})$ when $n$ is even.  The cases $n=1$ and $n=2$ are both trivial.

 The inductive argument depends on whether $n$ is even or odd.   We begin with the case when $n=2\ell + 1$ is odd.  In this case, the $K=SO(2\ell)$-orbits are parametrized in Section \ref{ss:Korbits} as $Q_0, \dots, Q_{\ell-1}, Q_+,$ and $Q_-$.  Abusing notation, we define $Q_{\ell}:=Q_{+}\cup Q_{-}$.  Since $B_{SO(2\ell)} \subset K$, it follows that the number of $B_{SO(2\ell)}$-orbits on
$\B_{\fso(2\ell + 1)}$ equals $\sum_{i=0}^{\ell} |B_{SO(2\ell)}\backslash Q_i|$.  First, we compute $|B_{SO(2\ell)}\backslash Q_{\ell}|$.  By Remark \ref{r:Korbitlength}, each orbit 
$Q_{\pm}$ is $K$-equivariantly isomorphic to the flag variety of $K$, so that $|B_{SO(2\ell)}\backslash Q_{\pm}|=|W_{K}|=2^{\ell-1} \ell !$, whence
\begin{equation}\label{eq:closedcount}
|B_{SO(2\ell)}\backslash Q_{\ell}|=2^{\ell} \ell !. 
\end{equation}
Now to compute $|B_{SO(2\ell)}\backslash Q_i|$ for $i\leq \ell-1$, note that in Equations (3.4) and (3.5) and Theorem 3.1 of \cite{CE21I}, we associate to each $Q_i$ a parabolic subgroup $R_i\subset G$ with Levi factor $L$.  The Lie algebra $\fl$ of $L$ is isomorphic to $\fso(2(\ell -i)+1) \oplus \fgl(1)^i$ and $\fl \cap \fk$ is the Lie algebra of a Levi factor of the parabolic subgroup $K\cap R_i$ of $K$, with $\fl \cap \fk \cong \fso(2(\ell -i)) \oplus \fgl(1)^i.$   Then using Theorem 3.5, Notation 3.6, and Theorem 3.8 of {\it loc. cit.}, $B_{SO(2\ell)}$-orbits in $Q_i$ are parametrized by pairs $(Q_{\fr}, Q_{\fl})$, where $Q_{\fr}$ runs through $B_{SO(2\ell)}$-orbits in $K/K\cap R_i$ and $Q_{\fl}$ runs through $B_{SO(2(\ell -i)-1)}$-orbits on $\B_{\fso(2(\ell-i))}.$   By induction using the case (1) where $n$ is even, the number of orbits $Q_{\fl}$ that arise is $|SPIL(\ell - i)|.$  The number of orbits $Q_{\fr}$ that arise in this parametrization is $$\frac{|W_{\fk}|}{|W_{\fk\cap \fl}|} = \frac{2^{\ell-1} \cdot {\ell}!}{2^{\ell - i - 1}\cdot (\ell - i)!} = \frac{2^i \cdot {\ell}! }{(\ell - i)!}.$$ 
Using these observations along with (\ref{eq:closedcount}), we have 
\begin{equation}\label{e.qiorbits}
|B_{SO(2\ell)}\backslash Q_i | =  
|SPIL(\ell - i)| \cdot \frac{2^i \cdot  {\ell}! }{(\ell - i)!} \mbox{ for } i=0,\dots, \ell.
\end{equation}

To show $|B_{SO(2\ell)} \backslash \B_{\fso(2\ell + 1)}|$ coincides with $|SPIL(\ell \cup \{ 0 \} )|$, we give an analogous partition of $SPIL(\ell \cup \{ 0 \}).$  For convenience, we assume that if $\Sigma = \{ \sigma_1, \dots, \sigma_k \} \in SPIL(\ell \cup \{ 0 \}),$ then $0 \in \sigma_1.$
For $i=0, \dots, \ell$, we let $$SPIL(\ell \cup \{ 0 \})^i = \{ \Sigma =  \{ \sigma_1, \dots, \sigma_{k} \} \in SPIL(\ell \cup \{ 0 \}) : \ell(\sigma_1)=i+1 \}.$$
We note that 
$$
{SPIL(\ell \cup \{ 0 \})}^{\ell} = \{ \sigma_1 = (b_1, \dots, b_{\ell}, 0) : \mbox{supp}(|\sigma_{1}|)=\{1,\dots, \ell\}  \}.
$$
\noindent It is clear that $SPIL(\ell \cup \{ 0 \}) =  \sqcup_{i=0}^{\ell} {SPIL(\ell \cup \{ 0 \})}^i$ .  Hence, to complete the case when $n=2\ell + 1$, it suffices to show that
$|{SPIL(\ell \cup \{ 0 \})}^{i}|=|B_{SO(2\ell)}\backslash Q_i|$ for $i=0, \dots, \ell $.
Let $SLI(i, \ell)$ denote the set of signed lists on an $i$-tuple of distinct integers chosen from $\{ 1, \dots, \ell \}.$
We have a surjective map
\begin{equation}\label{eq:mappi}
\pi: SPIL(\ell \cup \{ 0 \})^{i}\to SLI(i,\ell)\mbox{ given by } \pi:\{ \sigma_1, \dots, \sigma_k \}\to (b_1, \dots, b_i), \mbox{where } \sigma_1=(b_1, \dots, b_i, 0).
\end{equation}
For a signed list $(b_1, \dots, b_i)\in SLI(i,\ell)$, the fibre $\pi^{-1}((b_1, \dots, b_i))$ has cardinality $|SPIL(\ell-i)|$.  
Since $|SLI(i, \ell)| = \frac{2^i\cdot \ell !}{(\ell -i)!}$, it follows that 
\begin{equation}\label{eq:countSPILi}
|{SPIL(\ell \cup \{ 0 \})}^{i}|= \frac{2^i \cdot \ell !}{(\ell -i)!} \cdot |SPIL(\ell - i)|,
\end{equation} and this coincides with
$|B_{SO(2\ell)}\backslash Q_i |$ by Equation \eqref{e.qiorbits}.  

It remains to establish the inductive step for the even case $n=2\ell$, and this follows the same outline as in the odd case, although some features of the calculation are different.  On the orbit side, the $K=SO(2\ell-1)$-orbits on 
$\B_{\fso(2\ell)}$ are parametrized in Section \ref{ss:Korbits} as $Q_i$, $i=1, \dots, \ell-1$, and the closed orbit $Q_+$.  For convenience, denote $Q_+$ by $Q_{\ell}$.  On the combinatorial side, we adopt the convention that for an element $\Sigma = \{ \sigma_1, \dots, \sigma_k \}$ in $SPIL(\ell)$, the elements $\pm \ell$ occur in $\sigma_1$.  For $i=1, \dots, \ell$, let ${\SPILell}^i := \{ \Sigma = \{ \sigma_1, \dots, \sigma_k \} : \sigma_1(i)=\pm \ell \}.$  To establish the even case, it suffices to show that $|B_{SO(2\ell - 1)}\backslash Q_i| = |{\SPILell}^i|$ for $i=1, \dots, \ell.$  On the geometric side, in Equation (3.7)  of \cite{CE21I}, we associate to the orbit $Q_i$ a parabolic subgroup $R_i\subset G$ with Levi factor $L$ whose Lie algebra $\fl \cong \fso(2(l-i)+2) \oplus \fgl(1)^{i-1}$
such that $K\cap R_i$ is a parabolic subgroup of $K$ with Levi factor $K\cap L$ with $\mbox{Lie}(K\cap L)=\fk \cap \fl \cong \fso(2(l-i)+1) \oplus \fgl(1)^{i-1}$.  Further, by Theorem 3.5, Notation 3.6,  and Theorem 3.8 of {\it loc.cit.}, the $B_{SO(2\ell-1)}$-orbits on $Q_{i}$ are parametrized by pairs $(Q_{\fr}, Q_{\fl})$ where $Q_{\fr}$ runs through $B_{SO(2\ell - 1)}$-orbits on $K/K\cap R_i$ and $Q_{\fl}$ runs through
 $B_{SO(2(\ell-i))}$-orbits on $\B_{\fso(2(\ell - i)+1)}.$  By induction, the latter number is equal to $|SPIL(\ell - i \cup \{ 0 \})|.$   The number of $B_{SO(2\ell - 1)}$-orbits on $K/K\cap R_i$ is 
$\frac{|W_K|}{|W_{K\cap L}|}=\frac{2^{i - 1}\cdot (\ell -1)!}{(\ell - i)!},$ as follows easily by the above description of $\fk \cap \fl$.   Hence,
\begin{equation}\label{e.Bevencaseorbits} 
|B_{SO(2\ell - 1)}\backslash Q_i|=\frac{2^{i - 1}\cdot (\ell -1)!}{(\ell - i)!}\cdot |SPIL(\ell - i \cup \{ 0 \})|.
\end{equation}
To count the elements on the combinatorial side, for $\Sigma = \{ \sigma_1, \dots, \sigma_k \} \in {\SPILell}^i,$ let $\sigma_1 =(b_1, \dots, b_{i-1}, \pm \ell, b_i, \dots, b_{\ell(\sigma_1)})$.   We consider the surjective map
$$
\pi: SPIL(\ell)^{i}\to SLI(i-1,\ell-1) \mbox{ given by } \pi(\Sigma)=(b_{1},\dots, b_{i-1}).
$$
As we observed above in the case $n=2\ell+1$, $|SLI(i-1,\ell-1)|=\frac{2^{i-1} (\ell-1) !}{(\ell-i) !}.$  
Further, the fibre $\pi^{-1}(\pi(\Sigma))$ is in bijection with the set $SPIL_{\ell}(\{1,\dots, \ell\}\setminus \{b_{1},\dots, b_{i-1}\})$.  
It now follows from Lemma \ref{l:shift} that 
\begin{equation} \label{e:evencombinatorics}
|{\SPILell}^i|=\frac{2^{i - 1}\cdot (\ell -1)!}{(\ell - i)!} \cdot 
|SPIL(\ell - i \cup \{ 0 \}) |.
\end{equation}
The even case now follows by Equations \eqref{e.Bevencaseorbits} and \eqref{e:evencombinatorics}, and this completes the proof.

\end{proof}

\subsection{Formulas and Exponential Generating Functions for $B_{n-1}$-orbits} \label{ss:formulae}

To find formulas for $|B_{n-1}\backslash \B_{n}|$ in the different cases 
using Theorems \ref{thm:counttypeA} and \ref{thm:orthocount}, we introduce the 
\emph{Lah numbers} which play an important role in enumerative combinatorics and applied mathematics (see \cite{daboullah} and other references). 
\begin{dfn}\label{d:Lahnumbers}
Given $n\in \mathbb{N}$ and $1\leq k\leq n$, the $(n,k)$ Lah number, $L(n,k)$, is defined by the formula
\begin{equation}\label{eq:Lahnumber}
L(n,k)=\frac{n !}{k!}{n-1\choose k-1}.
\end{equation}
It is well known that the number $L(n,k)$ is the number of ways to partition a set of 
$n$ elements into $k$ lists, so that $|PIL(n)|=\sum_{k=1}^n L(n,k)$.
\end{dfn}
\noindent The following result follows immediately from Definitions \ref{dfn:lists} and \ref{d:Lahnumbers} and 
Theorem \ref{thm:counttypeA}. 
\begin{prop}\label{p:formulaA}
Let $\fg=\fgl(n)$.  
A formula for the sequence $\{|B_{n-1}\backslash \B_{n}|\}_{n\in \mathbb{N}}$ is
\begin{equation}\label{eq:typeAformula}
|B_{n-1}\backslash \B_{n}|= \displaystyle\sum_{k=1}^{n} L(n,k)=\displaystyle\sum_{k=1}^{n}\frac{n!}{k!}{n-1\choose k-1}.
\end{equation}
\end{prop}

Recall that if $\{ f(n) \}_{n\in \mathbb{Z}_{\ge 0}}$ is a sequence of complex numbers, the corresponding exponential generating function $E_f(x)$ is the formal power series $\sum \frac{f(n) }{n!} x^n.$  
For functions $f, g:{\mathbb{Z}}_{\ge 0} \to \C$, define $s=s(f,g, \cdot):{\mathbb{Z}}_{\ge 0} \to \C$ by $s(n)=\sum_{k=0}^{n}{n\choose k} f(k)g(n-k)$. Then the multiplication principle for exponential generating functions asserts that $E_s(x)=E_f(x)\cdot E_g(x)$ (Proposition 5.1.3 and its proof in \cite{Stanley99}).
Let $P_k(n)$ denote the set of partitions $\lambda_1 + \dots + \lambda_k = n$ of $n$ into $k$ parts. Now for functions $f, g:{\mathbb{Z}}_{\ge 0} \to \C$, assume $g(0)=1$ and $f(0)=0$, and define a new function $h=h(f,g, \circ):{\mathbb{Z}}_{\ge 0} \to \C$ by $h(0)=1$ and $h(n)=\sum_{k=1}^n \sum_{\lambda \in P_k(n)} g(k)\cdot \prod_{i=1}^k f(\lambda_i).$  Then the composition principle of exponential generating functions asserts that $E_h(x)=E_g(E_f(x))$ (see Theorem 5.1.4 in \cite{Stanley99}).

We can now obtain the exponential generating function for the sequence in (\ref{eq:typeAformula}).  Let $f(n) = |PIL(n)|$ for $n\in\mathbb{Z}_{\geq 0}$.  
\begin{prop}\label{p:egf}
The exponential generating function for the sequence 
$\{|B_{n-1}\backslash \B_{n}|\}_{n\in{\mathbb{Z}}_{\ge 0}}$
 is
 $e^{\frac{x}{1-x}}$.  
\end{prop}
\begin{proof}
By Theorem \ref{thm:counttypeA}, the sequence $\{|B_{n-1}\backslash \B_{n}|\}_{n\in\mathbb{Z}_{\geq 0}}$ coincides 
with the sequence $\{|PIL(n)|\}_{n\in \mathbb{Z}_{\geq 0}}$.   The reader can easily check that $|PIL(n)|=\sum_{k=1}^n \sum_{\lambda \in P_k(n)} \prod_{i=1}^k \lambda_i !$ for $n\geq 1$.  Thus, if we take $f(n)=n!$ for $n > 0$ and $f(0)=0$ and $g(n)=1$ for all $n$, then $h(f,g, \circ)(n)=|PIL(n)|$ by an easy argument.
Since $E_f(x)= \frac{x}{1-x}$ and $E_g(x)= e^x$, the assertion follows by the composition principle.
\end{proof}

\begin{rem}\label{r.anrecursion}
It is well-known that the sequence $f(n)=|PIL(n)|=\sum_{k=1}^{n} L(n,k)$ satisfies the two-term 
recursion relation 
$$
f(n+1)=(2n+1)f(n)-(n^2-n)f(n-1).  
$$
To verify this relation, let $f(x)=e^{\frac{x}{1-x}} = \sum \frac{f(n)}{n!}x^n$. By calculus, $(x-1)^2 f^{\prime}(x) = f(x)$ and if we compare the $x^n$ coefficients in the power series representing each side, we obtain the recursion relation.
Theorem \ref{thm:counttypeA} then implies that the sequence $\{|B_{n-1}\backslash \B_n|\}_{n\in \mathbb{N}}$ satisfies this recursion as well.  
This is the assertion made without proof in the footnote on page 18 of \cite{Hashi}.
\end{rem}

We can also use Lah numbers to develop formulas for the number of $B_{n-1}$-orbits on $\B_n = \B_{\fso(n)}$ 
in the orthogonal cases.  We begin with the type $D$ case, $n=2\ell$.  Let 
$SPIL(\ell,k) = \{ \Sigma = \{ \sigma_1, \dots, \sigma_k \} \in SPIL(\ell) \}.$
 Note that $|SPIL(\ell,k)|=2^{\ell-k}L(\ell,k)$ using Property (1) from Equation \eqref{eq:Sigmaconds} in the definition of $SPIL(\ell).$

\begin{prop}\label{p:formulaD}
For  $\ell \in\mathbb{N}$, 
\begin{equation}\label{eq:formulaD}
|B_{SO(2\ell-1)}\backslash \B_{\fso(2\ell)}|=\displaystyle\sum_{k=1}^{\ell} 2^{\ell-k} L(\ell,k)=\displaystyle\sum_{k=1}^{\ell}  2^{\ell-k} \frac{\ell !}{k !} {\ell-1 \choose k-1}.
\end{equation}
\end{prop}
\begin{proof}   
Equation (\ref{eq:formulaD}) follows from Part (1) of Theorem \ref{thm:orthocount}, and the above formula for $|SPIL(\ell, k)|.$
\end{proof}
As in the type $A$ case, we can also obtain a generating function for the sequence $\{|B_{SO(2\ell-1)}\backslash \B_{\fso(2\ell)}|\}_{\ell\in\mathbb{Z}_{\geq 0}}$.
 
\begin{prop}\label{p:typeDgenfun}
The exponential generating function for the sequence 
$$\{|B_{SO(2\ell-1)}\backslash \B_{\fso(2\ell)}|\}_{\ell\in\mathbb{Z}_{\geq 0}}\mbox{ is }e^{\frac{x}{1-2x}}.$$
\end{prop}
\begin{proof}
By Part (1) of Theorem \ref{thm:orthocount}, for each $\ell\in {\mathbb{Z}}_{\geq 0}$, $|SPIL(\ell)|=|B_{SO(2\ell - 1)}\backslash \B_{\fso(2\ell)}|.$
Note that  $|SPIL(\ell)| = \sum_{k=1}^\ell \sum_{\lambda \in P_k(\ell)} 2^{\lambda_i - 1}\cdot \lambda_i !,$ for $\ell \geq 1.$   
We let $f(\ell)=2^{\ell-1}\cdot \ell!$ and $f(0)=0$ and let $g(\ell)=1$ for all $\ell$, and note that if $h=h(f,g, \circ)$, then $h(\ell)=|SPIL(\ell)|.$    By the composition principle, we conclude that $E_h(x)=E_g(E_f(x)).$ Since $\frac{x}{1-2x} = \sum_{\ell \in {\mathbb N}}  2^{\ell-1}\cdot x^\ell,$ $E_f(x)=\frac{x}{1-2x}$ and the proposition follows.  
\end{proof}


The type $B$ case is slightly more involved than the other cases.  

\begin{prop}\label{p:formulaB}
For  $\ell \in\mathbb{N}$,  
\begin{equation}\label{eq:formulaB}
|B_{SO(2\ell)}\backslash \B_{\fso(2\ell + 1)}|=\displaystyle\sum_{k=1}^{\ell+1}\frac{\ell !}{(k-1) !} {\ell\choose k-1} 2^{\ell-k+1}.
\end{equation}
\end{prop}
\begin{proof}
By Part (2) of Theorem \ref{thm:orthocount}, $|B_{SO(2\ell)}\backslash \B_{\fso(2\ell + 1)}|=|SPIL(\ell\cup\{0\})|$. 
 Let 
$$
SPIL(\ell\cup\{0\},k):=\{\Sigma\in SPIL(\ell\cup\{0\}):\, \Sigma=\{\sigma_{1},\dots, \sigma_{k}\}\}.
$$
Of course, 
\begin{equation}\label{eq:SPILcount}
|SPIL(\ell\cup\{0\})|=\sum_{k=1}^{\ell+1} |SPILS(\ell\cup\{0\},k)|.
\end{equation}
To compute $|SPIL(\ell\cup\{0\},k)|$, we note that we can construct an element in $SPIL(\ell \cup \{ 0 \}, k)$ in steps as follows.   First choose a list
$(i_1, \dots, i_{\ell})$ on the set $\{ 1, \dots, \ell \}$ and identify it with the list $(i_1, \dots, i_{\ell}, 0)$ on the set $\{ 0, 1, \dots, \ell \}$ satisfying 
Condition (2) of Equation \eqref{eq:Sigmaconds}.  We next choose $k-1$ elements $\{ j_1, \dots, j_{k-1} \}$ from $\{ 1, \dots, \ell \}$ and let
$\tsigma_r = (i_{j_{r-1} + 1}, \dots, i_{j_r})$ for $r=1, \dots, k-1$ with $j_{0}=0$.  Let $\tsigma_k = (i_{j_{k-1} + 1}, \dots, i_\ell, 0 ).$   There are $\ell ! \cdot {\ell \choose k-1 }$ ways to choose an ordered sequence of lists $\tilde{\Sigma} = (\tsigma_1, \dots, \tsigma_{k-1},\tsigma_{k})$ in this way.   
The sequence of lists $\tilde{\Sigma}$ defines an element of $PIL(\ell\cup\{0\})$ by simply forgetting the 
ordering of the lists $\tilde{\sigma_{i}}$.  The reader can verify that another sequence of $k$ ordered lists $\Sigma^{\prime}=(\sigma_{1}^{\prime},\dots, \sigma_{k-1}^{\prime}, \sigma_{k}^{\prime})$ 
constructed in this fashion defines the same element of $PIL(\ell\cup\{0\})$ if and only if $\sigma_{i}^{\prime}=\tilde{\sigma}_{w(i)}$, $i\leq k-1$ for some permutation 
$w$ of the set $\{1,\dots, k-1\}$ and $\sigma_{k}^{\prime}=\tilde{\sigma}_{k}$.  
Thus, we can construct exactly $\frac{\ell ! {\ell \choose k-1}}{(k-1) !}$ elements of $PIL(\ell\cup\{0\})$ in this way.  
From each of the $\frac{\ell ! {\ell \choose k-1}}{(k-1) !}$ such $PILs$, we can construct an element of $SPIL(\ell\cup \{0\}, k)$ by
assigning a sign to all of the non-final elements in each of the lists of the $PIL$ .  There 
are evidently $2^{\ell-k+1}$ ways to assign these signs to the $PIL$.  It follows that 
 \begin{equation}\label{eq:SPILkcount}
 |SPIL(\ell\cup\{0\},k)|= 2^{\ell-k+1}\frac{\ell !}{(k-1) !}{\ell\choose k-1}.
 \end{equation}
Equation \eqref{eq:formulaB} now follows from Equations \eqref{eq:SPILcount} and \eqref{eq:SPILkcount}. 
\end{proof}

We can also obtain an exponential generating function for the sequence $\{|B_{SO(2\ell)}\backslash \B_{\fso(2\ell+1)}|\}_{\ell\in\mathbb{Z}_{\geq 0}}.$

\begin{prop}\label{p:typeBgenfun}
The exponential generating function for the sequence
$$\{|B_{SO(2\ell)}\backslash \B_{\fso(2\ell + 1)}|\}_{\ell\in\mathbb{Z}_{\geq 0}}\mbox{ is } \frac{e^{\frac{x}{1-2x}}}{1-2x}.$$
\end{prop}
\begin{proof}
We observed in (\ref{eq:countSPILi}) that 
\begin{equation}\label{eq:long}
|SPIL(\ell\cup\{0\})|=\displaystyle\sum_{i=0}^{\ell} |SPIL(\ell\cup \{0\})^{i}|=\displaystyle\sum_{i=0}^{\ell} \frac{2^{i} \ell ! } {(\ell-i) !} |SPIL(\ell-i)|=\displaystyle\sum_{i=0}^{\ell}{\ell\choose i} 2^{i} i! |SPIL(\ell-i)|.
\end{equation}
Now let $f:\mathbb{Z}_{\geq 0}\to \mathbb{N}$ be the sequence $f(i)=2^{i} i !$, and let $g:\mathbb{Z}_{\geq 0}\to \mathbb{N}$ be the 
sequence $g(i)=|SPIL(i)|$.  Then $E_{f}(x)=\frac{1}{1-2x}$, and $E_{g}(x)=e^{\frac{x}{1-2x}}$ by Proposition \ref{p:typeDgenfun} and Part (2) of Theorem \ref{thm:orthocount}.
The result now follows from Equation \ref{eq:long} and the multiplication principle for 
exponential generating functions.

\end{proof}



\begin{rem}\label{r.bdrecursions}
For $n\in\mathbb{Z}_{\geq 0}$, let $d_n = |B_{SO(2n-1)}\backslash \B_{\fso(2n)}|$, and let $b_n = |B_{SO(2n)}\backslash \B_{\fso(2n+1)}|$.  Then the sequences $\{ d_n \}$ and $\{ b_n \}$ satisfy the recursion relations
\[
d_{n+1}=(1+4n)d_n - 4(n^2 - n)d_{n-1} \mbox{ and } b_{n+1}=(3+4n) b_n - 4n^2  b_{n-1}.
\]
To verify these formulas, let $f(x)=\sum_{n=0} \frac{d_n}{n!} x^n = e^{\frac{x}{1-2x}}$ and verify that $(2x-1)^2 f^{\prime}(x)=f(x).$   Comparison of the $x^n$ coefficient gives the $d_n$-relation.   For the $b_n$-relation, let $f(x)=\sum_{n=0} \frac{b_n}{n!} x^n = \frac{e^{\frac{x}{1-2x}}}{1-2x}$, and observe that $(2x-1)^2 f^{\prime}(x)=(3-4x) f(x)$.  Now compare $x^n$ coefficients.
\end{rem}

\begin{rem}\label{r:Lahtransform}
Given a sequence $\{a_{n}\}_{n\in\mathbb{Z}_{\geq 0}}$, we define its Lah transform to be 
the sequence $\{b_{n}\}_{n\in \mathbb{Z}_{\geq 0}}$ with $b_{n}:=\sum_{k=1}^{n} L(n,k) a_{k}$, and $b_{0}=a_{0}$ where
$L(n,k)$ is the $(n,k)$ Lah number defined in Equation (\ref{eq:Lahnumber}).  It follows immediately 
from Proposition \ref{p:formulaA} that the sequence $\{|B_{\fgl(n-1)}\backslash \B_{\fgl(n)}|\}_{n\in\mathbb{Z}_{\geq 0}}$ is the 
Lah transform of the constant sequence $\{1\}_{n\in\mathbb{Z}_{\geq 0}}$.  
We also claim that the sequence $\{|B_{\fso(2n-1)}\backslash \B_{\fso(2n)}|\}_{n\in\mathbb{Z}_{\geq 0}}$ is the Lah transform of the sequence 
$\{|B_{\fgl(n-1)}\backslash \B_{\fgl(n)}|\}_{n\in\mathbb{Z}_{\geq 0}}$.  
This follows from the following general fact concerning Lah transforms and exponential generating functions.  
If $f(n)=a_{n}$ is a sequence and $g(n)=b_{n}$ is a sequence, then 
$\{b_{n}\}_{n\in\mathbb{Z}_{\geq 0}}$ is the Lah transform of $\{a_{n}\}_{n\in\mathbb{Z}_{\geq 0}}$ if and only 
if $E_{g}(x)=E_{f}(x/1-x)$.   
This assertion can be demonstrated using the well-known recursion relation $L(n+1, k)=(n+k) L(n,k)+L(n,k-1)$ for 
Lah numbers to show that for any smooth function $h(x)$, 
$$
\frac{d^{n}}{(dx)^{n}} h\left( \frac{x}{1-x}\right)\Bigr |_{x=0}=\displaystyle\sum_{k=1}^{n} L(n,k) h^{(n)}(0),
$$
for $n\geq 1$.
\end{rem}

\section{Representatives for $B_{n-1}$-orbits}\label{s:std}

In this section, we give explicit representatives for $B_{n-1}$-orbits on the flag variety $\B_n.$   The constructions and proofs rely on the results on $PIL(n)$ and $SPIL(\ell)$ from the previous section.  The cases of $GL(n), SO(2\ell)$, and $SO(2\ell + 1)$ are discussed separately.

\subsection{Representatives in the case of $GL(n)$}\label{ss:glreps}

We first describe the form of our standard representatives.

\begin{dfn-nota}\label{d:std}

\begin{enumerate}
\item For a standard basis vector $e_{i}\in\C^{n}$ with $i\leq n-1$, we define $\he_{i}:=e_{i}+e_{n}$ and refer to $\he_{i}$ as a hat vector of index $i$.
\item We say that a flag 
\begin{equation}\label{eq:basicflag}
\mathcal{F}:=(v_{1}\subset \dots \subset v_{i}\subset\dots\subset v_{n})
\end{equation}
in the flag variety $\B_n$ for $GL(n)$
is in \emph{standard form} if $v_{i}=e_{j_{i}}$ or $v_{i}=\he_{j_{i}}$ for all $i=1,\dots, n$, and 
$\mathcal{F}$ satisfies the following three conditions:
\begin{enumerate}
\item $v_i = e_n$ for some $i.$
\item If $v_{i}=e_{n}$, then $v_{k}=e_{j_{k}}$ for all $j>i$.
 \item If $i<k$ with $v_{i}=\he_{j_{i}}$ and $v_{k}=\he_{j_{k}}$, then $j_{i}>j_{k}$. 
\end{enumerate}
\end{enumerate}
\end{dfn-nota}

\begin{rem}\label{r.glstandard}
If $\mathcal{F}=(v_{1}\subset \dots \subset v_{i}\subset\dots\subset v_{n})$ is a flag in standard form, then it follows easily  that if $v_i = e_{j_i}$ or $\he_{j_i}$, then $\{ j_1, \dots, j_n \} = \{ 1, \dots, n \}$ consists of $n$ distinct indices.
\end{rem}

\begin{rem}\label{r.schubert}  Let $B=B_n$ be the upper triangular Borel subgroup of $G=GL(n)$.   Then it follows from the Bruhat decomposition that every $B$-orbit on the Grassmannian $\Gr(k,n)$ of $k$-planes in $\C^n$ can be written as $B\cdot \lspan\{e_{j_1}, \dots, e_{j_k}\}$ for a unique subset $\{ j_1, \dots, j_k \}$ of $\{ 1, \dots, n \}.$   Furthermore, every $B$-orbit on a partial flag variety can be written uniquely as the $B$-orbit through a collection of nested subspaces, where each subspace is a span of the vectors of the form $e_{j_i}$.
\end{rem}

It is useful for our purposes to classify $B_{n-1}$-orbits on $\Gr(k,n)$.  Recall from Section 2.2 of \cite{CE21I} the automorphism on $GL(n)$, $\theta(g)=tgt^{-1}$, where $t=\mbox{diag}[1,\dots, 1,-1]$ so that $G_{n-1}=G^{\theta}$ (up to centre).  Consider the $G_{n-1}$-stable decomposition $\C^n = U_1 \oplus U_2$, where $U_1 = (\C^{n})^{t}$ and $U_2=(\C^{n})^{-t}$ are the $n-1$ and $1$ dimensional subspaces consisting of $t$ and $-t$ fixed spaces on $\C^n.$

\begin{lem} \label{l:AGrass} Let $Q=B_{n-1}\cdot W$ be the $B_{n-1}$-orbit through a subspace $W \subset \Gr(k,n)$.
\begin{enumerate}
\item If $W\subset U_1$, then $Q={\mathcal{O}}_{j_1, \dots, j_k} := B_{n-1}\cdot \lspan\{e_{j_1}, \dots, e_{j_k}\}$ for a unique subset $\{ j_1, \dots, j_k \}$ of $\{ 1, \dots, n-1 \}.$
\item If $\dim(W \cap U_1)=k-1$ and $U_2 \subset W$, then $Q={\mathcal{O}}_{j_1, \dots, j_{k-1}, n} := B_{n-1}\cdot \lspan\{e_{j_1}, \dots, e_{j_{k-1}}, e_n \}$ for a unique subset $\{ j_1, \dots, j_{k-1} \}$ of $\{ 1, \dots, n-1 \}.$
\item If $\dim(W \cap U_1)=k-1$ and $U_2 \not\subset W$, then
$Q= {\mathcal{O}}_{j_1, \dots, j_{k-1}, \hat{j}_k} := B_{n-1}\cdot \lspan\{e_{j_1}, \dots, e_{j_{k-1}}, \he_{j_k} \}$ for a unique subset
$\{ j_1, \dots, j_k \}$ of $\{ 1, \dots, n-1\}.$
\end{enumerate}
\end{lem}

\begin{proof}
By linear algebra, the integers $(\dim(W \cap U_1), \dim(W \cap U_2))$ are
either $(1) \ (k, 0)$, $ (2) \ (k-1, 1)$, or $(3) \ (k-1, 0)$ and these conditions are evidentally invariant under the $B_{n-1}$-action.   The classification of $B_{n-1}$-orbits on planes $W$ satisfying condition (1) is the same as the classification of $B_{n-1}$-orbits on $\Gr(k,n-1)$, and follows from Remark \ref{r.schubert}.  The classification of $B_{n-1}$-orbits on planes $W$ satisfying condition (2) is the same as the classification of $B_{n-1}$-orbits on $\Gr(k-1,n-1)$, and again follows from Remark \ref{r.schubert}.  Lastly, the classification of $B_{n-1}$-orbits on planes $W$ satisfying condition (3) reduces easily to the classification of $B_{n-1}$-orbits of partial flags $(W\cap U_1 \subset p_{1}(W))$ in  $\Gr(k-1, n-1) \times \Gr(k,n-1)$, where $p_{1}:\C^n \to U_1$ is projection off the subspace $U_2.$    By Remark \ref{r.schubert}, we may assume $W \cap U_1$ is the span of vectors $e_{j_1}, \dots, e_{j_{k-1}}$ with each $j_i < n$ and $p_{1}(W)$ is given by adjoining another vector $e_{j_k}$ with $j_k < n.$  The assertion now follows easily.
\end{proof}

\begin{prop}\label{p:AflagsandPILS}
There is a bijection 
\begin{equation}\label{eq:AGamma}
\Gamma: PIL(n) \longleftrightarrow \{\mbox{Complete flags in standard form in $\C^n$}\}.
\end{equation}

\end{prop}
\begin{proof} 
 Let $\Sigma\in PIL(n)$ with $\Sigma=\{\sigma_{1},\, \sigma_{2},\,\dots, \, \sigma_{t}\},$ let $\ell(\sigma_j)=k_j$ and let $\sigma_{j}=(i_{j,1},\dots, i_{j, k_{j}}).$  
 The lists $\sigma_j$ in $\Sigma$ may be ordered uniquely so that $i_{1,k_{1}}>i_{2,k_{2}}>\dots>i_{t-1, k_{t-1}}$ and $n\in\sigma_{t}$.   If we order $\Sigma$ in this way, 
 and define $\Gamma(\Sigma)$ to be the flag
 \begin{equation}\label{eq:Gammadef}
 \begin{split}
 \Gamma(\Sigma)=&(e_{i_{1,1}}\subset e_{i_{1,2}}\subset e_{i_{1,k_{1}-1}}\subset\he_{i_{1,k_{1}}}\subset  e_{i_{2,1}}\subset\dots\subset \he_{i_{2,k_{2}}}\subset\dots\\
 &\dots\subset e_{i_{t-1,1}}\dots\subset \dots \subset \he_{i_{t-1,k_{t-1}}}\subset e_{i_{t,1}}\subset \dots\subset e_{i_{t,k_{t}}}),
 \end{split}
 \end{equation}
then the flag $\Gamma(\Sigma)$ is in standard form.   We now define a map $\Lambda:\{\mbox{Complete flags in standard form}\}\to PIL(n)$ inverse to $\Gamma.$   Let 
$\mathcal{F}$ be a flag in standard form with $\mathcal{F}=(v_{1}\subset \dots \subset v_{n})$.  Let $1\leq k_{1}<\dots< k_{r}< n$ be the subsequence of 
$\{1,\dots, n\}$ such that $\{v_{k_{1}},\dots, v_{k_{r}}\}$ consists of all the hat vectors in the given flag.   It follows that each $v_{k_{i}}=\he_{i_{k_{i}}}$ and $v_{j}=e_{i_{j}}$ for $j\notin\{k_{1},\dots, k_{r}\}$. 
Then we define
\begin{equation}\label{eq:Lambdadefn}
\Lambda(\mathcal{F}):=\{(i_{1},\dots ,i_{k_{1}-1}, i_{k_{1}}), (i_{k_{1}+1},\dots ,i_{k_{2}}),\dots, (i_{k_{1}+\dots+k_{r-1}+1},\dots  ,i_{k_{r}}),(i_{k_{1}+\dots+k_{r}+1}, \dots, i_{n})\}.  
\end{equation}
It is routine to check that $\Lambda$ and $\Gamma$ are mutual inverses.
\end{proof}

We need a preliminary result concerning the span of vectors in a flag in standard form.   For $\mathcal{F}=(v_1 \subset \dots \subset v_n)$, let $V_m = \lspan\{ v_1, \dots, v_m \}.$   If ${\pi}_m:\B_n \to \Gr(m,n)$ is the natural projection, then $\pi_m(\mathcal{F})=V_m.$  We say $j_i$ is the index of $v_i$ if $v_i = e_{j_i}$ or $\he_{j_i}$.

\begin{lem}\label{l:spanstandard}
Let $\mathcal{F}=(v_1 \subset \dots \subset v_n)$ be a flag in standard form.
\par\noindent (i) If $v_{i}=e_{j_{i}}$ with $j_{i}\leq n$ for all $i=1,\dots,m$, 
then $V_{m}=\mbox{span}\{e_{j_{1}},\dots, e_{j_{m}}\}$.
\par\noindent (ii) If  $v_i=e_n$ for some $i$, then for all $m \ge i$, the space $V_m = \lspan\{e_{r_1}, \dots, e_{r_{m-1}}, e_{n}\}$ where  $\{ r_1, \dots, r_{m-1}, n \}$ are the indices of $v_1, \dots, v_m.$
\par\noindent (iii) Let $k_1 >\dots > k_s$ be the indices of the hat vectors appearing in the first $m$ vectors in $\mathcal{F}$ and let $r_1, \dots, r_{m-1}\in \{1,\dots, n-1\}$ be the indices of $v_i$ for $1\le i \le m$ besides $k_s$.   Then $B_{n-1}\cdot V_m = {\mathcal{O}}_{r_1, \dots, r_{m-1}, \hat{k}_{s}}.$
\end{lem}

\begin{proof} The proof of (i) is trivial.  The proof of (ii) is based on the observation that $\lspan \{ \he_j , e_n\} = \lspan\{ e_j, e_n \}.$  To prove (iii), note that
by part (c) of Definition \ref{d:std}, $k_i > k_s$ for $i < s$.  Then the span of
$\he_{k_i}$ and $\he_{k_s}$ coincides with the space $\lspan \{e_{k_i}-e_{k_s}, \he_{k_s}\}$, and this is $B_{n-1}$-conjugate to $\lspan\{e_{k_i}, \he_{k_s}\}$ without altering any of the other vectors in $\mathcal{F}.$  Assertion (iii) follows by repeated application of this assertion.
\end{proof}

\begin{thm}\label{thm:std}
The map  
\begin{equation}\label{eq:flagassoc} 
\Psi:\mathcal{F} \to B_{n-1}\cdot \F
\end{equation}
 is a bijection between flags in standard form and $B_{n-1}$-orbits on $\B_n$.
\end{thm}

\begin{proof}
  By Theorem \ref{thm:counttypeA}, the set $PIL(n)$ has the same cardinality as
the set of $B_{n-1}$-orbits on $\B_n$, and by Proposition \ref{p:AflagsandPILS}, 
it follows that the number of flags in standard form coincides with the number of $B_{n-1}$-orbits on $\B_n.$  Thus, it suffices to show that the map $\Psi$ in (\ref{eq:flagassoc}) is injective.  
To this effect, let $Q=B_{n-1}\cdot \F$, and $Q^{\prime}=B_{n-1}\cdot \F^{\prime}$ where $\F=(v_{1}\subset v_{2}\subset \dots\subset v_{i}\subset \dots \subset v_n )$, 
and  $\mathcal{F}^{\prime}=(v_{1}^{\prime}\subset v_{2}^{\prime}\subset \dots\subset v_{i}^{\prime}\subset \dots v_n^{\prime})$ are flags in standard form.  
Suppose that $Q=Q^{\prime}$.   We prove that $v_{i}=v_{i}^{\prime}$ for all $i$.   The proof proceeds by induction on $i$ using the classification 
of $B_{n-1}$-orbits on $\Gr(i, n)$ given in Lemma \ref{l:AGrass}.   
Since $Q=Q^{\prime}$, we conclude that
 $B_{n-1}\cdot V_{i}=B_{n-1}\cdot V_{i}^{\prime}$ for all $i$.

To begin the induction,  consider the orbit $\mathcal{O}_{1}=B_{n-1}\cdot V_{1}=B_{n-1}\cdot V_{1}^{\prime}$.   Then we must be in one of the three cases of Lemma \ref{l:AGrass}.  In case (1), $v_1=e_{j_1}$ and $v_1^{\prime}=e_{k_1}$ with $j_1, k_1 < n$, and by Remark \ref{r.schubert}, $j_1=k_1$.  In case (2), $v_1=e_n=v_1^{\prime}$.  In case (3), $v_1 = \he_{j_1}$ and $v_1^{\prime}=\he_{k_1}$ with $j_1, k_1 < n$, and since the $B_{n-1}$-orbits through the lines $p_{1}(V_1)$ and $p_{1}(V_1^{\prime})$ coincide, we conclude by Remark \ref{r.schubert} that $j_1=k_1.$  Thus, in any case $v_1=v_1^{\prime}.$ For the inductive step, we may assume 
 that $v_{i}=v_{i}^{\prime}$ for $i=1,\dots, m-1$, so that $V_{m-1}=V_{m-1}^{\prime}$.
We prove that $v_m = v_m^{\prime}$ using our assumption that $B_{n-1}\cdot V_m = B_{n-1}\cdot V_m^{\prime}.$  To prove this, we argue based on the case of
$B_{n-1}\cdot V_{m-1}$ described in Lemma \ref{l:AGrass}.  In case (1), $V_{m-1}\subset U_1$ and $V_{m-1}=\lspan\{e_{j_1}, \dots, e_{j_{m-1}}\}$.  If $v_m = e_{j_m}$ with $1 \le j_m \le n$, then $V_m = \lspan\{e_{j_1}, \dots, e_{j_m}\}$ and $B_{n-1}\cdot V_m = {\mathcal{O}}_{j_1, \dots, j_m}$ or ${\mathcal{O}}_{j_1, \dots, j_{m-1}, n}.$  If $v_m = \he_{j_m}$, then $V_m = \lspan\{e_{j_1}, \dots, e_{j_{m-1}}, \he_{j_m}\}$ and 
$B_{n-1}\cdot V_m = {\mathcal{O}}_{j_1, \dots, j_{m-1}, \hat{j}_m}.$
 Thus, we can recover $v_m$ from $V_{m-1}$
 and $B_{n-1}\cdot V_m,$ whence $v_{m}=v_{m}^{\prime}$.  In case (2), $B_{n-1}\cdot V_{m-1}=\mathcal{O}_{j_{1},\dots, j_{m-2}, n}$.  By Lemma \ref{l:spanstandard}, it follows that $v_i = e_n$ for some $i < m$, and hence by the definition of a standard flag, $v_m = e_{j_m}$ for some $j_m < n.$    Another application of Lemma \ref{l:spanstandard} yields $V_m = \lspan\{e_{j_1}, \dots, e_{j_m}\}$ and $j_{m}$ is the unique index that appears in  ${\mathcal{O}}_{j_1, \dots, j_{m-2}, j_{m}, n} = B_{n-1}\cdot V_m$ that does not appear in $B_{n-1}\cdot V_{m-1}.$   In case (3), $B_{n-1}\cdot V_{m-1}=\mathcal{O}_{j_{1},\dots, j_{m-2}, \hat{k}_{s}}$, so by Lemma \ref{l:spanstandard}, the vector $\he_{k_s}$ must be $v_r$ for some $r \le m-1$ and $v_i \not= e_n$ for any $i \le m-1.$  Then if $v_m = e_{j_m}$ with $j_m < n$, then $B_{n-1}\cdot V_m=\mathcal{O}_{j_{1},\dots, j_{m-2}, j_{m}, \hat{k}_{s}}$,  and we can recover $j_m$ from $V_{m-1}$ and $B_{n-1}\cdot V_m.$   If $v_m = e_n$, then $V_m$ is the span of $\{ e_{j_i} : i \le m-2 \},$ $e_n$, and $e_{k_s}$ and we can recover $e_n$ from $V_{m-1}$ and $V_m.$  If $v_m = \he_{j_m}$, then by the definition of a flag in standard form, we know that $j_m < k_s$.  Using Lemma \ref{l:spanstandard}, we see easily that $B_{n-1}\cdot V_m={\mathcal{O}}_{j_1, \dots, j_{m-2}, k_s,\hat{j}_m,}.$  Hence, $v_m=\he_{j_m}$ is the hat index in the label of $B_{n-1}\cdot V_m$  which replaces $\hat{k}_s$ in $B_{n-1}\cdot V_{m-1}$ and thus is uniquely determined by $B_{n-1}\cdot V_m$ and $V_{m-1}.$  Thus in any case, $v_m$ is uniquely determined by $B_{n-1}\cdot V_m$ and $V_{m-1}$, and this establishes the theorem in the final case.
\end{proof} 

\begin{rem}\label{r:pilorbit}
By composing the bijections $\Gamma$ and $\Psi$ from Proposition \ref{p:AflagsandPILS} and Theorem \ref{thm:std}, we obtain a bijection $\Psi \circ \Gamma: PIL(n) \to B_{n-1}\backslash \B_n.$  This bijection improves on the numerical equality of Theorem \ref{thm:counttypeA}, and in particular, reflects the twisting aspect of the bundle structure of orbits in a better way.   It would be interesting to see whether this equivalence between elements of $PIL(n)$ and orbits is useful in understanding the geometry of orbit closures.   We note that similar remarks apply in the orthogonal cases discussed later in this section, but we will omit them.
\end{rem}

\subsection{Representatives in the case of $SO(2\ell)$}\label{ss:Dreps}

In this subsection, we give explicit representatives for $B_{n-1}$-orbits on $\B_n$ in the case where $G=SO(2\ell).$  We omit some arguments which are similar to the case of $GL(n)$.   Recall the description of $\B_{2\ell}$ in terms of isotropic flags from Section \ref{ss:Borels}.

\begin{dfn-nota}\label{d:stdD} 
\begin{enumerate}
\item For $i=1,\dots, \ell-1$, we define $\he_{i}:=e_{-i}+e_{-\ell}$, and $\te_{i}:=e_{-i}+e_{\ell}$, which we refer to as hat vectors and tilde vectors respectively in $\C^{2\ell}.$
  
\item We say that the isotropic flag
\begin{equation}\label{eq:basicflagD}
\mathcal{F}:=(v_{1}\subset \dots \subset v_{i}\subset\dots\subset v_{\ell-1})
\end{equation}
is in \emph{standard form} if $v_{i}$ is either a standard basis vector, a tilde vector, or a hat vector,
and $\mathcal{F}$ satisfies the following conditions:
\begin{enumerate}
\item If $v_{i}=e_{\pm\ell}$, then $v_{k}=e_{j_{k}}$ for $k>i$.  
\item If $\F$ contains a tilde vector, then $e_{\ell}$ must also occur in $\F$.
\item If $v_{i}=\he_{j_{i}}$ and $v_{k}=\he_{j_{k}}$ for $i<k$, then $j_{i}<j_{k}$, and the analogous condition holds for tilde vectors. 
\item The vectors $\he_{i}$ and $e_{-i}$ (respectively $\te_{i}$ and $e_{-i}$) do not both occur in $\F$.
\end{enumerate}
\end{enumerate}

We call such an isotropic flag a $SO(2\ell)$-flag in standard form.
If $v_i = e_{j_i}, \he_{j_i},$ or $\te_{j_i}$, then we say the index of $v_i$ is $|j_i|.$
\end{dfn-nota}

\begin{rem}\label{r.standardD}
It follows from definitions and the isotropic nature of the flag that the set of indices of $v_1, \dots, v_{\ell-1}$ consists of $\ell - 1$ distinct indices from $1$ to $\ell$.  Further, hat vectors and tilde vectors cannot occur in the same flag in standard form, and if some hat vector $\he_i$ (resp. tilde vector $\te_i$) occurs, then $e_{\ell}$ (resp. $e_{-\ell}$) cannot also occur.  
\end{rem}

\begin{rem}\label{r.outerD}
In Section 2.2 of \cite{CE21I}, we discuss an element $\sigma_{2\ell}$ of the orthogonal group $O(2\ell)$ which interchanges $e_{\ell}$ and $e_{-\ell}.$  The automorphism $\theta$ induced by conjugation by $\sigma_{2\ell}$ defines $SO(2\ell-1)$ in $SO(2\ell)$ and stabilizes $B_{2\ell-1}.$  If we considered orbits of the semidirect product of $B_{2\ell-1}$ and $\theta$, then we would not need tilde vectors.
\end{rem}

\begin{prop}\label{p:flagsandPILSD}
There is a bijection 
\begin{equation}\label{eq:GammaD}
\Gamma:\{SPIL(\ell)\}\longleftrightarrow \{ SO(2\ell)-\mbox{Flags in standard form}\}.
\end{equation}

\end{prop}
\begin{proof}
Let  $\Sigma=\{\sigma_{1},\, \sigma_{2},\,\dots, \, \sigma_{t}\},$ let $\ell(\sigma_j)=k_j$ and let $\sigma_{j}=(i_{j,1},\dots, i_{j, k_{j}}).$  We may choose the ordering of the $\sigma_j$ in a unique way so that
$i_{1,k_{1}}<\dots <i_{t-1,k_{t-1}}$ and $\pm \ell \in \sigma_{t}$.  We say $\Sigma$ with this ordering is in case 1 if either  $i_{t,k_{t}}=\ell$ or $-\ell\in\sigma_{t}$, and if so,  
then $\Gamma(\Sigma)$ is given by Equation \eqref{eq:Gammadef} with the last vector $e_{i_{t,k_{t}}}$ omitted.   
We say $\Sigma$ is in case 2 if $\ell\in \sigma_{t}$ and $\ell\neq i_{t,k_{t}}$, then we define $\Gamma(\Sigma)$  by (\ref{eq:Gammadef}), but with the hat vectors replaced by tilde vectors, and again omit the final vector $e_{i_{t,k_{t}}}$.  Note that in case 1, the vector $e_{\ell}$ does not occur in $\Gamma(\Sigma)$ and in case 2, the vector $e_{\ell}$ does occur in $\Gamma(\Sigma)$.  We note that in either case $\Gamma(\Sigma)$ is in standard form, as follows from the definition and Remark \ref{r.standardD}.  We define the inverse map $\Lambda$ as follows.   For a $SO(2\ell)$-flag $\F$ in standard form, then exactly one index in $\{ 1, \dots, \ell \}$ does not occur in $\F$ by Remark \ref{r.standardD}.   If the vector $e_{\ell}$ does not occur in $\F$, then we use the formula of Equation \eqref{eq:Lambdadefn}, except the last index $i_{\ell}$ is the index in $\{ 1, \dots, \ell \}$ not appearing among the indices in $\F$.   If the vector $e_{\ell}$ does occur in $\F$, then we use the same prescription as in the last case, but with tilde vectors replacing hat vectors.   It is routine to check that $\Lambda$ and $\Gamma$ are inverse bijections.
\end{proof}

As in the type A case, we need to describe $B_{n-1}$-orbits on the space 
$\OGr(i,n)$ consisting of isotropic $i$-planes on $\C^{n}$ for $1\le i \le n- 1.$  Let $B=B_{n}$ consist of the upper triangular matrices in $SO(n)$.

\begin{rem}\label{r:schubertD}
Let $n=2\ell$ or $2\ell + 1$.   By the Bruhat decomposition, the distinct $B$-orbits on $\OGr(i,n)$  are given by the
orbits through the planes  $\lspan \{ e_{j_1}, \dots, e_{j_i} \}$ where $\{j_{1},\dots, j_{i-1}, \, j_{i}\}\subset \{\pm 1,\dots, \pm \ell \}$, 
with $j_{k}\neq \pm j_{m}$.  
\end{rem}

In the case where $n=2\ell$, for $j=1, \dots, \ell$, we let $V_{< j}=\lspan\{e_{1},\dots, e_{j-1}\}$
and for $j=0, \dots, \ell-2$, let $V_{<-j}=\lspan\{ e_{-(j+1)},\dots, e_{-(l-1)} \}.$  Let $V_{+}=V_{<\ell}$ and $V_{-}=V_{<0}.$

\begin{rem}\label{r:bvectororbit}
Let $N_{2\ell-1}$ be the derived subgroup of $B_{2\ell-1}.$
\begin{enumerate}
\item For $j=1, \dots, \ell,$ the affine set  $e_j + V_{<j} = N_{2\ell-1}\cdot e_j$ and $e_{-\ell} + V_{+} = N_{2\ell-1}\cdot e_{-\ell}.$
\item For $j=1, \dots, \ell -1,$  the set $$I_{-j}:=\{ v \in e_{-j} + V_{<-j} + \C(e_{\ell} + e_{-\ell}) + V_{+}: \beta(v,v)=0 \} = N_{2\ell-1}\cdot e_{-j}.$$
\end{enumerate}
  Indeed, to prove the second assertion, we first note that $I_{-j}$ is irreducible of dimension $2\ell - j - 2$.  The variety $I_{-j}$ is clearly a quadratic hypersurface in the irreducible $2\ell-j-1$ dimensional affine set $e_{-j} + V_{<-j} + \C(e_{\ell} + e_{-\ell}) + V_{+}$.  To see that $I_{-j}$ is irreducible, we compute the equation $\beta(v,v)=0$.  For $v\in I_{-j}$, we write 
$v=e_{-j}+\sum_{k=j+1}^{\ell-1} c_{-k}e_{-k}+\lambda(e_{\ell}+e_{-\ell})+\sum_{k=1}^{\ell-1} c_{k}e_{k}$.  Then 
$$
\beta(v,v)=0 \Leftrightarrow \displaystyle \sum_{k=j+1}^{\ell-1} c_{k}c_{-k}+2\lambda^{2}+c_{j}=0.
$$
The polynomial on the right-hand side is clearly irreducible. 

Now note that $N_{2\ell-1}$ maps $I_{-j}$ to itself.  By computing the tangent space of the $N_{2\ell-1}$-orbit $N_{2\ell-1}\cdot e_{-j}$ on $I_{-j}$, we can see that $\dim(N_{2\ell-1}\cdot e_{-j})=\dim(I_{-j}).$  It follows that $N_{2\ell-1}\cdot e_{-j}=I_{-j}$ since orbits of unipotent groups on affine varieties are closed.  The proof of the first assertion follows by a simpler version of this argument.
\end{rem}

\begin{lem}\label{l:GrassD}
For $i \in \{ 1, \dots, \ell -1 \}$, the distinct $B_{2\ell-1}$-orbits on $\OGr(i,2\ell)$ are given by the following list:

\noindent 

\begin{equation}\label{eq:firstsubspaceD}
\mathcal{O}_{j_1, \dots, j_i}:=B_{2\ell-1}\cdot \mbox{span}\{e_{j_{1}},\dots, e_{j_{i}}\}.
\end{equation}

\begin{equation}\label{eq:secondsubspaceD}
\mathcal{O}_{j_1, \dots, j_{i-1}, \ell } :=B_{2\ell-1}\cdot \mbox{span}\{e_{j_{1}},\dots, e_{j_{i-1}}, e_{\ell}\}. 
\end{equation}

\begin{equation}\label{eq:thirdsubspaceD}
\mathcal{O}_{j_1, \dots, j_{i-1}, -\ell } :=B_{2\ell-1}\cdot \mbox{span}\{e_{j_{1}},\dots, e_{j_{i-1}}, e_{-\ell}\}. 
\end{equation}

\begin{equation}\label{eq:fourthsubspaceD}
\mathcal{O}_{j_1, \dots, j_{i-1}, \tilde{|j_i|}}:=B_{2\ell-1}\cdot \mbox{span}\{e_{j_{1}},\dots, e_{j_{i-1}}, \te_{|j_{i}|}\},
\end{equation}
where $\{j_{1},\dots, j_{i-1}, \, j_{i}\}\subset \{\pm 1,\dots, \pm (\ell-1)\}$, 
with $j_{k}\neq \pm j_{m}$.   
\end{lem}

\begin{proof}
We must show each orbit is one of the listed orbits, and no two of the listed orbits coincide.  
We first decompose $\C^{2\ell}=U_{1}\oplus U_{2}$, where $U_{1}=\mbox{span}\{e_{\pm 1},\dots, e_{\pm(\ell-1)}, e_{\ell}+e_{-\ell}\}$ is the fixed set of the involution $\sigma_{2\ell}$ from Section 2.2 of \cite{CE21I}, and $U_{2}=\mbox{span}\{e_{\ell}-e_{-\ell}\}$ is the $-1$-eigenspace of $\sigma_{2\ell}.$  For a plane $V\in \OGr(i, 2\ell)$, then either
\begin{equation}\label{eq:2conds}
\begin{split}
&(AD)\;   V\subset U_{1} \mbox{ or }\\
&(BD)\; \dim V\cap U_{1}=i-1.   \\
\end{split}
\end{equation}
Note that both of these conditions are $B_{2\ell-1}$-stable.  
First, suppose that $U$ satisfies (AD).  We recall that $K=SO(2\ell-1)$ is 
realized as determinant $1$ orthogonal transformations of the space 
$U_{1}$ (see Remark 2.1 of \cite{CE21I}).  Thus, the classification of $B_{2\ell-1}-$orbits on the space of $V$ in $\OGr(i, 2\ell)$  satisfying condition (AD) is equivalent to the classification of $B_{2\ell-1}$-orbits on $\OGr(i, 2\ell - 1)$. By Remark \ref{r:schubertD}, this is given by the list in Equation \eqref{eq:firstsubspaceD}.

We now consider orbits on the set of $V$ satisfying condition (BD).  For each such $V$, the subspace $V \cap U_1 \in \OGr(i-1, 2\ell -1 )$, and so by Remark \ref{r:schubertD}, we may conjugate it by $B_{2\ell-1}$  to the span of  $e_{j_1}, \dots, e_{j_{i-1}}$ where each $j_k \in \{ \pm 1, \dots, \pm (\ell - 1) \}$ and the absolute values of the $j_k$ are distinct.  Hence,  up to $B_{2\ell-1}$-conjugation 
$U=\mbox{span}\{e_{j_{1}},\dots, e_{j_{i-1}}, \, u\}$ for $\{j_{1},\dots, j_{i-1}\}\,$ as above and $u\notin V_{1}$.   We can write $u$ as 
\begin{equation}\label{eq:udecomp}
u=v_{+}+\mu e_{\ell}+\lambda e_{-\ell} +v_{-},
\end{equation}  
where $v_{+}\in V_{+},$
$v_{-}\in V_{-},$
 and $\mu\neq \lambda$.   Further, since the subspace $V$ is isotropic, we can assume that the coefficient of any $e_{\pm j_{k}}$ is zero in $v_{+}$ and $v_{-}$.  

Suppose that $v_{-}=0$.  Then since $u$ is an isotropic vector, at least
one of $\mu$ or $\lambda$ must be equal to zero.  If $\lambda=0$, 
then by Remark \ref{r:bvectororbit} there is $b\in B_{2\ell-1}$ such $b\cdot e_{\ell}=e_{\ell} + v_+$ and $b\cdot e_{j_{k}}=e_{j_{k}}$ 
for $k=1,\dots, i-1$.  Thus, $b\cdot U=\lspan\{e_{j_{1}},\dots, e_{j_{i-1}}, e_{\ell}\}$ is one of 
the orbits in Equation (\ref{eq:secondsubspaceD}).  If $\mu=0$, we can show similiarly that 
$B_{2\ell-1}\cdot U$ is one of the orbits in (\ref{eq:thirdsubspaceD}).  

Now suppose that $v_{-}\neq 0$ in Equation (\ref{eq:udecomp}), and let us assume that $\mu\neq 0$.  
Since $\beta(u, u)=0$, it follows that $\beta(v_{+}, v_{-})=-\lambda\mu$.
Let $v_{-} = \sum c_j e_{-j}$ and let $s$ be minimal so $c_s\not= 0.$ 
We claim that there is $b\in B_{2\ell-1}$ such that $b\cdot u$ is a nonzero scalar multiple of $ e_{-s} + e_{\ell}$
and $b$ fixes each $e_{j_k}$ in our basis of $U\cap V_1.$   Let $SO(2\ell - 2)$ be the subgroup of $SO(2\ell)$ fixing pointwise the subspace spanned by $e_{\ell}$ and $e_{-\ell}$, and let $B_{2\ell-2} = B_{2\ell-1} \cap SO(2\ell - 2).$   By Remark \ref{r:bvectororbit}, there is $b_1\in B_{2\ell-2}$ such that $b_1\cdot v_- = e_{-s}.$  By the same remark, there is $b_2$ in $N_{2\ell-1}$ such that $b_2 \cdot e_{-s} = e_{-s} - \lambda(e_{\ell} + e_{-\ell}) + \lambda^2 e_s.$     Then $b_2\cdot b_1\cdot u=w_+ + (\mu - \lambda)e_{\ell} + e_{-s}$ for some $w_+ \in V_{+}.$
Again by Remark \ref{r:bvectororbit}, there is $b_3 \in B_{2\ell-2}$ such that 
$b_3 \cdot e_{-s} =e_{-s} - w_+.$ Then there is an element $h$ in the diagonal torus of $B_{2\ell-1}$ so that  $b=h b_3b_2b_1$ satisfies the claim, and  we may assume $b$ stabilizes each $e_{j_k}$ in our basis of $U \cap V_1$ by working in the special orthogonal group of the span of the $e_k$ with $k \not= \pm j_r$ for $r=1, \dots, i-1$.
If we had assumed instead that $\lambda\neq 0$ and $\mu=0$, then we can argue similarly 
to obtain $B_{2\ell-1}\cdot U=B_{2\ell-1}\cdot\mbox{span}\{e_{j_{1}},\dots, e_{j_{i-1}}, \he_{s}\}$.  Using again Remark \ref{r:bvectororbit},
is easy to see that 
\begin{equation}\label{eq:hattildeorbit}
B_{2\ell-1}\cdot \mbox{span}\{e_{j_{1}},\dots, e_{j_{i-1}}, \tilde{e}_{s}\}=B_{2\ell-1}\cdot \mbox{span}\{e_{j_{1}},\dots, e_{j_{i-1}}, \hat{e}_{s}\}.
\end{equation}

For uniqueness, note that for each $i$ element subset $\{ r_1, \dots, r_i \}$ of
$\{ \pm 1, \dots, \pm l \}$ with distinct absolute values, each of the $B_{2\ell-1}$-orbits $\mathcal{O}_{r_1, \dots, r_i}$ is contained in $B\cdot \lspan \{ e_{r_1}, \dots, e_{r_i} \}.$  Further, the $B_{2\ell-1}$-orbit $\mathcal{O}_{j_1, \dots, j_{i-1}, \tilde{|j_i|}}$ is contained in $B\cdot \lspan \{ e_{j_1}, \dots, e_{j_{i-1}}, e_{-|j_i|}\}.$  From Remark \ref{r:schubertD}, the only possible coincidences are between orbits $\mathcal{O}_{j_1, \dots, j_{i-1}, \tilde{|j_i|}}$ and $\mathcal{O}_{j_1, \dots, j_{i-1},-|j_i|}$ with $j_1, \dots, j_i$ elements with distinct absolute values in $\{ \pm 1, \dots, \pm (\ell - 1) \}.$ But the first orbit is of type (BD) and the second is of type (AD), so they cannot coincide.

\end{proof}

\begin{lem}\label{l:spanstandardD}
Let $\mathcal{F}=(v_1 \subset \dots \subset v_{\ell-1})$ be a $SO(2\ell)$-flag in standard form and let $V_{m}=\lspan\{v_{1},\dots, v_{m}\}$.  
\par\noindent (i) If for all $i \le m$, $v_i = e_{j_i}$ with $j_i \in \{ \pm 1, \dots, \pm \ell \}$, then $V_m = \lspan\{ e_{j_1}, \dots, e_{j_m} \}.$
\par\noindent (ii) Suppose $v_{i}=e_{-\ell}$ (resp. $v_{i}=e_{\ell}$) for some $i$.  Then for all $m\geq i$, 
$V_{m}=\lspan\{\overline{v_{1}},\dots, \overline{v_{m}}\}$, where $\overline{v_{k}}=v_{k}$ 
if $v_{k}=e_{j_{k}}$ is a standard basis vector, and $\overline{v_{k}}=e_{-j_{k}}$ if $v_{k}=\he_{j_{k}}$ 
(resp. if $v_{k}=\tilde{e}_{j_{k}}$) for $k=1,\dots, m$.  
\par\noindent (iii)  Suppose that for all $j\leq m$, $v_{j}\neq e_{\pm \ell}$ and at least one 
of the vectors $v_{1},\dots, v_{m}$ is a tilde or hat vector.  Let $1\leq k_{1}<\dots <k_{r}\leq m$ be 
the subsequence of $\{1,\dots, m\}$ such that $v_{k_{i}}=\he_{i_{k_{i}}}$ or $v_{k_{i}}=\tilde{e}_{i_{k_{i}}}.$ 
Let $s_{1}<\dots <s_{m-r}$ be the complementary subsequence with $v_{s_{i}}=e_{i_{s_{i}}}$.  Then 
\begin{equation}\label{eq:hatandtilde}
B_{2\ell-1}\cdot V_{m}=B_{2\ell-1}\cdot\mbox{span}\{e_{i_{s_{1}}},\dots, e_{i_{s_{m-r}}}, e_{-i_{k_{1}}},\dots, e_{-i_{k_{r-1}}}, \tilde{e}_{i_{k_{r}}}\}.
\end{equation}
\end{lem}

\begin{proof}
Assertions (i) and (ii) are similar to the proof of Lemma \ref{l:spanstandard}, as is the proof of (iii) when $v_{j_m}$ is a tilde vector.  In case $v_{j_m}$ is a hat vector, a similar argument proves Equation \eqref{eq:hatandtilde} with $\tilde{e}_{j_m}$ replaced with $\he_{j_m}$.  By Equation \eqref{eq:hattildeorbit}, the Remark follows.
\end{proof}

\begin{thm}\label{thm:stdD}
The map  
\begin{equation}\label{eq:flagassocD} 
\Psi:\mathcal{F} \to B_{2\ell-1}\cdot \F
\end{equation}
 is a bijection between $SO(2\ell)$-flags in standard form and $B_{2\ell-1}$-orbits on $\B_{2\ell}$.
\end{thm}

\begin{proof}
By the same reasoning as in the proof of Theorem \ref{thm:std}, and using Theorem \ref{thm:orthocount} (i) and Proposition \ref{p:flagsandPILSD}, it suffices to show that the map $\Psi$ is injective. As in that proof, we let $Q=B_{2\ell-1}\cdot \F$, and $Q^{\prime}=B_{2\ell-1}\cdot \F^{\prime}$ where $\F=(v_{1}\subset v_{2}\subset \dots\subset v_{i}\subset \dots \subset v_{\ell - 1} )$, 
and  $\mathcal{F}^{\prime}=(v_{1}^{\prime}\subset v_{2}^{\prime}\subset \dots\subset v_{i}^{\prime}\subset \dots v_{\ell - 1}^{\prime})$ are $SO(2\ell)$-flags in standard form.  We prove by induction on $i$ that $v_i = v_{i}^{\prime}$ for $i=1, \dots, \ell - 1.$

To begin the induction, consider the orbit $\mathcal{O}_{1}=B_{2\ell-1}\cdot V_{1}=B_{2\ell-1}\cdot V_{1}^{\prime}$.   In the case the orbit is one of those from Equations
\eqref{eq:firstsubspaceD},  \eqref{eq:secondsubspaceD}, or \eqref{eq:thirdsubspaceD}, it follows from Lemma \ref{l:GrassD} and the definition of standard form that $v_{1} = v_{1}^{\prime}.$  The remaining case is when $\mathcal{O}_1=\mathcal{O}_{\te_{j_1}}.$  In this case, by using Lemma \ref{l:GrassD},
it suffices to show that assuming that $v_1=\te_{j_1}$ and $v_{1}^{\prime}=\he_{j_1}$ leads to a contradiction.  By the definition of standard form, some $v_k=e_{\ell}$ for $k > \ell$.  By Lemma \ref{l:spanstandardD},
it follows that $B_{2\ell-1}\cdot V_k$ is one of the orbits appearing in Equation \eqref{eq:secondsubspaceD}.  Since $\he_{j_i}$ occurs in $\F^{\prime}$, then $e_{\ell}$ does not occur in $\F^{\prime}$ by Remark \ref{r.standardD}.  Thus, either $e_{-\ell}$ occurs in $\F^{\prime}$ or neither of $e_{\pm \ell}$ occurs in $\F^{\prime}.$  But then we can use Lemma \ref{l:spanstandardD} to see that we cannot obtain one of the orbits in Equation \eqref{eq:secondsubspaceD}, and this is a contradiction.  Hence, $v_1 = v_1^{\prime}.$
 For the inductive step, we may assume 
 that $v_{i}=v_{i}^{\prime}$ for $i=1,\dots, m-1$, so that $V_{m-1}=V_{m-1}^{\prime}$.
We prove that $v_m = v_m^{\prime}$ using our assumption that $B_{2\ell-1}\cdot V_m = B_{2\ell-1}\cdot V_m^{\prime}.$  To prove this, we argue based on the case of
$B_{2\ell-1}\cdot V_{m-1}$ described in Lemma \ref{l:GrassD}.  If we are in the case of Equation \eqref{eq:firstsubspaceD}, then recovering $v_m$ from $B_{2\ell-1}\cdot V_m$ is similar to the $m=1$ case.  If we are in the case of Equation \eqref{eq:secondsubspaceD}, then by Lemma \ref{l:spanstandardD}, one of the vectors $v_k = e_{\ell}$ for $k \le m-1$.  Thus, from the definition of standard form, we see that $v_m = e_{j_m}$ with $j_m \not= \pm \ell$, and then by computing $V_{m-1} + \C v_m$ we see that $v_m = v_m^{\prime}.$   If we are in the case of Equation \eqref{eq:thirdsubspaceD}, then we argue similarly using $e_{-\ell}$ in place of $e_{\ell}.$   
Finally, suppose that 
$B_{2\ell-1}\cdot V_{m-1}$ is one of the orbits in (\ref{eq:fourthsubspaceD}).  Then by Lemma \ref{l:spanstandardD} some vectors $v_{k}$ 
with $k\leq m-1$ are either tilde vector or hat vectors and $v_{k}\neq e_{\pm \ell}$ for all $k\leq m-1$.  Let $B_{2\ell-1}\cdot V_{m-1}=\mbox{span}\{e_{j_{1}},\dots, e_{j_{m-2}}, \tilde{e}_{j_{m-1}}\}$.  If $v_{m}=e_{j_{m}}$ or $e_{\pm \ell}$, then $v_{m}=v_{m}^{\prime}$ by Lemmas \ref{l:GrassD} and \ref{l:spanstandardD}.
If $v_{m}=\tilde{e}_{j_{m}}$ or $v_{m}=\he_{j_{m}}$, then condition (c) in the definition of the standard form and Equation (\ref{eq:hatandtilde}) yield 
$B_{2\ell-1}\cdot V_{m}=\mbox{span}\{e_{j_{1}},\dots, e_{j_{m-2}}, \, e_{-j_{m-1}}, \tilde{e}_{j_{m}}\}.$   Since tilde and 
hat vectors cannot both occur in the standard form, we are forced to have $v_{m}=v_{m}^{\prime}$ in this case as well. 
\end{proof}

\subsection{Representatives in the case of $SO(2\ell + 1)$}\label{ss:Breps}

In this subsection, we give explicit representatives for $B_{n-1}$ on $\B_n$ in the case where $G=SO(n)$ with $n=2\ell + 1$ odd.  We omit most proofs, since they are generally simpler versions of the proofs in the type D case in Section \ref{ss:Dreps}.  Recall the description of $\B_n$ in terms of isotropic flags from Section \ref{ss:Borels}.

\begin{dfn-nota}\label{d:stdB}

\begin{enumerate} 
\item For $i=1,\dots, \ell$, define  $\he_{i}:=e_{i}+\sqrt{2}e_{0}-e_{-i}$.  For $1\leq j<i\leq \ell$, define 
$\he_{i,-j}:=\he_{i}+e_{-j}$.  We refer to a vector of the form $\he_{i}$ as a hat vector of the first kind 
and a vector of the from $\he_{i,-j}$ as a hat vector of the second kind.  We say that the index of both $\he_i$ and of the standard basis vector $e_{i}$ is $|i|$,  and the index of a hat vector of the second kind $\he_{i,-j}$ is $j.$

\item We say that the isotropic flag
\begin{equation}\label{eq:basicflagB}
\mathcal{F}:=(v_{1}\subset \dots \subset v_{i}\subset\dots\subset v_{\ell})
\end{equation}
is in \emph{standard form} if $v_{i}$ is either a standard basis vector, a hat vector of the first kind, or a 
hat vector of the second kind, and $\mathcal{F}$ satisfies the following conditions:
\begin{enumerate}
\item  If $v_{i}=\he_{j_{i}}$ is a hat vector of the first kind, then $v_{k}=e_{j_{k}}$ for all $k > i$.
\item If $i < k$, $v_{i}=\he_{a,-j_{i}}$ and $v_{k}=\he_{a, -j_{k}}$,  then $j_{i}<j_{k}$.  
\item  Each integer in $\{ 1, \dots, \ell\}$ is the index of exactly one vector $v_i$ occurring in $\F.$   
\end{enumerate}
\end{enumerate}
\end{dfn-nota}

\begin{rem}\label{r:firstkind}  
We assume that a hat vector of the second kind $\he_{a,-j}$ occurs in a $SO(2\ell + 1)$-standard flag $\F.$   Then the hat vector of the first kind $\he_a$ must occur in $\F$, and if $\he_{b,-k}$ is another hat vector of the second kind in $\F$, then $b=a$.   Indeed, assume that $\he_{a,-j}$ occurs in $\F$.  Then the vectors $e_{\pm a}$ does not occur in  $\F$ since $\F$ is isotropic, and for the same reason, $\he_{b, -a}$ does not occur.  By property (c) in the definition of a $SO(2\ell+1)$-standard flag, the index $a$ must occur, and hence $\he_a$ occurs in $\F$.  If $b\not= a$, then the vector $\he_b$ cannot occur in our flag since it is isotropic, and hence by the first observation, $\he_{b,-k}$ cannot occur in $\F$ for any $k$.
\end{rem}

\begin{prop}\label{p:flagsandPILSB}
There is a bijection 
\begin{equation}\label{eq:GammaB}
\Gamma:\SPILzero\longleftrightarrow \{SO(2\ell + 1)-\mbox{Flags in standard form}\}.
\end{equation}

\end{prop}
\begin{proof}
For $\Sigma=\{\sigma_{1},\, \sigma_{2},\,\dots, \, \sigma_{t}\} \in \SPILzero,$ let $\ell(\sigma_j)=k_j$ and let $\sigma_{j}=(i_{j,1},\dots, i_{j, k_{j}}).$  We may choose the ordering of the $\sigma_j$ uniquely so that
$i_{1,k_{1}}<\dots <i_{t-1,k_{t-1}}$ and $0 \in \sigma_{t}$.  Let $a:=i_{t-1,k_{t-1}}.$  Then we define 
\begin{equation}\label{eq:GammaBdef}
\begin{split}
\Gamma(\Sigma):= 
&(e_{i_{1,1}}\subset e_{i_{1,2}}\subset\dots\subset \he_{a,-i_{1,k_{1}}}\subset e_{i_{2,1}}\subset\dots\subset \he_{a, -i_{2,k_{2}}}\subset\dots\\
&\subset e_{i_{t-1,1}}\subset \dots \subset \he_{a}
\subset e_{i_{t,1}}\subset\dots\subset e_{i_{t,k_t - 1}}).
\end{split}
\end{equation}
By construction, this is a $SO(2\ell + 1)$-flag in standard form.
To define the inverse $\Lambda$, note that for a $SO(2\ell + 1)$-flag $(v_1\subset \dots\subset v_{\ell})$ in standard form,  if we let $k_1 < \dots < k_r$ be the subsequence consisting of hat vectors, then by Remark \ref{r:firstkind} it follows that 
$v_{k_{i}}=\he_{a, i_{k_{i}}}$ for $i=1,\dots, r-1$ and $v_{k_{r}}=\he_{a}$.
We define $\Lambda(\F) \in \SPILzero$  by Equation (\ref{eq:Lambdadefn}) with  zero adjoined at the end of the last list, and it is routine to check that $\Gamma$ and $\Lambda$ are inverse bijections.
\end{proof}


\begin{lem}\label{l:GrassB}
Let $i \in \{ 1, \dots, \ell \}$. The distinct $B_{2\ell}$-orbits on $\OGr(i,2\ell+1)$ are given by the orbits

\begin{equation}\label{eq:firstsubspaceB}
\mathcal{O}=B_{n-1}\cdot \mbox{span}\{e_{j_{1}},\dots, e_{j_{i}}\}.
\end{equation}

\begin{equation}\label{eq:secondsubspaceB}
\mathcal{O}=B_{n-1}\cdot \mbox{span}\{e_{j_{1}},\dots, e_{j_{i-1}}, \he_{|j_{i}|}\},
\end{equation}
where $\{j_{1},\dots, j_{i}\}\subset \{\pm 1,\dots, \pm \ell\}$, and with $j_{k}\neq \pm j_{m}$.
\end{lem}
  
To prove this lemma, it is useful to recall that $K=SO(2\ell)$ is realized in $G$ as the elements fixing $e_0$ (\cite{CE21I}, Remark 2.1) and to prove the following analogue of Remark \ref{r:bvectororbit}. Recall that $N_{2\ell}$ is the derived subgroup of $B_{2\ell}.$
 First, for $1 \le j \le \ell$, $N_{2\ell}\cdot e_j =  e_j + \sum_{k < j} \C e_k$ and 
\begin{equation}\label{eq:odd4.13one}
N_{2\ell}\cdot e_{-j} = \{ v\in e_{-j} +\sum_{k > j} \C e_{-k} + \sum_{r=1}^{\ell} \C e_r : \beta(v,v)=0\}.
\end{equation}
Second, for $1 \le j \le \ell$, 
$$
N_{2\ell}\cdot (e_{-s} - e_s) = \{ v\in e_{-s} + \sum_{k > s} \C e_{-k} + \sum_{r=1}^{\ell} \C e_r : \beta(v,v)=-2\}.
$$
These assertions follow similarly to the proof of Remark \ref{r:bvectororbit}, and the proof of the Lemma is similar to the proof of Lemma \ref{l:GrassD}.

\begin{thm}\label{thm:stdB}
The map  
\begin{equation}\label{eq:flagassocB} 
\Psi:\mathcal{F} \to B_{2\ell}\cdot \F
\end{equation}
 is a bijection between $SO(2\ell+ 1)$-flags in standard form and $B_{2\ell}$-orbits on $\B_{2\ell + 1}$.
\end{thm}

To prove this theorem, it is useful to compute $V_{m-1} + \C v_m$ when
$(v_1 \subset \dots \subset v_\ell)$ is a $SO(2\ell+1)$-flag in standard form and $V_k=\lspan \{ v_1, \dots, v_k \}.$   In particular, it is useful to note the following assertions.  First, $B_{2\ell} \cdot \he_{a,-j}=B_{2\ell}\cdot \he_{j}$ for $1 \le j \le \ell$.   Secondly, if $k_j > k_i$, then 
$B_{2\ell} \cdot \lspan\{ \he_{a,-k_i}, v_2 \} = B_{2\ell}\cdot \lspan \{ e_{-k_i}, v_2 \}$ when $v_2 = \he_{a,-k_j}$ or $\he_a.$   Given these assertions, the proof follows the same steps as the proof of Theorem \ref{thm:stdD}.

\section{Monoid action on standard forms}\label{s:monoid}

 
An essential aspect of the geometry of closures of $B_{n-1}$-orbits in $\B_n$ is an understanding of the action of the monoid given by simple roots of $\fg$ and of $\fk=\fg_{n-1}$, which gives information about the relation between different orbits.   In this section, we compute the monoid action for roots of $\fg.$

 \subsection{Monoid Action by simple $\fg$-roots}\label{ss:monoid}

Let $R$ be a complex, reductive algebraic group, let $\B=\B_{R}$ be the flag 
variety of $R$, and let $M$ be an algebraic subgroup of $R$ acting on $\B$ with finitely many orbits. 
 Let $S$ be the simple reflections of the Weyl group of $R$, and recall  the operator $m(s)$ on $M\backslash \B$ from \cite{RS} and other sources.  For this, let $s=s_{\alpha}$ for a simple root $\alpha$ and let $\pi_{\alpha}:\B \to {\mathcal{P}}_{\alpha}$ be the canonical projection from the flag variety to the variety of parabolics of $R$ of type $\alpha.$   Then for $Q_{M}\in M\backslash \B$, recall that 
$m(s)*Q_{M}$ is the unique $M$-orbit open and dense in $\pi_{\alpha}^{-1}(\pi_{\alpha}(Q_{M})).$ Further, $Q_{M} \not= m(s)*Q_{M}$ if and only if $\dim(m(s)*Q_{M})=\dim(Q_{M})+1.$  
If $\alpha$ is a simple root of $\fr$, recall that $\alpha$ can be either real, non-compact imaginary, compact imaginary, or complex for $Q_{M}$, depending on how $Q_{M}$ and $m(s)*Q_{M}$ meet a fibre $\pi_{\alpha}^{-1}(\pi_{\alpha}(\fb))$ for $\fb \in Q_{M}$ (see Section 2.3 of \cite{CE21I} for details).  

For the remainder of this section, we let $M=K=G_{n-1}$ or $M=B_{n-1}$ and $R=G$  with $G=GL(n)$ or $SO(n).$  To compute the monoid action explicitly, it is useful to make the following standard observations, which may be taken to be the definition of the monoid action in these cases.

 Let $\F_{+}\in \B_{n}$ be the standard upper triangular flag given in Equation (\ref{eq:upper}).  
\begin{prop}\label{prop:monoid}
Let $Q_M=M\cdot \Ad(v)\F_{+}$ be an $M$-orbit, where $v\in G$, and let $\alpha\in \Pi_{\fg}$ be a standard simple root.  
Let $s_{\alpha}\in W$ be the simple reflection determined by $\alpha$, and let $u_{\alpha}\in G$ be the Cayley transform with respect to the root $\alpha$ defined in Notation 2.5 of \cite{CE21I}.
\begin{enumerate}
\item The root $\alpha$ is non-compact imaginary or real for $Q_{M}$ if the variety 
$\pi_{\alpha}^{-1}(\pi_{\alpha}(Q_{M}))$
consists of exactly three orbits.  The root $\alpha$ is non-compact imaginary if $\dim (m(s_{\alpha})*Q_{M}) = \dim(Q_{M})+1$ and 
real if $m(s_{\alpha})*Q_{M}=Q_{M}$. If $\alpha$ is non-compact imaginary then
$m(s_{\alpha})*Q_{M}=M\cdot \Ad(vu_{\alpha})\F_{+}$.
\item We say $\alpha$ is complex for $Q_{M}$ if the variety 
$$\pi_{\alpha}^{-1}(\pi_{\alpha}(Q_{M}))=Q_{M}\cup M\cdot \Ad(v)s_{\alpha}(\F_{+}).$$
The root $\alpha$ is  \emph{complex stable} if $m(s_{\alpha})*Q_{M}\neq Q_{M}$ and \emph{complex unstable} if 
$m(s_{\alpha})*Q_{M}=Q_{M}$.
\item For $M=K$, we say that $\alpha$ is compact imaginary for $Q_{K}$ if
$$
\pi_{\alpha}^{-1}(\pi_{\alpha}(Q_{K}))=Q_{K}.
$$
\end{enumerate}
\end{prop}

The following observation will be useful for computing the monoid action on standard forms.  
\begin{lem}\label{l:cplxvncpt}
Let $Q_{M}=M\cdot \Ad(v)\F_{+}$ for $v\in G$ and let $\alpha\in\Pi_{\fg}$ be a simple root 
with $m(s_{\alpha})*Q_{M}\neq Q_{M}$.  Then $\alpha$ is non-compact for $Q_{M}$ if and only if $M\cdot \Ad(vu_{\alpha})\F_{+}\neq M\cdot \Ad(v)s_{\alpha}(\F_{+})$. 

\end{lem}
\begin{proof}
Since the variety $\pi_{\alpha}^{-1}(\pi_{\alpha}(Q_{M}))\mapsto \pi(Q_{M})$ is an $M$-homogeneous $\mathbb{P}^{1}=\B_{\fsl(2)}$-bundle, it suffices to prove the statement when $G=SL(2)$ and $M\subset SL(2)$ is a subgroup with finitely many orbits on $\B_{\fsl(2)}$, which is an easy computation. 

\end{proof}

\begin{nota}
If $Q=B_{n-1}\cdot\F$ with $\F$ a flag in standard form and $\alpha\in\Pi_{\fg}$, then we denote by $m(s_{\alpha})*\F$ the unique flag in standard form contained in $m(s_{\alpha})*Q$.  
\end{nota}

 \subsection{Computing monoid actions}\label{ss:monoidcompute}

Let $Q$ be a $B_{n-1}$-orbit contained in a $K$-orbit  $Q_{K}\in \B_{n}$.
  In Theorems 4.11 and 4.12 of \cite{CE21I}, we describe the monoid action on $\Borbitspace$ by roots $\alpha\in \Pi_{\fg}$ for which $\alpha$ is compact imaginary for $Q_{K}$ as defined in Part (3) of Proposition \ref{prop:monoid}.
  We now compute the monoid action by roots $\alpha\in \Pi_{\fg}$ on orbits $Q\in\Borbitspace$ such that $m(s_{\alpha})*Q_{K}\neq Q_{K}$.  It follows from Proposition 4.8 of \emph{loc.cit} that for such an $\alpha$,  $m(s_{\alpha})*Q\neq Q$. 

We being with the type $A$ case.  Recall our labelling of the elements of $K\backslash\B_{n}$ by $Q_i$ or $Q_{i,j}$ from Section \ref{ss:Korbits}.
Recall the $K$-stable decomposition of $\C^{n}$, 
$\C^{n}=U_{1}\oplus U_{2}$ where $U_{1}=\mbox{span}\{e_{1},\dots, e_{n-1}\}$ 
and $U_{2}=\mbox{span}\{e_{n}\}$ given after Remark \ref{r.schubert}. Let 
$p_{1}:\C^{n}\to U_{1}$ be the projection onto $U_{1}$ off of $U_{2}$.  
Let $\F =(v_1 \subset \dots\subset v_n)\in\B_{n}$ be a flag in standard form. By Remark \ref{r.glstandard}, exactly one vector $p_{1}(v_i)=0$ and 
if we let $p_1(\F)$ be the sequence  $(p_{1}(v_1) \subset \dots\subset p_{1}(v_n))$ with $p_{1}(v_i)$ omitted, then $p_1(\F)$ is in $\B_{n-1}$
 and is stabilized by the standard diagonal Cartan subalgebra of $\fk$.  

\begin{lem}\label{l:5.1}
Let $\F=(v_{1}\subset v_{2}\subset\dots\subset v_{n})$ be a flag in standard form.  
\begin{enumerate}
\item The flag $\F\in Q_{j}$ if and only if $\F$ contains no hat vectors and $v_{j}=e_{n}$. 
\item  The flag $\F\in Q_{i,j}$ with $i<j$ if and only if $v_{i}$ is a hat vector, 
$v_{k}$ is a standard basis vector for all $k<i$, and $v_{j}=e_{n}$.  
\end{enumerate}
\end{lem}
\begin{proof}
We first prove the sufficiency of both (1) and (2).  Suppose that $\F$ contains no hat vectors and that $v_{j}=e_{n}$.  Since $W_K$ acts on $\B_n$ as permutation matrices fixing $e_n$,  there exists a $w\in W_{K}$ such that $w(p_{1}(\F))=\F_{+,n-1}$ 
where $\F_{+,n-1}$ is the standard upper triangular flag in $\C^{n-1}$ (see Equation (\ref{eq:upper})).  
Since the map $p_{1}$ is $K$-equivariant it follows that $w(\F)=\F_{j}$,  where $\F_{j}$ is the flag in Equation (\ref{eq:typeAflagclosed}) with $i=j$.   
Now suppose that $\F=(v_{1}\subset \dots\subset v_{i}\subset\dots \subset v_{j}\subset \dots\subset v_{n})$ with $v_{i}$ a hat vector, $v_{k}$ a standard basis vector for all $k<i$, and $v_{j}=e_{n}$.  Then arguing as above, we can find a $\sigma\in W_{K}$ such that $\sigma(p_{1}(\F))=\F_{i,j}^{\prime}$, where 
$$
\F_{i,j}^{\prime}=(e_{1}\subset \dots \subset e_{i-1}\subset \underbrace{e_{j-1}}_{i}\subset e_{i}\subset \dots \subset \underbrace{e_{j-2}}_{j-1}\subset\underbrace{e_{j}}_{j}\subset \dots \subset e_{n-1}). 
$$
As above, it follows that 
$$
\Ad(\dot{\sigma})\F=(e_{1}\subset \dots\subset e_{i-1}\subset\underbrace{\he_{j-1}}_{i}\subset v_{i+1}^{\prime}\subset \dots \subset v_{j-1}^{\prime}\subset \underbrace{e_{n}}_{j}\subset e_{j}\subset\dots\subset e_{n-1}),
$$
where $v_{t}^{\prime}=e_{t-1}$ or $\he_{t-1}$ for all $t=i+1, \dots, j-1$.  Now if $v_{t}^{\prime}=e_{t-1}$ for all $t$, then 
$\Ad(\dot{\sigma})\F=\F_{i,j}$, where $\F_{i,j}$ is the flag in Equation (\ref{eq:typeAflag}) whence $\F_{i,j}\in Q_{i,j}$.  
If, on the other hand, $v_{t}^{\prime}=\he_{t-1}$ for some $t$, then consider the space 
$\mbox{span}\{\he_{j-1}, \, \he_{t-1}\}=\mbox{span}\{\he_{j-1}, e_{t-1}-e_{j-1}\}.$   
This space is $K$-conjugate to $\mbox{span}\{\he_{j-1}, e_{t-1}\}$ without altering any other vectors 
in $\Ad(\dot{\sigma})\F$.  Thus, there exists $k\in K$ such that $\Ad(k \dot{\sigma})\F=\F_{i,j}$ and $\F_{i,j}\in Q_{i,j}$ and this completes the proof of sufficiency.  Necessity now follows since every flag in standard form satisfies exactly one of the sufficient conditions in the Lemma.
\end{proof}

 Let $\alpha \in \Pi_{\fg}.$  Note that if $Q_K=Q_i$ is a closed $K$-orbit in $\B_n$,  then by Proposition 4.9 of \cite{CE21I}, $m(s_{\alpha})*Q_i\not= Q_i$ if and only if $\alpha=\alpha_{i-1}$ or $\alpha_i$.   By the same result, if $Q_K=Q_{i,j}$ is a non-closed $K$-orbit in $\B_n$, then $m(s_{\alpha})*Q_{i,j} \not= Q_{i,j}$ if and only if $\alpha=\alpha_{i-1}$ or $\alpha_j.$

\begin{prop}\label{p:monoidactiononPILS}
Let $\fg=\fgl(n)$ and $Q=B_{n-1}\cdot \F$ and let $Q$ be in the $K$-orbit $Q_{K}$.\\
\noindent Case I:  Suppose $Q_{K}$ is not closed so that $Q_{K}=Q_{i,j}$ with $1\leq i<j\leq n$. Then $\mathcal{F}$ has the form
\begin{equation}\label{eq:baseflagmonoid}
\F=(e_{\ell_{1}}\subset e_{\ell_{2}}\subset \dots\subset e_{\ell_{i-1}}\subset\underbrace{\he_{\ell_{i}}}_{i}\subset v_{i+1}\subset\dots \subset v_{j-1}\subset\underbrace{e_{n}}_{j}\subset e_{\ell_{j+1}}\subset\dots\subset e_{\ell_{n-1}}).
\end{equation}

\begin{enumerate}
\item For the simple root $\alpha_{i-1},$  
\begin{enumerate}
\item if $\ell_{i-1}<\ell_{i}$, then $\alpha_{i-1}$ is complex stable for $Q$ and 
\begin{equation}\label{eq:i-1cplx}
m(s_{\alpha_{i-1}})*\mathcal{F}=(e_{\ell_{1}}\subset\dots\subset e_{\ell_{i-2}}\subset\underbrace{\he_{\ell_{i}}}_{i-1}\subset \underbrace{e_{\ell_{i-1}}}_{i}\subset v_{i+1}\subset\dots\subset v_{j-1}\subset e_{n}\subset \dots \subset e_{\ell_{n-1}}).
\end{equation}
\item If $\ell_{i-1}>\ell_{i}$, then $\alpha_{i-1}$ is non-compact for $Q$ and 
\begin{equation}\label{eq:i-1nc}
m(s_{\alpha_{i-1}})*\mathcal{F}=(e_{\ell_{1}}\subset\dots\subset e_{\ell_{i-2}}\subset\underbrace{\he_{\ell_{i-1}}}_{i-1}\subset\underbrace{\he_{\ell_{i}}}_{i}\subset v_{i+1}\subset \dots\subset v_{j-1}\subset e_{n}\subset \dots\subset e_{\ell_{n-1}}).
\end{equation}
\end{enumerate}
\item For the simple root $\alpha_{j}$, 
let $k:=\min(\{\ell_{i}\}\cup\{t_{m}|\, v_{m}=\he_{t_{m}}, \, m=i+1,\dots, j-1.\}).$
\begin{enumerate}
\item If $\ell_{j+1}>k$, then $\alpha_{j}$ is complex stable for $Q$ and 
\begin{equation}\label{eq:jcplx}
m(s_{\alpha_{j}})*\mathcal{F}=(e_{\ell_{1}}\subset \dots\subset e_{\ell_{i-1}} \subset \he_{\ell_{i}}\subset v_{i+1}\subset\dots\subset  v_{j-1}\subset \underbrace{e_{\ell_{j+1}}}_{j}\subset\underbrace{e_{n}}_{j+1}\subset \dots\subset e_{\ell_{n-1}}).
\end{equation}
\item If $k>\ell_{j+1}$, then $\alpha_{j}$ is non-compact for $Q$ and 
\begin{equation}\label{eq:jnc}
m(s_{\alpha_{j}})*\mathcal{F}=(e_{\ell_{1}}\subset \dots\subset e_{\ell_{i-1}}\subset\he_{\ell_{i}}\subset v_{i+1}\subset \dots \subset v_{j-1} \subset\underbrace{\he_{\ell_{j+1}}}_{j}\subset \underbrace{e_{n}}_{j+1} \subset\dots\subset e_{\ell_{n-1}}).
\end{equation}
\end{enumerate}
\end{enumerate}

\noindent Case II: Let $Q_{K}$ be closed, so $Q_{K}=Q_{i}$ for $i=1,\dots, n$.
 Then 
$$
\F=(e_{\ell_{1}}\subset\dots\subset e_{\ell_{i-1}}\subset \underbrace{e_{n}}_{i} \subset e_{\ell_{i+1}}\subset\dots\subset e_{\ell_{n-1}}).
$$
The roots $\alpha_{i-1}$ and $\alpha_{i}$ are both non-compact for $Q$ and 
\begin{equation}\label{eq:typeAclosed}
\begin{split}
m(s_{\alpha_{i-1}})*\F&=(e_{\ell_{1}}\subset\dots\subset \underbrace{\he_{\ell_{i-1}}}_{i-1}\subset \underbrace{e_{n}}_{i} \subset e_{\ell_{i+1}}\subset\dots\subset e_{\ell_{n-1}}), \\  m(s_{\alpha_{i}})*\F&=(e_{\ell_{1}}\subset\dots\subset e_{\ell_{i-1}}\subset \underbrace{\he_{\ell_{i+1}}}_{i}\subset \underbrace{e_{n}}_{i+1} \subset e_{\ell_{i+2}}\subset\dots\subset e_{\ell_{n-1}}).
\end{split}
\end{equation}
\end{prop}

\begin{proof}
We  prove the statements when $Q_{K}=Q_{i,j}$.    
Equation (\ref{eq:baseflagmonoid}) follows from Lemma \ref{l:5.1}. 
Let $\F= \Ad(v)\F_{+}$ for some $v\in G$.  By Proposition 4.8 of \emph{loc. cit.}, 
we know that $m(s_{\alpha_{j}})*Q\neq Q.$ 
We need only determine whether $\alpha_{j}$ is non-compact or complex stable for $Q$.  
For this, it suffices by Lemma \ref{l:cplxvncpt} to compare the orbits $Q_{1}:=B_{n-1}\cdot\Ad(v) s_{\alpha_{j}}(\F_{+})$ and $Q_{2}:=B_{n-1}\cdot \Ad(v)\Ad(u_{\alpha_{j}})\F_{+}$.  Let $\F_{1}$ and $\F_{2}$ be the unique flags in standard form in the orbits $Q_{1}$ and $Q_{2}$ respectively.  
Then straightforward computations with flags show that $\F_{1}$ is the flag in (\ref{eq:jcplx}) 
and $Q_{2}$ contains the flag in (\ref{eq:jnc}) (though this flag need not be in standard form).  Let $k$ be given as in Part (2) of Case I of the Proposition. 
Now if $\ell_{j+1}>k$, then as in the proof of Lemma  \ref{l:spanstandard}, the space $\mbox{span}\{\hat{e}_{k}, \hat{e}_{\ell_{j+1}}\}$ is $B_{n-1}$-conjugate to $\mbox{span}\{\hat{e}_{k}, e_{\ell_{j+1}}\}$, where the conjugation leaves all other vectors in the flag unchanged.  Thus, $\F_{2}=\F_{1}$, so $Q_{2}=Q_{1}$.  
Lemma \ref{l:cplxvncpt} now implies $\alpha_{j}$ is complex stable for $Q$ and $m(s_{\alpha_{j}})*\F=\F_{1}$ 
is given by (\ref{eq:jcplx}) by Proposition \ref{prop:monoid}.  
On the other hand, if $\ell_{j+1}<k$ then the flag in (\ref{eq:jnc}) is in standard form 
so that $\F_{2}$ is given by (\ref{eq:jnc}).  Thus, $\F_{1}\neq \F_{2}$, and it follows 
from Theorem \ref{thm:std} that $Q_{1}\neq Q_{2}$.  
Now Lemma \ref{l:cplxvncpt} implies that $\alpha_{j}$ is non-compact for $Q$ 
and $m(s_{\alpha_{j}})*\F=\F_{2}$ is given by (\ref{eq:jnc}) by Proposition \ref{prop:monoid}.

Now we consider the simple root $\alpha_{i-1}$.  Let $Q_{1}:=B_{n-1}\cdot\Ad(v) s_{\alpha_{i-1}}(\F_{+})$ and $Q_{2}:=B_{n-1}\cdot \Ad(v)\Ad(u_{\alpha_{i-1}})\F_{+}$ and let $\F_{1}$ and $\F_{2}$ be as above.  In this case,
$\F_{1}$ is the flag in Equation (\ref{eq:i-1cplx}) and $Q_{2}$ contains the flag (not in standard form)
\begin{equation}\label{eq:helperflag}
(e_{\ell_{1}}\subset\dots\subset e_{\ell_{i-2}}\subset\underbrace{e_{\ell_{i-1}}+\hat{e}_{\ell_{i}}}_{i-1}\subset\underbrace{\he_{\ell_{i}}}_{i}\subset v_{i+1}\subset \dots\subset v_{j-1}\subset e_{n}\subset \dots\subset e_{\ell_{n-1}}).
\end{equation}
Now if $\ell_{i}>\ell_{i-1}$, then the flag in (\ref{eq:helperflag}) is $B_{n-1}$-conjugate to the flag 
in (\ref{eq:i-1cplx}) where the conjugation maps $e_{\ell_{i}}$ to $e_{\ell_{i}}-e_{\ell_{i-1}}$ and fixes all other vectors.  So $Q_{1}=Q_{2}$ and Lemma \ref{l:cplxvncpt} implies that the root $\alpha_{i-1}$ is complex stable for $Q$.  The flag $m(s_{\alpha_{i-1}})*\F$ is then given by (\ref{eq:i-1cplx}) by Proposition \ref{prop:monoid}.  
If, on the other hand, $\ell_{i}< \ell_{i-1}$, then the flag in (\ref{eq:helperflag}) is
$B_{n-1}$-conjugate to the flag in Equation (\ref{eq:i-1nc}), which is in standard form.  
It follows that $Q_{1}\neq Q_{2}$, so that $\alpha_{i-1}$ is non-compact for $Q$ by Lemma \ref{l:cplxvncpt}, and 
$m(s_{\alpha_{i-1}})*\F$ is given by Equation (\ref{eq:i-1nc}) by Proposition \ref{prop:monoid}.  The proof when $Q_{K}$ is closed is simpler and we leave the details to the reader.
\end{proof}

%


\noindent We now state the analogues of Proposition \ref{p:monoidactiononPILS} for 
types D and B.  The proofs are analogous to the type A case.   
The analogue of Lemma \ref{l:5.1} can be proven using Equation (\ref{eq:typeDflag}) in the type D case 
and Equation (\ref{eq:typeBflag}) in the type B case.  The assertions about the monoid action are proven using Lemma \ref{l:cplxvncpt} and computations with isotropic flags using Remark \ref{r:bvectororbit} in the type D case and Equation (\ref{eq:odd4.13one}) in the type B case.  The details are left to the reader.

Let $\fg=\fso(2\ell)$ and $\alpha\in\Pi_{\fg}$.  If $Q_{K}=Q_{+}$ is the unique closed $K$-orbit, 
then by Proposition 4.9 of \cite{CE21I}, $m(s_{\alpha})*Q_{+}\neq Q_{+}$ if and only if $\alpha=\alpha_{\ell-1},\, \alpha_{\ell}$.  If $Q_{K}=Q_{i}$ with $i=1,\dots, \ell-1$ is not closed, then the same result implies 
that $m(s_{\alpha})*Q_{i}\neq Q_{i}$ if and only if $\alpha=\alpha_{i-1}$.  

\begin{prop}\label{p:monoidactiononPILSD}
Let $\fg=\fso(2\ell)$ and $Q=B_{2\ell-1}\cdot \F$ with $Q\subset Q_{K}$.\\

\noindent Case I:  $Q_{K}=Q_{i}$, $i=1, \dots, \ell-1$ is not closed. 
In this case, $\F$ has the form:
$$
\F=(e_{j_{1}}\subset \dots \subset e_{j_{i-1}}\subset v_{i}\subset \dots\subset v_{\ell-1}),
$$
where $v_{i}=e_{\ell},\, e_{-\ell},\, \he_{j_{i}},\, \mbox{or } \te_{j_{i}}$.    
There are two cases to consider:
\begin{enumerate}
\item $v_{i}=e_{\pm\ell}$.
\begin{enumerate}
\item If $j_{i-1}>0$, then $\alpha_{i-1}$ is complex stable for $Q$, and 
\begin{equation}\label{eq:cplxDroot}
m(s_{\alpha_{i-1}})*\F=(e_{j_{1}}\subset \dots \subset e_{j_{i-2}}\subset\underbrace{e_{\pm \ell}}_{i-1}\subset e_{j_{i-1}} \subset v_{i+1}\subset \dots\subset v_{\ell-1}).
\end{equation}
\item If $j_{i-1}<0$, then $\alpha_{i-1}$ is non-compact for $Q$, and 
\begin{equation}\label{eq:noncptDroot}
m(s_{\alpha_{i-1}})*\F=(e_{j_{1}}\subset \dots \subset e_{j_{i-2}}\subset\underbrace{v_{i-1}}_{i-1}\subset v_{i} \subset v_{i+1}\subset \dots\subset v_{\ell-1}),
\end{equation}
where $v_{i-1}=\te_{|j_{i-1}|}$ if $v_{i}=e_{\ell}$ and $v_{i-1}=\he_{|j_{i-1}|}$ if $v_{i}=e_{-\ell}.$
\end{enumerate}
\item Let $v_{i}=\te_{j_{i}} \mbox{ or } \he_{j_{i}}$.   
\begin{enumerate}
\item If $j_{i-1}>0$ or $|j_{i-1}|>j_{i}$, then $\alpha_{i-1}$ is complex stable for $Q$ and $m(s_{\alpha_{i-1}})*\F$ is given by replacing $e_{\pm \ell}$ with $v_i$ in Equation (\ref{eq:cplxDroot}).
\item If $j_{i-1}<0$ and $|j_{i-1}|<j_{i}$, then $\alpha_{i-1}$ is non-compact and $m(s_{\alpha_{i-1}})*\F$ is given by (\ref{eq:noncptDroot}),
where  $v_{i-1}=\te_{|j_{i-1}|}$ if $v_{i}=\te_{j_{i}}$ and $v_{i-1}=\he_{|j_{i-1}|}$ if $v_{i}=\he_{j_{i}}.$
\end{enumerate}
\end{enumerate}
\noindent Case II: Suppose that $Q_{K}=Q_{+}$ is the unique closed $K$-orbit on $\B$. 
In this case, $\F$ is of the form
$$
\F=(e_{j_{1}}\subset e_{j_{2}}\subset\dots\subset e_{j_{\ell-1}}),
$$
where $j_{k}\neq \pm \ell$ for $k=1,\dots,\ell-1$. 

\begin{enumerate}
\item If $j_{\ell-1}>0$, then $\alpha_{\ell-1}$ and $\alpha_{\ell}$ are complex stable for $Q$.  However, their monoid actions are different.
We have 
\begin{equation}\label{eq:closedD1}
m(s_{\alpha_{\ell-1}})*\F=(e_{j_{1}}\subset \dots \subset e_{j_{\ell-2}}\subset e_{\ell})
\end{equation}
and 
\begin{equation}\label{eq:closedD2}
m(s_{\alpha_{\ell}})*\F=(e_{j_{1}}\subset\dots\subset e_{j_{\ell-2}}\subset e_{-\ell}).
\end{equation}
\item If $j_{\ell-1}<0$, then the roots $\alpha_{\ell-1}$ and $\alpha_{\ell}$ are both non-compact, and 
$$
m(s_{\alpha_{\ell-1}})*\F=m(s_{\alpha_{\ell}})*\F=(e_{j_{1}}\subset \dots\subset e_{j_{\ell-2}}\subset\he_{|j_{\ell-1}|} ).
$$
\end{enumerate}

\end{prop}

Let $\fg=\fso(2\ell+1)$ and $\alpha\in\Pi_{\fg}$.  If $Q_{K}=Q_{+}$ or $Q_{-}$ is a closed $K$-orbit, 
then by Proposition 4.9 of \cite{CE21I}, $m(s_{\alpha})*Q_{K}\neq Q_{K}$ if and only if $\alpha=\alpha_{\ell}$.  If $Q_{K}=Q_{i}$ with $i=0,\dots, \ell-1$ is not closed, then the same result implies 
that $m(s_{\alpha})*Q_{i}\neq Q_{i}$ if and only if $\alpha=\alpha_{i}$.

\begin{prop}\label{p:monoidactiononPILSB}
Let $\fg=\fso(2\ell+1)$ and let $Q=B_{2\ell}\cdot \F$ with $Q\subset Q_{K}$.  

\noindent Case I: $Q_{K}=Q_{i}$, $i=0,\dots, \ell-1$ is not closed.  
In this case, $\F$ is of the form
\begin{equation}\label{eq:typeBflagi}
\F=(e_{j_{1}}\subset \dots \subset e_{j_{i}}\subset v_{i+1}\subset \dots\subset v_{\ell}),
\end{equation}
where $v_{i+1}=\he_{j_{i+1}}$ or $v_{i+1}=\he_{a,-j_{i+1}}$.  
\begin{enumerate}
\item If $j_{i}>0$ or $|j_{i}|> j_{i+1}$, then $\alpha_{i}$ is complex stable and $m(s_{\alpha_{i}})*\F$ is given by changing the sequence $e_{j_i} \subset v_{i+1}$
to $v_{i+1} \subset e_{j_i}.$
\item If $j_{i}<0$ and $|j_{i}|<j_{i+1}$, then $\alpha_{i}$ is non-compact and
$$
m(s_{\alpha_{i}})*\F=(e_{j_{1}}\subset\dots\subset e_{j_{i-1}}\subset\ v_{i}\subset v_{i+1}\subset\dots\subset v_{\ell}),
$$
\end{enumerate}
where $v_{i}=\he_{j_{i+1}, \,j_{i}}$ if $v_{i+1}=\he_{j_{i+1}}$, and $v_{i}=\he_{a,\,j_{i}}$ if $v_{i+1}=\he_{a,-j_{i+1}}$. 

\noindent Case II: $Q_{K}=Q_{+}\mbox{ or } Q_{-}$ is closed.  
In this case, 
$$
\F=(e_{j_{1}}\subset \dots \subset e_{j_{\ell}}).
$$
consists entirely of standard basis vectors.  For $Q\subset Q_{K}$, the root $\alpha_{\ell}$ is non-compact for $Q$ and 
\begin{equation}\label{eq:closedB}
m(s_{\alpha_{\ell}})*\F= (e_{j_{1}}\subset \dots \subset e_{j_{\ell-1}}\subset \he_{|j_{\ell}|}).
\end{equation}

\end{prop}
\noindent We note that the computation of Remark 2.11 of \cite{CE21I} is useful in the proof of (\ref{eq:closedB}).


To fully describe the monoid action by simple roots of $\fg$ it remains to compute $m(s_{\alpha})*Q$ when 
$Q\subset Q_{K}$ and $\alpha$ is either real or complex unstable for $Q_{K}$.  This can be done 
using Proposition \ref{prop:monoid} along with Propositions \ref{p:monoidactiononPILS}, \ref{p:monoidactiononPILSD}, and \ref{p:monoidactiononPILSB}.  We omit the calculations because we will not require these monoid actions in this paper.  
\section{The Monoid action and the closure ordering on $B_{n-1}\backslash\B_{n}$}\label{s:closure}

In this section, we prove one of the main results of the paper which states that the closure ordering coincides with a certain standard order defined by Richardson and Springer in \cite{RS}.   This result enables us to determine the closure ordering in terms of a monoid action using simple roots both of $\fg$ and $\fk$.

\subsection{The monoid action and the standard order}\label{ss:RS}

Recall from Section 5.1 the action of a subgroup $M$ of a group $R$ on the flag variety $\B:=\B_{R}$ with finitely many orbits and the monoid action by a simple reflection $s\in S$ from the Weyl group of $R$ on $M\backslash \B$.  Given a sequence $\vec{s}=(s_{1},\dots, s_{k})$ of elements in $S$ and 
an $M$-orbit $Q\in\Morbitspace$, we let 
\begin{equation} \label{eq:Ssequence}
m(\vec{s})*Q:=m(s_{k})*\dots *m(s_{1})*Q \mbox{ if } k > 0, m(\vec{s})*Q=Q \mbox{ if } k = 0.
\end{equation}
 Then $\mathfrak{M}:=\{ m(\vec{s}): \vec{s} \mbox{ a sequence} \} $ is a finite monoid with $1$, which follows from the well-known fact that the operators $m(s)$ satisfy both braid relations and $m(s)^{2}=m(s)$.

The \emph{weak order } $\leq_w$ is defined by the property that if $Q, Q^{\prime} \in M\backslash \B$, then $Q \leq_w Q^{\prime}$ if and only if $Q^{\prime}=m(\vec{s})*Q$ for some sequence $\vec{s}$ as above.  It is the weakest partial order such that $Q$ is less than or equal to $m(s)*Q$ for each $s\in S$ and $Q \in M\backslash \B.$

\begin{dfn}\label{dfn:minimal}
We say $Q\in\Morbitspace$ is minimal if $Q\in \mathfrak{M}*Q^{\prime}$ implies 
$Q=Q^{\prime}$ for every $Q^{\prime}\in\Morbitspace$. 
\end{dfn}
The minimal orbits are the minimal elements of $\Morbitspace$ in the weak order. We need the following assumption about minimal elements in $\Morbitspace.$
\begin{equation}\label{eq:minimalcond}
\mbox{The minimal elements } Q \in \Morbitspace \mbox{ are exactly the orbits of minimal dimension } d.
\end{equation}

Given the assumption in (\ref{eq:minimalcond}), we define a length function $\ell: \Morbitspace\to \mathbb{Z}_{\geq 0}$ by 
\begin{equation}\label{eq:lengthfn}
\ell(Q):=\dim Q-d
\end{equation}
so the minimal orbits are exactly the orbits of length $0.$

There are two other partial orderings that we consider on $\Morbitspace$ called the closure order and the standard order.
The closure order (sometimes called the Bruhat order) is defined by 
$$
Q^{\prime}\leq Q\mbox{ if and only if } Q^{\prime}\subset \overline{Q}.
$$
The standard order is denoted by $\preceq$ and is defined in stages.  
We first define a relation $\dashv$ on $\Morbitspace$ as follows.   
Let $Q,\, Q^{\prime}\in\Morbitspace$.  We say that $Q^{\prime}\dashv Q$ if 
there exists a sequence $\vec{s}=(s_{1},\dots, s_{k})$ of elements of $S$, $t\in S$, and 
$Q^{\prime\prime}\in \Morbitspace$ such that the following two conditions hold:
\begin{equation}\label{eq:stdconditions}
\begin{split}
&(i)\; Q^{\prime}=m(\vec{s})*Q^{\prime\prime} \mbox{ and } \ell(Q^{\prime})=\ell(Q^{\prime\prime})+k.\\
&(ii) \; Q=m(\vec{s})*(m(t)*Q^{\prime\prime}) \mbox{ and } \ell(Q)=\ell(Q^{\prime\prime})+k+1.
\end{split}
\end{equation}
\noindent The conditions in (\ref{eq:stdconditions}) can be easily visualized in the  following diamond diagram:  
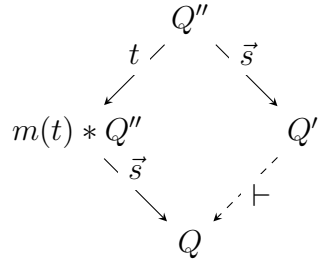
\begin{center}
 \begin{tikzpicture} [scale=1.5,auto=center,every node/.style={rectangle,fill=white!20}]
  \node (a1) at  (0,5) {$Q^{\prime\prime}$}; \node (a2) at (-1,4) {$m(t)*Q^{\prime\prime}$}; \node (a3) at (1,4) {$Q^{\prime} $}; 
 \node (a4) at (0,3) {$ Q$};  \draw [-stealth] (a1) -- (a2) node[midway, above] {$t$}; \draw [-stealth] (a1) -- (a3) node[midway, above] {$\vec{s}$}; \draw [-stealth] (a2) -- (a4) node[midway, above] {$\vec{s}$}; \draw [-stealth] [dashed](a3) -- (a4) node[near start, below] {$\vdash$}; 
  \end{tikzpicture} \end{center}

The standard order is defined by the property that
$Q^{\prime}\preceq Q$ if there exists a sequence $\vec{Q}=(Q_{0},\dots, Q_{k})$, $k\geq 0$, 
with $Q_{0}=Q^{\prime}$ and $Q_{k}=Q$ such that $Q_{i-1}\dashv Q_{i}$ for all $i=1,\dots, k$.

Assuming Equation (\ref{eq:minimalcond}), a partial order $\leq_O$ is said to be compatible with the  
action of the monoid $\mathfrak{M}$ on $\Morbitspace$ if for all $s\in S$ and $Q, Q^{\prime} \in \Morbitspace$, we have
\begin{equation}\label{eq:compatible}
\begin{split}
&\; (1)\; Q\leq_{O} m(s)*Q.\\
&\; (2) \mbox{ If } Q^{\prime}\leq_{O} Q, \mbox{ then } m(s)*Q^{\prime}\leq_{O} m(s) *Q.\\ 
& \;(3) \mbox{ If } Q^{\prime}\leq_{O} Q \mbox{ and } \ell(Q^{\prime})\geq \ell(Q),\, \mbox{then } Q^{\prime}=Q.
\end{split}
\end{equation}

It follows easily from definitions that the closure order is compatible with the $\mathfrak{M}$-action.

\begin{thm}\label{thm:closureorder}[Theorem 7.11 in \cite{RS}]
Under the assumption in (\ref{eq:minimalcond}) the closure order on 
$\Morbitspace$ coincides with the standard order. 
\end{thm} 

\begin{rem}\label{r:rsthm711}
To be precise, this theorem is only proved in \cite{RS} in the case where $M$ is the fixed points $H$ of an algebraic involution of $R$.  However, the proof given in \cite{RS} applies in the more general situation of  Theorem \ref{thm:closureorder}.  In more detail, Richardson and Springer introduce a property on a partial order on $\Morbitspace$ called the one-step property which in the setting of  Theorem \ref{thm:closureorder} means the following.  For a $M$-orbit $Q$ and $s\in S$ corresponding to a simple root $\alpha$, let $p(s,Q)=\{ Q^{\prime} \in M\backslash \B : Q^{\prime} \subset \pi_{\alpha}^{-1}(\pi_{\alpha}(Q))\}.$  Then the one-step property is said to hold if $$\overline{m(s)*Q}=\displaystyle\bigcup_{ Q^{\prime} \leq Q} p(s, Q^{\prime})$$ for all $s\in S$ and $Q\in \Morbitspace.$   The one-step property is shown to hold for the closure order when $M=H$ in Lemmas 7.5 and 7.6 of \cite{RS}.  The same proof applies in our more general context.   In Propositions 6.3, 6.4, and 6.5 and Lemma 6.1 of \cite{RS}, the authors prove that if we consider a  partial order compatible with the $\mathfrak{M}$-action and with a length function such that the minimal elements are precisely the elements of length $0$, then if the partial order satisfies the one-step property, then it coincides with the standard order.  This proves the theorem.
\end{rem}

\subsection{Applications to $\Borbitspace$}\label{ss:Bmonoid}

We require an extended monoid action to study $\Borbitspace.$  For this, let
$K_{\Delta}=\{ (x,x) \in K \times G \}$ and note that the map $\Borbitspace \to 
K_{\Delta}\backslash (\B_{\fk}\times  \B_{n})$ given by 
$Q\mapsto K_{\Delta}\cdot (\fb_{n-1}, Q)$
is bijective. Further,  $\dim(K_{\Delta}\cdot (\fb_{n-1}, Q))=\dim(\B_{\fk}) + Q$ and this bijection preserves the closure order.  If we let the group $R=K\times G$ in Section \ref{ss:monoid}, then we have a monoid action by simple roots from 
$\Pi_{\fk \oplus \fg} = \Pi_{\fk} \cup \Pi_{\fg}$ on the orbits in
$K_{\Delta}\backslash (\B_{\fk}\times  \B_{n}).$  We use the above bijection to obtain monoid actions of $\Pi_{\fk}$ and $\Pi_{\fg}$ on $\Borbitspace.$  We refer
to this monoid action by $\Pi_{\fk}$ and $\Pi_{\fg}$ as the \emph{extended monoid action}.   In particular, if $\alpha \in \Pi_{\fk}$, and $Q \in \Borbitspace$, then
$m(s_{\alpha})*Q$ is the $B_{n-1}$-orbit in $\B_n$ such that 
$K_{\Delta}\cdot (eB_{n-1}, m(s_{\alpha})*Q)$ is the orbit given by 
$m(s_{\alpha})*K_{\Delta} \cdot (eB_{n-1}, Q)$ where $\alpha$ is regarded as a simple root for $\fk\oplus \fg$.
We will prove in this section that the collection of $B_{n-1}$-orbits in $\B_n$ satisfies the minimal condition in Equation (\ref{eq:minimalcond}) for this extended monoid action.  On the other hand, even in the case of $GL(3)$, the $B_{n-1}$-orbits do not satisfy this minimal condition if we only use the simple roots of $\Pi_{\fg}$ (see Example \ref{exam:bigone} below).  

Let $Q=B_{n-1}\cdot \Ad(v)\F_{+}$ be an orbit in $\Borbitspace.$   Let $\alpha \in \Pi_{\fk}.$  By Equation (4.1) of \cite{CE21I}, 
\begin{equation}\label{eq:leftmonoid}
m(s_{\alpha})*Q \mbox{ is the unique open $B_{n-1}$-orbit in } P_{K, \alpha}\cdot \Ad(v) \F_{+},
\end{equation}
where $P_{K, \alpha}\supset B_{n-1}$ is the standard parabolic subgroup of $K$ determined by the root $\alpha$.  We refer to this monoid action as ``the left monoid action.''  Let $\alpha \in \Pi_{\fg}.$  By Equation (4.2) of \cite{CE21I}, 
\begin{equation}\label{eq:rightmonoidagain}
m(s_{\alpha})*Q\mbox{ is the unique open $B_{n-1}$-orbit in } 
B_{n-1}\cdot\Ad(v P_{\alpha})\F_{+},
\end{equation}
where $P_{\alpha}\supset B_{n}$ is the standard parabolic subgroup of $G$ determined by the root $\alpha$.   We refer to this monoid action as ``the right monoid action'' and note that it coincides with the monoid action by $m(s_{\alpha})$ on $\Borbitspace$ from Section \ref{ss:monoid}.

\begin{rem}\label{r:kexplicit}
An analogue of Proposition \ref{prop:monoid} also holds for the left monoid action except in this case  the simple reflection $s_{\alpha}$ and the Cayley transform $u_{\alpha}$ act on the left of the representative of the orbit $Q$ instead of on the right.  That is to say that if $Q=B_{n-1}\cdot \Ad(v)\F_{+}$ for $v\in G$ and $\alpha\in\Pi_{\fk}$ is complex stable (resp. non-compact) for $Q$ then $m(s_{\alpha})*Q=B_{n-1}\cdot \Ad(\dot{s}_{\alpha}v)\F_{+}$ (resp $m(s_{\alpha})*Q=B_{n-1}\cdot  \Ad(u_{\alpha}^{-1}v)\F_{+})$.

\end{rem}

\begin{rem}\label{r:kmodpright}
Let $P^{\prime}$ be a parabolic subgroup of $K$ containing $B_{n-1}.$
By a construction similar to the one above,  there is a left monoid action by roots in $\Pi_{\fk}$ on $B_{n-1}\backslash K/P^{\prime}$ (see Part (1) of Remark 4.5 from \cite{CE21I}).  It is determined by the property that for $\alpha \in \Pi_{\fk}$ and an orbit $Q \in B_{n-1}\backslash K/P^{\prime}$,
\begin{equation}\label{eq:monoidpartial}
m(s_{\alpha})*Q \mbox{ is the unique open  $B_{n-1}$-orbit in } P_{\alpha}\cdot Q.
\end{equation}
By fundamental properties of the Weyl group, the only minimal orbit for the monoid action by $\Pi_{\fk}$ on $B_{n-1}\backslash K/P^{\prime}$ is the one point orbit
$eP^{\prime}=B_{n-1}\cdot eP^{\prime}.$
\end{rem}

In the notation 
of Section \ref{ss:RS}, let $S:=\{s_{\alpha}:\, \alpha\in \Pi_{\fk}\cup\Pi_{\fg}\},$ and let $\mathfrak{M}=\{ m(\vec{s})|\, \vec{s}=(s_{1},\dots, s_{k}),\,s_{i}\in S,\, k\geq 0\}$.  
Then $\mathfrak{M}$ acts on $\Borbitspace$ via the extended monoid action described above.  
It then follows that Theorem \ref{thm:closureorder} holds 
for $\Borbitspace$ as long as (\ref{eq:minimalcond}) holds for the action of the monoid $\mathfrak{M}$.  To that effect, we prove

\begin{thm}\label{thm:bigthm}
The minimal elements for the action of the monoid $\mathfrak{M}$ on $\Borbitspace$ 
(see Definition \ref{dfn:minimal}) are precisely the zero dimensional $B_{n-1}$-orbits.  
\end{thm}

We  review some results about the fibre bundle structure of 
 the $B_{n-1}$-orbits on $\B_{n}$ developed in \cite{CE21I} which will be needed in the proof of this Theorem, beginning with results concerning the $K$-orbits $Q_K.$  Let $Q\in\Borbitspace$ and let $Q_K = K\cdot Q$ be the $K$-orbit containing $Q$.  
In Theorem 3.1 and Remark 3.3 of \cite{CE21I}, we 
associate to the $K$-orbit $Q_{K}$ a parabolic subgroup $R\subset G$ with a Levi factor
$L$ containing the standard Cartan subgroup of diagonal matrices $H$.  
In more detail, we denote by $\fp_{S}\supset\fb_{n}$ the standard 
parabolic subalgebra of $\fg$ given by a subset $S\subset\Pi_{\fg}$ of standard simple roots.  Recall the explicit description of $K$-orbits on $\B_n$ given in  Section \ref{ss:Korbits}.  
Then $\fr=\mbox{Lie}(R)=\tilde{w}(\fp_{S})$ where $S\subset \Pi_{\fg}$ and $\tilde{w}\in W$  depend on $Q_K$ according to the formulas
\begin{equation}\label{eq:rootsforLevi}
S=\left\lbrace\begin{array}{lll}  \{\alpha_{i},\dots, \alpha_{j-1}\} & \mbox{ for } \fg=\fgl(n), & Q_{K}=Q_{i,j},\, i<j\\
\{\alpha_{i+1},\dots, \alpha_{\ell}\} & \mbox{ for } \fg=\fso(2\ell+1), & Q_{K}=Q_{i}\\ 
\{\alpha_{i},\dots , \alpha_{\ell}\} & \mbox{ for } \fg=\fso(2\ell), & Q_{K}=Q_{i}\\
\emptyset  &\mbox{ for } \fg=\fgl(n) \mbox{ or }\fso(n),  &Q_{K}\mbox{ closed} \end{array}\right\rbrace_{,}
\end{equation} 
and
\begin{equation}\label{eq:guyinW}
\tilde{w}=\left\lbrace\begin{array}{llll} w_{i}=(n, n-1, \dots , i+1, i) & \mbox{ if } \fg=\fgl(n),   &\mbox { and } &Q_{K}=Q_{i,j} \mbox{ or } Q_{i}\\
\mbox{ e }&  \mbox{ if } \fg=\fso(n),  &\mbox{ and }& Q_{K}=Q_{i} \mbox{ or } \,Q_{+}\\ 
 s_{\alpha_{l}}&  \mbox{ if } \fg=\fso(2\ell+1),& \mbox{ and }  &Q_{K}=Q_{-}\end{array}\right\rbrace_{.}
 \end{equation}
To describe a Levi factor $\fl$ of $\fr$, consider the index 
\begin{equation}\label{eq:km}
k(m):=\left\lbrace\begin{array}{ll}  m & \mbox{if }  \fg=\fgl(n)\\
2m & \mbox{if } \fg=\fso(2\ell+1)\\
2m+1 & \mbox{if } \fg=\fso(2\ell)\end{array}\right.
\end{equation}
Let $\fg_{k(m)}=\fgl(k(m))$ or $ \fso(k(m))$ in the general linear and orthogonal cases respectively.  Then $\fl\cong \fg_{k(m)+1}+ \fz(\fl),$ 
where $\fz(\fl)$ is the centre of $\fl$.
In addition, the group $K\cap R$ is a standard parabolic subgroup of $K$ with Levi factor $K\cap L=G_{k(m)}\cdot Z$ where $Z$ is the connected group with Lie algebra $\fk\cap\fz(\fl)$.  Under these identifications, in the $GL(n)$-case,  $G_{k(m)+1}$ is identified as the group of all invertible linear transformations of the space spanned by $\{e_{i}, \dots, e_{j-1}, e_{n}\}$.  
In the $SO(2\ell+1)$-case (resp. $SO(2\ell)$-case), $G_{k(m)+1}$ is identified with the group of 
all orthogonal transformations of determinant one of the space spanned by $\{ e_{\pm (i+1)}, \dots, e_{\pm \ell}, e_0 \}$ (resp. $\{ e_{\pm i}, \dots, e_{\pm \ell}\}$) with respect to the restriction of the form $\beta$ defined in Equation (\ref{eq:beta}).  Then the subgroup $G_{k(m)}=G_{k(m)+1}\cap K$ is embedded in $G_{k(m)+1}$ as in Section \ref{ss:Korbits}.  Note that $\B_{\fl}\cong \B_{k(m)+1}$ and $\B_{\fl\cap\fk}\cong \B_{k(m)}$ and $B_{n-1}\cap L$ is a Borel subgroup of $K\cap L$. Thus, the $K\cap L$-action and $B_{n-1}\cap L$-actions on $\B_{\fl}$ may be identified with the $G_{k(m)}$ and $B_{k(m)}$-actions on $\B_{k(m) + 1}$ and we use these identifications freely in the future.   We let $\tilde{Q}_{K\cap L}$ denote the open $K\cap L$-orbit on $\B_{\fl}$, which we identify with the open $G_{k(m)}$-orbit on $\B_{k(m)+1}.$  In Theorem 3.8 of \cite{CE21I}, we prove that there is a point ${\F}_{\fl}$ in $\tilde{Q}_{K\cap L}$ with stabilizer in $K\cap L$ corresponding to the upper triangular matrices $B_{k(m)-1}$ in $G_{k(m)-1},$ where,
in the $GL(n)$-case, $G_{k(m)-1}\subset G_{k(m)}$ is the subgroup preserving the subspace spanned by 
$\{e_{i},\dots, e_{j-2}\}$ and fixing the vectors $e_{j-1}$ and $e_{n}$, and in the orthogonal cases the embedding of $G_{k(m)-1}$ as a subgroup of 
$G_{k(m)}$ is described in Remark 2.1 of \emph{loc.cit.}  Explictly, in Equation (4.22) of \emph{loc.cit.}, we note that 
  \begin{equation}\label{eq:basepoint}
\F_{\fl}:=\left\lbrace\begin{array}{lll} (\hat{e}_{j-1}\subset e_{i}\subset e_{i+1}\subset\dots\subset e_{j-2}\subset e_{n})  & \mbox{ for } \fg=\fgl(n), & Q_{K}=Q_{i,j}, i<j\\
(\hat{e}_{\ell}\subset e_{i+1}\subset \dots\subset e_{\ell-1})& \mbox{ for } \fg=\fso(2\ell+1), & Q_{K}=Q_{i}\\ 
(e_{\ell}\subset e_{i}\subset\dots\subset e_{\ell-2}) & \mbox{ for } \fg=\fso(2\ell), & Q_{K}=Q_{i}\end{array}\right\rbrace_{,}
\end{equation}
so that $B_{k(m)-1}$ stabilizes the flag in $\fg_{k(m)-1}$ given by
\begin{equation*}
\F_{+,k(m)-1}:=\left\lbrace\begin{array}{lll} (e_{i}\subset e_{i+1}\subset \dots \subset e_{j-2}) &  \mbox{ for } \fg=\fgl(n), & Q_{K}=Q_{i,j}, i<j\\
(e_{i+1}\subset\dots\subset e_{\ell-1}) & \mbox{ for } \fg=\fso(2\ell+1), & Q_{K}=Q_{i}\\ 
(e_{i}\subset \dots\subset e_{\ell-2})& \mbox{ for } \fg=\fso(2\ell), & Q_{K}=Q_{i}\end{array}\right\rbrace_{.}
\end{equation*}

We now discuss the $B_{n-1}$-orbit structure.  By
Theorem 3.5 and Notation 3.6 of \cite{CE21I}, the canonical projection 
$\pi_{R}:\B_{n}\to G/R$ makes $Q$ into a $B_{n-1}$-equivariant fibre bundle, $Q=\mathcal{O}(Q_{\fr}, Q_{\fl})$ with base a $B_{n-1}$-orbit $Q_{\fr}=\pi_{R}(Q)\cap \fk$ on the partial flag variety $K/K\cap R$ of $\fk$ and fibre a $B_{n-1}\cap L$-orbit $Q_{\fl}$ contained in the above open orbit $\tilde{Q}_{K\cap L}$ on $\B_{\fl}.$   Since $B_{n-1}$ is a Borel subgroup of $K$, the orbits $Q_{\fr}$ are Schubert cells in $K/K\cap R$, and we denote them by $X_w^{K\cap R}=B_{n-1}w K\cap R$ in $K/K\cap R$.   To understand the orbits $Q_{\fl}$, given the above comments about the stabilizer of the point ${\F}_{\fl}$, we note that $B_{n-1}\cap L$-orbits on
$\tilde{Q}_{K\cap L}$ correspond to double cosets in $B_{k(m)}\backslash G_{k(m)}/B_{k(m)-1}$, and the map $B_{k(m)}\ell B_{k(m)-1} \mapsto B_{k(m)-1}\ell^{-1} B_{k(m)}$ is a bijection on double cosets.   In this correspondence,
\begin{equation}\label{eq:op}
 Q_{\fl}=B_{k(m)}\cdot \Ad(\ell) {\F}_{\fl} \Leftrightarrow Q_{\fl}^{op} := B_{k(m)-1}\cdot \Ad(\ell^{-1}) {\F}_{+, k(m)}, 
\end{equation}
 where 
${\F}_{+,k(m)}$ is the flag stabilized by the upper triangular matrices $B_{k(m)}$ in $G_{k(m)}$.  

\begin{rem}\label{r:Borelintersect}
We can make this correspondence more explicit as follows.  Let $Q=\mathcal{O}(X_w^{K\cap R},Q_{\fl})$.   The following assertions are restatements of Remark 3.7 of \cite{CE21I}.
\par\noindent Type A:  Consider the point ${\F}_{\fl} \in \B_{\fl}$ in the case whre $\fg=\fgl(n)$ where $Q_K=Q_{i,j}$ from Equation \eqref{eq:basepoint}.   Suppose that $Q_{\fl}=(B_{n-1}\cap L)\cdot \Ad(\ell){\F}_{\fl}$ for $\ell \in K\cap L.$  Then $Q=B_{n-1}\cdot \Ad(\dot{w}\ell){\F}_{i,j}$, where ${\F}_{i,j}$ is from Equation \eqref{eq:typeAflag}.
\par\noindent Type B:  Consider the point ${\F}_{\fl} \in \B_{\fl}$ in the case $\fg=\fso(2\ell + 1)$ and $Q_K=Q_i$ from Equation \eqref{eq:basepoint}.  Suppose that $Q_{\fl}=(B_{n-1}\cap L)\cdot \Ad(\ell){\F}_{\fl}$ for $\ell \in K\cap L.$  Then $Q=B_{n-1}\cdot \Ad(\dot{w}\ell){\F}_{i}$, where ${\F}_{i}$ is from Equation \eqref{eq:typeBflag}.
\par\noindent Type D:  Consider the point ${\F}_{\fl} \in \B_{\fl}$ in the case $\fg=\fso(2\ell)$ where $Q_K=Q_i$ from Equation \eqref{eq:basepoint}.  Suppose that $Q_{\fl}=(B_{n-1}\cap L)\cdot \Ad(\ell){\F}_{\fl}$ for $\ell \in K\cap L.$  Then $Q=B_{n-1}\cdot \Ad(\dot{w}\ell){\F}_{i}$, where ${\F}_{i}$ is from Equation \eqref{eq:typeDflag}.
\end{rem}

\begin{rem}\label{r:opencorrespondence}
When $Q_{K}$ is the open $K$-orbit on $\B_{n}$, the parabolic subalgebra $\fr$ coincides with $\fg$, so that $\fl=\fg$ and the partial flag variety $K/K\cap R$ is a point.  In this case, the orbit $\tilde{Q}_{K\cap L}=Q_{K}$ 
and any $Q\in B_{n-1}\backslash \tilde{Q}_{K}$ corresponds to a unique $Q^{op}\in B_{n-2}\backslash \B_{n-1}$.   
In the examples in Section \ref{s:graphs}, we will demonstrate how the Bruhat graph of $ B_{n-2}\backslash \B_{n-1}$ can be realized as a subgraph of the Bruhat graph of $B_{n-1}\backslash \B_{n}$ using this correspondence.
\end{rem}

To discuss the action of the monoid $\mathfrak{M}$ on orbits $\mathcal{O}(X_w^{K\cap R}, Q_{\fl})$, we recall the left monoid action by roots of
$\Pi_{\fk}$ on $B_{n-1}\backslash K/K\cap R$ from Remark \ref{r:kmodpright}.
Suppose $\alpha\in\Pi_{\fk}$ is a root of the Levi subalgebra $w(\fl\cap\fk).$
It follows by (\ref{eq:monoidpartial}) that
$m(s_{\alpha})*X_w^{K\cap R}=X_w^{K\cap R}$.  Hence,  if $\alpha\in\Pi_{\fk}$ is a root such that $m(s_{\alpha})*X_w^{K\cap R}\neq X_w^{K\cap R}$,
then $\alpha$ is not a root of $w(\fk\cap\fl)$.  Combining this observation with Equation (4.7) of Theorem 4.7 of \cite{CE21I} we obtain
\begin{equation}\label{eq:baseaction}
\mbox{If } m(s_{\alpha})*X_w^{K\cap R}\neq X_w^{K\cap R},\mbox{ then }  m(s_{\alpha})*\mathcal{O}(X_w^{K\cap R},Q_{\fl})=\mathcal{O}(m(s_{\alpha})*X_w^{K\cap R}, Q_{\fl}).
\end{equation}

We now study the extended monoid action on the orbits in the fibre.   In particular, we explain how the 
results of \cite{CE21I} may be applied to transfer monoid actions on the orbit $Q_{\fl}^{op}$ in  $B_{k(m)-1}\backslash \B_{k(m)}$ to monoid actions on the orbit $Q_{\fl}$ in $(B_{n-1}\cap L)\backslash\tilde{Q}_{K\cap L}$.  The set of orbits $B_{k(m)-1}\backslash\B_{k(m)}$ has a right monoid action via 
the simple roots of $\Pi_{\fg_{k(m)}}$ and a left monoid action via the simple roots
$\Pi_{\fg_{k(m)-1}}.$   Explicitly,    
\begin{equation}\label{eq:smallerroots}
\Pi_{\fg_{k(m)-1}}=\left\lbrace\begin{array}{lll}  \{\alpha_{i},\dots, \alpha_{j-3}\} &\mbox{ for } \fg=\fgl(n), & Q_{K}=Q_{i,j}, \, i<j\\
\{\alpha_{i+1},\dots , \alpha_{\ell-2}, \alpha_{\ell-1}\} &\mbox{ for } \fg=\fso(2\ell+1), & Q_{K}=Q_{i}\\ 
\{\alpha_{i},\dots, \alpha_{\ell-2}, \alpha_{\ell-1}\} &\mbox{ for } \fg=\fso(2\ell), & Q_{K}=Q_{i}\end{array}\right\rbrace_{.}
\end{equation}
Let $S^{\prime}\subset \Pi_{\fg}$ be the subset of simple roots 
\begin{equation}\label{eq:Sprime}
S^{\prime}:=\{\alpha_{k+1}: \, \alpha_{k}\in \Pi_{\fg_{k(m)-1}}\}.
\end{equation}
Then $S^{\prime}\subset S$ where $S$ is given in (\ref{eq:rootsforLevi}), and  
Proposition 4.9 of \cite{CE21I} implies that $S^{\prime}$ consists of precisely the roots in $S$ which are compact 
imaginary for the $K$-orbit $Q_{K}$ (see Part (3) of Proposition-Definition \ref{prop:monoid}.).
Then Equation (4.24) of \emph{loc.cit.} implies that
\begin{equation}\label{eq:transfer}
\mbox{for } \alpha_{k}\in\Pi_{\fg_{k(m)-1}},\, m(s_{\alpha_{k}})*Q_{\fl}^{op}=(m(s_{\tilde{w}(\alpha_{k+1})})*Q_{\fl})^{op},
\end{equation}
where $\tilde{w}$ is given by Equation (\ref{eq:guyinW}).  The monoid action 
on the left-hand side of (\ref{eq:transfer}) is a left monoid action and the monoid action on the right-hand
side is a right monoid action.  Note that on the right-hand side of
(\ref{eq:transfer}) we think of $\alpha_{k+1}\in S^{\prime}\subset S$, so that $\tilde{w}(\alpha_{k+1})\in \Pi_{\fl}$ and 
$\tilde{w}(\alpha_{k+1})$ is compact imaginary for the $K\cap L$-orbit $\tilde{Q}_{K\cap L}$, whence $m(s_{\tilde{w}(\alpha_{k+1})})*Q_{\fl}$ is a well-defined right monoid action on the space $(B_{n-1}\cap L)\backslash\tilde{Q}_{K\cap L}$.  By Equation (4.23) of \emph{loc. cit.} we know that 
\begin{equation}\label{eq:lefttoright}
\mbox{for } \alpha\in\Pi_{\fg_{k(m)}},\; m(s_{\alpha})*Q_{\fl}^{op}=(m(s_{\alpha})*Q_{\fl})^{op},
\end{equation}
 where the left-hand side is a right monoid 
action and the right-hand side is a left monoid action.

We will apply these observations in the following situation. Let $Q^{\prime}=\mathcal{O}(Q_{\fr}, Q_{\fl}^{\prime})$ and $Q=\mathcal{O}(Q_{\fr}, Q_{\fl})$ with 
$Q_{\fr}=X_{e}^{K\cap R}$ the unique zero dimensional $B_{n-1}$-orbit on $K/K\cap R$.  
 We claim that
\begin{equation}\label{eq:weakorderclaim}
(Q_{\fl}^{\prime})^{op}\leq_{w} Q_{\fl}^{op}\Rightarrow Q^{\prime}\leq_{w} Q.
\end{equation}
By definition of the weak order, it suffices to consider the case where 
$Q_{\fl}^{op}=m(s_{\alpha})*(Q_{\fl}^{\prime})^{op}$, where 
$\alpha\in \Pi_{\fg_{k(m)-1}}\cup\Pi_{\fg_{k(m)}}$.  
 First, consider $\alpha\in \Pi_{\fg_{k(m)}}$.  It follows from (\ref{eq:lefttoright}) that the orbit $Q_{\fl}=m(s_{\alpha})*Q_{\fl}^{\prime}$.  If we take $w=e$ in Equation (4.6) of Theorem 4.7 of \cite{CE21I}, we obtain 
 $$
 m(s_{\alpha})*Q^{\prime}=m(s_{\alpha})*\mathcal{O}(X_{e}^{K\cap R}, Q_{\fl}^{\prime})=\mathcal{O}(X_{e}^{K\cap R}, m(s_{\alpha})*Q_{\fl}^{\prime})=\mathcal{O}(X_{e}^{K\cap R}, Q_{\fl})=Q,
 $$
yielding (\ref{eq:weakorderclaim}) in this case.  For $\alpha=\alpha_{k}\in \Pi_{\fg_{k(m)-1}}$, Equation (\ref{eq:transfer}) implies $Q_{\fl}=m(s_{\tilde{w}(\alpha_{k+1})})*Q_{\fl}^{\prime}$.  Since $\alpha_{k+1}\in S^{\prime}\subset S$ where $S^{\prime}$ is given in (\ref{eq:Sprime}), we can apply Equation (4.14) of Part (1) of Theorem 4.11 of \cite{CE21I} to obtain
$$
m(s_{\alpha_{k+1}})*Q^{\prime}=m(s_{\alpha_{k+1}})*\mathcal{O}(X_{e}^{K\cap R}, Q_{\fl}^{\prime})=\mathcal{O}(X_{e}^{K\cap R}, m(s_{\tilde{w}(\alpha_{k+1})})*Q_{\fl}^{\prime})=\mathcal{O}(X_{e}^{K\cap R}, Q_{\fl})=Q,
$$
establishing (\ref{eq:weakorderclaim}) in this case as well.  
    

\begin{lem}\label{l:zerodim}
Each closed $K$-orbit in $\B_n$ has a unique $B_{n-1}$-fixed point, and every $B_{n-1}$-fixed point in $\B_n$ is contained in a unique closed $K$-orbit.
\end{lem}
\begin{proof}
Let $Q_K$ be a closed $K$-orbit.  By the Borel fixed-point theorem, the Borel
subgroup $B_{n-1}$ of $K$ has a unique fixed point on $Q_K.$  Conversely, any $B_{n-1}$-fixed point in $\B_n$ is in a closed $K$-orbit.
\end{proof}

\begin{rem}\label{r:numberclosed}
It follows from Lemma \ref{l:zerodim} and our description of $K\backslash\B_{n}$ in Section \ref{ss:Korbits} 
that there are $n$ zero dimensional $B_{n-1}$-orbits on $\B_{n}$ for $\fg=\fgl(n)$, two for $\fg=\fso(2\ell+1)$, 
and one for $\fg=\fso(2\ell)$.  By Theorem 4.10 of \cite{CEexp},  these orbits are all of the form $B_{n-1}\cdot w(\F_{+})$ where $w\in W$ is given as follows.  For $\fg=\fgl(n)$, $w$ is a representative of the cycle $w=(n, n-1, \dots, i)$ for some $i=1,\dots, n$.  For 
$\fg=\fso(2\ell+1)$, $w=e$ or $w=s_{\alpha_{\ell}}$, and for $\fg=\fso(2\ell)$, $w=e$.



\end{rem}


\begin{proof}[Proof of Theorem \ref{thm:bigthm}]
Let $Q$ be a $B_{n-1}$-orbit on $\B_n$.  We must show that there is a zero dimensional orbit $Q_0$ with $Q_0 \leq_{w} Q.$  We prove the result by induction on $n$.  In the initial cases $n=2$ when 
$\fg=\fgl(n)$ and $n=3$ when $\fg=\fso(n)$, the group $B_{n-1}=\C^{\times}$ and 
$\B_{n}=\mathbb{P}^{1}$ and the assertion is elementary.  For the inductive step,
we use a second induction on $m=\ell(Q_K)$ where $Q_K=K\cdot Q$.  In the case $m=0$, $Q_K$ is closed, and by Equation (\ref{eq:rootsforLevi}), $K\cap L$ is abelian, $Q_{\fl}$ is a point, $Q\cong Q_{\fr}$ and $K/K\cap R \cong \B_{\fk}.$  By Theorem 4.7 of \cite{CE21I}, it follows that for $\alpha\in\Pi_{\fk}$, 
\begin{equation}\label{eq:leftequiv}
m(s_{\alpha})*Q\cong m(s_{\alpha})*Q_{\fr},
\end{equation}
where the action on the right-hand side is the action of the monoid $\mathfrak{M}_{\fk}$ generated by simple roots of $\fk$  on $B_{n-1}\backslash \B_{\fk}$ in (\ref{eq:monoidpartial}).  By Remark \ref{r:kmodpright}, there is a unique minimal element in $B_{n-1}\backslash \B_{\fk}$ for the $\mathfrak{M}_{\fk}$-action, and the minimal orbit is the zero dimensional orbit corresponding to $\fb_{n-1}.$  It follows easily that this orbit is minimal for the action of the full monoid $\mathfrak{M}.$

Now suppose that $m>0$ and let $Q=\mathcal{O}(Q_{\fr},Q_{\fl}).$  By Remark 
\ref{r:kmodpright}, there is a sequence $\vec{s}=(s_{\alpha_{i_{1}}}, \dots, s_{\alpha_{i_{k}}})$ with 
$\alpha_{i_{j}}\in\Pi_{\fk}$ and a sequence of distinct
$B_{n-1}$-orbits in $K/K\cap R$, $\vec{Q}_{\fr}=(Q_{0, \fr},\dots, Q_{k,\fr})$ with 
$Q_{0,\fr}= X_e^{K\cap R}$, $Q_{k,\fr}=Q_{\fr}$ such that $m(s_{\alpha_{i_{j}}})*Q_{j-1, \fr}=Q_{j,\fr}$, so that $m(\vec{s})*(X_{e}^{K\cap R})=Q_{\fr}$ (see Equation \eqref{eq:Ssequence}).   Applying (\ref{eq:baseaction}) iteratively, we obtain
$$
m(\vec{s})*\mathcal{O}(X_e^{K\cap R}, Q_{\fl})=\mathcal{O}(m(\vec{s})*(X_e^{K\cap R}), Q_{\fl})=\mathcal{O}(Q_{\fr}, Q_{\fl}),
$$
so $\mathcal{O}(X_e^{K\cap R}, Q_{\fl})\leq_{w} Q$. 
  By our above remarks, $Q_{\fl}$ corresponds to a $B_{k(m)-1}$-orbit $Q_{\fl}^{op}$  on $\B_{k(m)}$ where $k(m)$ is given in (\ref{eq:km}).  By induction on $n$, there exists 
a zero dimensional $B_{k(m)-1}$-orbit $\{\fb_{k(m)}\}\subset \B_{k(m)}$ with 
$\{\fb_{k(m)}\}\leq_{w} Q_{\fl}^{op}$.  Let $Q_{\fl}^{\prime}$ be the unique $B_{k(m)}$-orbit in $\tilde{Q}_{K\cap L}$ such that $(Q_{\fl}^{\prime})^{op}=\fb_{k(m)}$ and let $Q^{\prime}=\mathcal{O}(X_e^{K\cap R},Q_{\fl}^{\prime}).$
Equation (\ref{eq:weakorderclaim}) now implies 
$\mathcal{O}(X_e^{K\cap R}, Q_{\fl}^{\prime})\leq_{w}\mathcal{O}(X_e^{K\cap R}, Q_{\fl})$, whence $Q^{\prime}\leq_{w} Q$.


Thus, it suffices to prove the Theorem for orbits $Q=\mathcal{O}(X_{e}^{K\cap R}, Q_{\fl})$ with $\dim Q_{\fl}^{op}=0$.  For this, we claim that there is a $B_{n-1}$-orbit $Q^{\prime}$ such that $Q^{\prime}\leq_{w} Q$ where $Q^{\prime}\subset Q_{K}^{\prime}$ with 
$\ell(Q_{K}^{\prime})<m$.  
It then follows by induction on $m$ that there exists a Borel subalgebra 
$\fb\in\B_{n}$ with $\fb_{n-1}\subset \fb$ and $\{\fb\}\leq_{w} Q^{\prime}$, whence 
$\{\fb\}\leq_{w} Q$ and the proof is complete.  The $B_{n-1}$-orbit $Q^{\prime}$ is constructed 
on a case-by-case basis using the results of Section \ref{s:monoid}. 

The first step is to describe the flag in standard form $\mathcal{F}$ contained 
in an orbit $Q=\mathcal{O}(X_{e}^{K\cap R}, Q_{\fl})$ with $\dim Q_{\fl}^{op}=0$.  
We begin with the case $\fg=\fgl(n)$.  
Since $m=\ell(Q_K)>0$, then by Remark \ref{r:Korbitlength}, the orbit
 $Q_{K}=Q_{i,j}$ with $i<j$.  
We claim that $Q=B_{n-1}\cdot \F$, where  
\begin{equation}\label{eq:closedflag}
\F=(e_{1}\subset \dots\subset e_{i-1}\subset\underbrace{\he_{k}}_{i}\subset e_{i}\subset \dots\subset e_{k-1} \subset e_{k+1}\subset\dots\subset e_{j-1}\subset\underbrace{e_{n}}_{j}\subset e_{j}\subset\dots\subset e_{n-1})
\end{equation}
for some $k=i,\dots, j-1.$  Recall in this case the index $k(m)=m$.  The orbit $Q_{\fl}=(B_{n-1}\cap L)\cdot \Ad(\ell)\F_{\fl}$ 
for some $\ell\in K\cap L$ and where $\F_{\fl}$ is the flag in Equation (\ref{eq:basepoint}).  Since $Q=\mathcal{O}(X_e^{K\cap R}, Q_{\fl})$, then by Remark \ref{r:Borelintersect}, the orbit $Q=B_{n-1}\cdot \Ad(\ell) \F_{i,j}$ with $\F_{i,j}$ the flag in Equation (\ref{eq:typeAflag}).  Then by Equation \eqref{eq:op}, the orbit $Q_{\fl}^{op}=B_{m-1}\cdot \Ad(\ell^{-1}){\F}_{+,m}.$  By Lemma \ref{l:zerodim} and Remark \ref{r:numberclosed}, we may assume that $\ell^{-1}$ is a representative for the cycle $ (j-1, j-2,\dots, k+1, k)$ for 
some $k\in \{ i, \dots, j-1 \}.$  But then $Q = B_{n-1}\cdot \Ad(\ell){\F}_{i,j}$, so by Equation (\ref{eq:typeAflag}), the orbit $Q$ contains the flag in Equation \eqref{eq:closedflag}.  By Theorem \ref{thm:std}, this flag $\F$ is the unique flag in standard form in the orbit $Q$.

Using this result, we now
use our results in Section \ref{s:monoid} to construct a $B_{n-1}$-orbit $Q^{\prime}\subset Q_{K}^{\prime}$
with $\ell(Q_{K}^{\prime})<m$ and $m(s_{\alpha})*Q^{\prime}=Q$ for some $\alpha\in \Pi_{\fg}$.  
 First, suppose that $m>1$, so that $j>i+1$.  There are two subcases to consider.  

\noindent Case I:  $i\leq k<j-1$:  Let $Q^{\prime}:=B_{n-1}\cdot \F^{\prime}$, where 
$$
\mathcal{F}^{\prime}=(e_{1}\subset \dots\subset e_{i-1}\subset\underbrace{\he_{k}}_{i}\subset e_{i}\subset \dots\subset e_{k-1} \subset e_{k+1}\subset\dots\subset\underbrace{e_{n}}_{j-1}\subset e_{j-1}\subset e_{j}\subset\dots\subset e_{n-1}).
$$
By Lemma \ref{l:5.1}, the orbit
 $Q_{K}^{\prime}=Q_{i,j-1}$, so $\ell(Q_{K}^{\prime})=m-1$.  Further, it follows from Proposition \ref{p:monoidactiononPILS} that $\alpha_{j-1}$ is complex stable for $Q^{\prime}$ and $m(s_{\alpha_{j-1}})*Q^{\prime}=Q$ (see Equation (\ref{eq:jcplx})).    

\noindent Case II: $k=j-1$.  In this case, 
we let $Q^{\prime}=B_{n-1}\cdot \F^{\prime}$, where 
$$
\mathcal{F}^{\prime}=(e_{1}\subset \dots\subset e_{i-1}\subset e_{i}\subset\underbrace{\he_{j-1}}_{i+1}\subset e_{i+1}\subset \dots\subset e_{j-2}\subset\underbrace{e_{n}}_{j}\subset e_{j}\subset \dots\subset e_{n-1}).
$$
Again by Lemma \ref{l:5.1}, the orbit $Q_{K}^{\prime}=Q_{i+1,j}$.   Since $i<j-1$, Proposition \ref{p:monoidactiononPILS} implies $\alpha_{i}$ is complex stable for $Q^{\prime}$ and $m(s_{\alpha_{i}})*Q^{\prime}=Q$ (see Equation \eqref{eq:i-1cplx}).

Lastly, suppose $m=1$, so that $j=i+1$.  Then  
$$
\F=(e_{1}\subset \dots\subset e_{i-1}\subset\underbrace{\he_{i}}_{i}\subset e_{n}\subset e_{i+1}\subset\dots\subset e_{n-1}).  
$$
Let $Q^{\prime}=B_{n-1}\cdot\F^{\prime}$ where
$$
\mathcal{F}^{\prime}=(e_{1}\subset \dots \subset e_{i-1}\subset e_{n}\subset e_{i}\subset \dots\subset e_{n-1}).  
$$
Then $Q^{\prime}$ is closed and $Q^{\prime}\subset Q_{K}^{\prime}$, where 
$Q_{K}^{\prime}=Q_{i}$ is a closed $K$-orbit on $\B$.  Case II of Proposition \ref{p:monoidactiononPILS} implies that $\alpha_{i}$ is non-compact for $Q$ and 
$m(s_{\alpha_{i}})* Q^{\prime}=Q$. 
This establishes the claim in each of the three cases, which proves the Theorem for $\fg=\fgl(n).$


Now we consider the case $\fg=\fso(2\ell)$.  Since $m=\ell(Q_K)>0$, then by Remark \ref{r:Korbitlength}, the $K$-orbit $Q_K = Q_i$ with $i \in \{ 1, \dots, \ell - 1 \}.$  We first consider the case when $m > 1$
 so that $Q_{K}=Q_{i}$ with 
$i<\ell-1$.  In this case, we claim that our $B_{2\ell-1}$-orbit $Q$ in $Q_K$ contains one of the two flags
\begin{equation}\label{eq:closedflagD}
\F_{\pm, i}:=(e_{1}\subset\dots\subset e_{i-1}\subset e_{\pm \ell} \subset e_{i} \subset \dots \subset e_{\ell-2}).  
\end{equation}
Indeed, $Q=\mathcal{O}(X_e^{K\cap R}, Q_{\fl})$ and $k(m)=2m+1$, so that our
$B_{2m+1}$-orbit $Q_{\fl}$ corresponds to a zero dimensional $B_{2m}$-orbit $Q_{\fl}^{op}$ in $\B_{2m+1}.$  By Lemma \ref{l:zerodim} and Remark \ref{r:numberclosed}, there are exactly two such orbits.  If we let $\F_{+,2m+1}$ be the flag in $\B_{2m+1}$ stabilized by upper triangular matrices in $G_{2m+1}\cong SO(2m+1)$, then the two zero dimensional orbits are the points $\F_{+,2m+1}$ and $s_{\alpha_{\ell}}s_{\alpha_{\ell-1}} (\F_{+,2m+1})$ (note that $s_{\alpha_{\ell}}s_{\alpha_{\ell-1}}\in W_{K\cap L}\cong W_{G_{2m+1}}$ by Proposition 2.22 of \cite{CEspherical} and is a representative of the simple reflection determined by the short simple root of $\fg_{2m+1}= \fg_{2m+2}\cap \fk$).  It now follows as in the $\fgl(n)$-case using Remark \ref{r:Borelintersect}, Equation \eqref{eq:op}, and Equation \eqref{eq:typeDflag} that our orbit $Q$ contains one of the two standard flags ${\F}_{\pm, i}.$

To complete the proof for $\fg=\fso(2\ell)$ when $m > 1$, consider the $SO(2\ell)$-standard flag
$$
\F^{\prime}=(e_{1}\subset\dots\subset e_{i-1}\subset e_{i} \subset e_{\pm \ell} \subset \dots \subset e_{\ell-2}).
$$
Then it follows from Proposition \ref{p:monoidactiononPILSD} that $\alpha_{i}$ is complex stable 
for $Q^{\prime}=B_{2\ell-1}\cdot \F^{\prime}\subset Q_{i+1}$ and $m(s_{\alpha_{i}})*Q^{\prime}=Q$ (see Equation (\ref{eq:cplxDroot})).  

Now suppose that $m=1$, so that $Q_{K}=Q_{\ell-1}$.  In this case, by reasoning as in the previously discussed cases, we can prove that our $B_{2\ell-1}$-orbit $Q=\mathcal{O}(X_e^{K\cap R},Q_{\fl})$ contains exactly one of the standard flags
\begin{equation}\label{eq:length1Dflag}
\F_{\pm}=(e_{1}\subset\dots\subset e_{i} \subset \dots \subset e_{\ell-2}\subset e_{\pm \ell}).  
\end{equation}
Now we take $\F^{\prime}=\F_{+}$.  By Proposition \ref{p:monoidactiononPILSD}, if the last vector in (\ref{eq:length1Dflag}) is $e_{\ell}$,
then $m(s_{\alpha_{\ell-1}})*\F_{+}=\F$ (by Equation (\ref{eq:closedD1})), and if the last vector is $e_{-\ell}$, then $m(s_{\alpha_{\ell}})*\F_{+}=\F$ (by Equation (\ref{eq:closedD2})).

Finally, we consider the  case $\fg=\fso(2\ell+1)$, which turns out to be the easiest of the three cases.  
First, note that if $m>1$, then $Q_{K}=Q_{i}$ with $i=0,\dots,\ell-2$.  In this case, similar analysis shows that our $B_{2\ell}$-orbit $Q=\mathcal{O}(X_e^{K\cap R}, Q_{\fl})$ contains the standard flag
\begin{equation}\label{eq:closedflagB}
\F=(e_{1}\subset\dots\subset e_{i}\subset\underbrace{ \he_{\ell} }_{i+1}\subset e_{i+1}\subset\dots\subset e_{\ell-1}).
\end{equation}
The key point is that there is exactly one closed $B_{2m-1}$-orbit on $\B_{2m}$ for any $m$ by Lemma \ref{l:zerodim} and Remark \ref{r:numberclosed}.  We take as our $\F^{\prime}$ the flag
$$
\F^{\prime}=(e_{1}\subset\dots\subset e_{i}\subset e_{i+1}\subset \he_{\ell}\subset\dots \subset e_{\ell-1}) 
$$
and again let $Q^{\prime}=B_{2\ell}\cdot \F^{\prime}.$
Then $Q^{\prime}\subset Q_{i+1}$ and by Proposition \ref{p:monoidactiononPILSB}, $m(s_{\alpha_{i}})*Q^{\prime}=Q$.  
If $m=1$, i.e. $Q_{K}=Q_{\ell-1}$, then we take $\F^{\prime}=\F_{+}$.  In this case, $m(s_{\alpha_{\ell}})*Q^{\prime}=Q$ by Equation (\ref{eq:closedB}).

\end{proof} 

By Theorem \ref{thm:closureorder}, Theorem \ref{thm:bigthm} has the following consequence.

\begin{cor}\label{c:bruhatisstandard}
The Bruhat order on $B_{n-1}\backslash \B_n$ coincides with the standard action for the extended monoid action by simple roots from $\fk$ and $\fg$.
\end{cor}

\newpage

\section{Examples of Bruhat graphs}\label{s:graphs}
We present the Bruhat graphs for $B_{\fgl(2)}\backslash\B_{\fgl(3)}$, $B_{\fso(3)}\backslash \B_{\fso(4)}$, and $B_{\fso(4)}\backslash \B_{\fso(5)}$.  Each orbit 
is represented by the unique flag in standard form contained in the orbit.  The 
first row of orbits are zero dimensional, the second row of orbits are the one dimensional orbits, and so 
on.  A blue line connecting two orbits indicates
that the lower orbit is obtained from the upper orbit via the monoid action of the specified root and 
that the given root is complex stable for the upper orbit.  
A red line, on the other hand, indicates that the given root is non-compact for the upper orbit.  
A solid red or blue line denotes a right monoid action by a simple root of $\fg$ as in (\ref{eq:rightmonoidagain}) and
a dashed line denotes a left monoid action by a simple root of $\fk$ as in (\ref{eq:leftmonoid}).  When $\mbox{rank}(\fk)=1$, we do not label the dashed line.  Note that we do not exhibit all the non-trivial monoid actions, but we instead exhibit a path from a zero dimensional orbit to any orbit.  In the case where a monoid action via either a simple root of $\fk$ or $\fg$ can be used to move between orbits, the dashed line is omitted and the root of $\fk$ does not appear as a label except to illustrate Equations (\ref{eq:transfer}) and (\ref{eq:lefttoright}) in the correspondence between $B_{n-1}$-orbits contained in the open $K$-orbit on $\B_{n}$ and $B_{n-2}\backslash \B_{n-1}$ described in Remark \ref{r:opencorrespondence}.  Finally, a solid green line indicates that the two orbits are related in the standard order, but not the weak order.  The explicit computations of the monoid actions can be performed using Propositions \ref{p:monoidactiononPILS}, \ref{p:monoidactiononPILSD}, \ref{p:monoidactiononPILSB}, Proposition
\ref{prop:monoid}, and Remark \ref{r:kexplicit}.

\newpage
\begin{exam}\label{exam:bigone}
 It follows from Example \ref{ex:lists} and Theorem \ref{thm:counttypeA} that $|B_{\fgl(2)}\backslash\B_{\fgl(3)}|=13.$  We label the simple roots of $\fgl(3)$ as $\alpha=\alpha_{1}=\eps_{1}-\eps_{2}$ and 
$\beta=\alpha_{2}=\eps_{2}-\eps_{3}$ (see Section \ref{ss:realization}).  
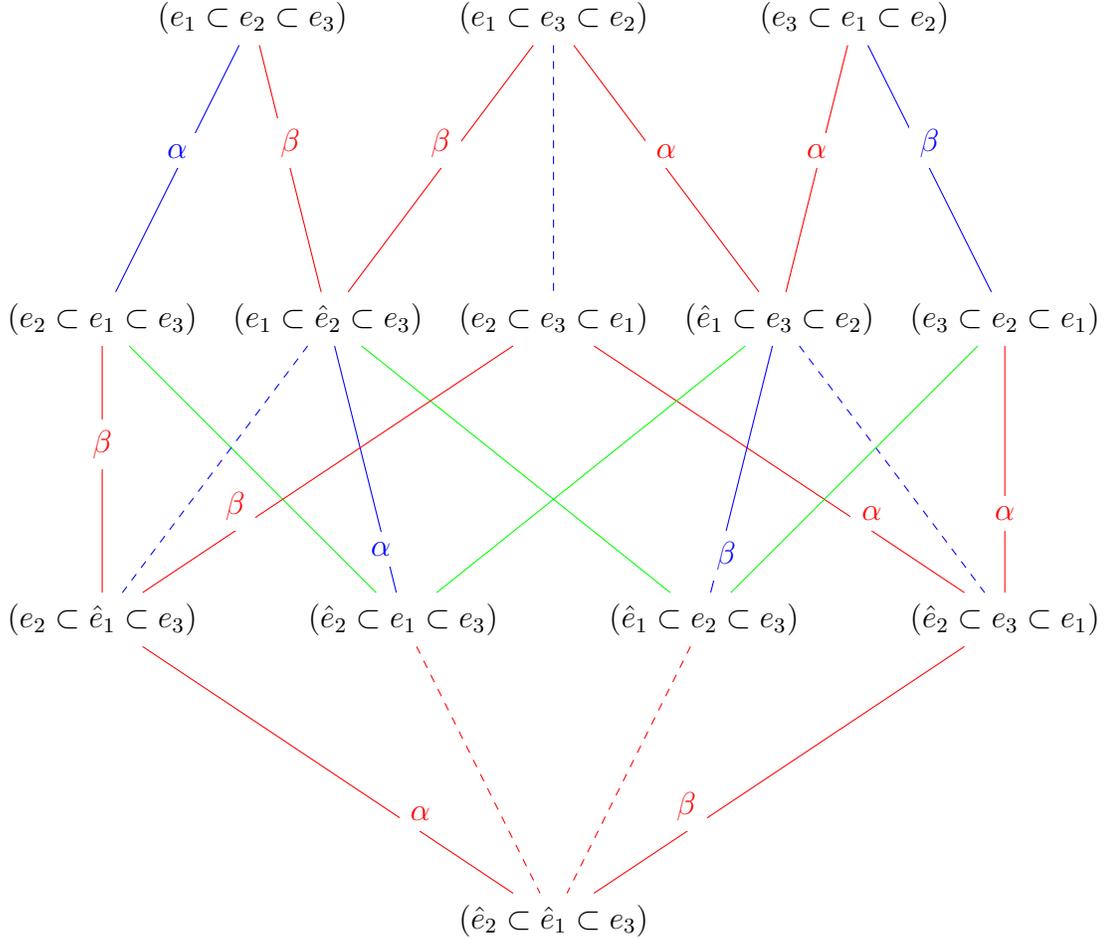
\begin{center}
\begin{tikzpicture}  
 [scale=2.0,auto=center,every node/.style={rectangle,fill=white!20}] 
\node (a1) at  (-1,5) {$(e_{1}\subset e_{2}\subset e_{3})$};
\node (a2) at (1,5) {$(e_{1}\subset e_{3}\subset e_{2})$};
\node (a3) at (3,5) {$(e_{3}\subset e_{1}\subset e_{2}) $};
\node (a4) at (-2,3) {$(e_{2}\subset e_{1}\subset e_{3})$};
\node (a5) at (-0.5,3) {$(e_{1}\subset \hat{e}_{2}\subset e_{3})$};
\node (a6) at (1,3) {$(e_{2}\subset e_{3}\subset e_{1})$};
\node (a7) at (2.5,3) {$(\hat{e}_{1}\subset e_{3}\subset e_{2})$};
\node (a8) at (4,3) {$(e_{3}\subset e_{2}\subset e_{1})$};
\node (a9) at (-2,1) {$(e_{2}\subset \hat{e}_{1}\subset e_{3})$};
\node (a10) at (0,1) {$(\hat{e}_{2}\subset e_{1}\subset e_{3})$};
\node (a11) at (2,1){$(\hat{e}_{1}\subset e_{2}\subset e_{3})$} ;
\node (a12) at (4,1) {$(\hat{e}_{2}\subset e_{3}\subset e_{1})$} ;
\node (a13) at (1,-1) {$(\hat{e}_{2}\subset \hat{e}_{1}\subset e_{3})$};
\draw [blue] (a1) -- (a4) node[midway, above] {$\alpha$}; 
  \draw [red] (a1) -- (a5) node[midway, above] {$\beta$};  
 \draw [red] (a2) -- (a5) node[midway, above] {$\beta$};  
 \draw [dashed] [blue] (a2)--(a6);
 \draw [red] (a2)--(a7) node[midway, above] {$\alpha$};
  \draw [red] (a3)--(a7) node[midway, above] {$\alpha$};
   \draw [blue] (a3)--(a8) node[midway, above] {$\beta$};
    \draw [red] (a4)--(a9) node[midway, above] {$\beta$};
    \draw[green] (a4)--(a10); 
    \draw[dashed][blue] (a5)--(a9);
    \draw[blue] (a5) --(a10) node[near end, below]  {$\alpha$}; 
    \draw[green] (a5)--(a11); 
    \draw[red] (a6)--(a9) node[near end, above] {$\beta$}; 
    \draw[red](a6)--(a12) node[near end, above] {$\alpha$}; 
    \draw[green] (a7)--(a10);
    \draw[dashed][blue] (a7)--(a12);
    \draw [blue] (a7)--(a11) node[near end, below] {$\beta$};
    \draw [red] (a8)--(a12)node[near end, above] {$\alpha$};
    \draw[green] (a8)--(a11); 
    \draw[red](a9)--(a13) node[near end, above] {$\alpha$};
\draw[dashed][red](a10)--(a13); 
\draw[dashed][red](a11)--(a13) ; 
\draw[red](a12)--(a13)node[near end, above] {$\beta$};
\end{tikzpicture}
\end{center}

We examine a few portions of the above graph in more detail.  The orbit 
$Q_{0}=B_{\fgl(2)}\cdot (e_{2}\subset e_{3}\subset e_{1})$ is minimal if we do not 
use the simple roots of $\fk$, but not minimal for the extended monoid action from Section \ref{ss:Bmonoid}.  This is illustrated since there are no solid lines ending in $Q_{0}$, but there is a dotted line ending in $Q_{0}$.
 
Consider the orbit $Q^{\prime}=B_{\fgl(2)}\cdot(e_{2}\subset e_{1}\subset e_{3})$ in the first entry 
of the second row of the graph.  Let $Q^{\prime\prime}=B_{\fgl(2)}\cdot (e_{1}\subset e_{2}\subset e_{3})$, 
and let $Q=B_{\fgl(2)}\cdot(\hat{e}_{2}\subset e_{1}\subset e_{3}).$  In the notation of (\ref{eq:stdconditions}), let $t=s_{\beta}$ and $\vec{s}=s_{\alpha}$.  
We see that $Q^{\prime}\preceq Q$ in the standard order.  However, 
$Q^{\prime}$ and $Q$ are not related in the weak order.  
Indeed, the root $\alpha$ is complex unstable for $Q^{\prime}$ and so
is the root $\eps_{1}-\eps_{2}\in\Pi_{\fgl(2)}$ (see Proposition \ref{prop:monoid}.) 
The root $\beta$ is non-compact for $Q^{\prime}$ with $m(s_{\beta})*Q^{\prime}=B_{\fgl(2)}\cdot (e_{2}\subset \hat{e}_{1}\subset e_{3}).$  The other green lines in the graph can be analyzed in a similar fashion. 

Consider the following subgraph in the third and fourth rows: 
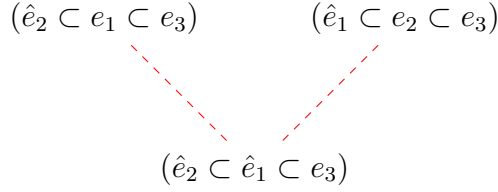
\begin{center}
\begin{tikzpicture}
[scale=2.0,auto=center,every node/.style={rectangle,fill=white!20}]
\node (a1) at (-1,4) {$(\he_{2}\subset e_{1}\subset e_{3})$};
\node (a2) at (1,4) {$(\he_{1}\subset e_{2}\subset e_{3})$};
\node (a3) at (0,3) {$ (\hat{e}_{2}\subset\he_{1}\subset e_{3})$};
\draw[dashed][red](a1)--(a3);
\draw[dashed][red](a2)--(a3);
\end{tikzpicture}
\end{center}
This is the Bruhat graph for the $B_{2}$-orbits contained in the open $GL(2)$-orbit on $\B_{\fgl(3)}$.
Taking $\fl=\fg=\fgl(3)$ as in Remark \ref{r:opencorrespondence} above, these orbits correspond with $B_{1}\cong \C^{\times}$-orbits on $\B_{\fgl(2)}$ with Bruhat graph:  
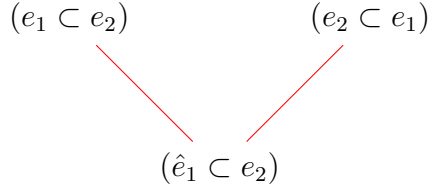
\begin{center}
\begin{tikzpicture}
[scale=2.0,auto=center,every node/.style={rectangle,fill=white!20}]
\node (a1) at (-1,4) {$(e_{1}\subset e_{2})$};
\node (a2) at (1,4) {$(e_{2}\subset e_{1})$};
\node (a3) at (0,3) {$ (\hat{e}_{1}\subset e_{2})$};
\draw[red](a1)--(a3);
\draw[red](a2)--(a3);
\end{tikzpicture}
\end{center}
Note that the left monoid actions by the single root of $\fk=\fgl(2)$ in the first graph correspond to 
right monoid actions in the second graph as predicted by Equation (\ref{eq:lefttoright}).
\end{exam}
\begin{exam}\label{ex:smallone}
It follows from Part (1) of Theorem \ref{thm:orthocount} and Example \ref{ex:spils} that $|B_{\fso(3)}\backslash \B_{\fso(4)}|=5$.  We label the simple roots of $\fso(4)$ as $\alpha=\alpha_{1}=\eps_{1}-\eps_{2}$ and 
$\beta=\alpha_{2}=\eps_{1}+\eps_{2}$.  
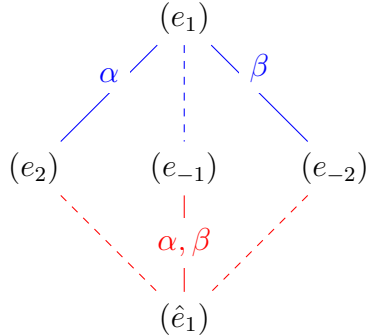
\begin{center}
\begin{tikzpicture}
[scale=2.0,auto=center,every node/.style={rectangle,fill=white!20}]
\node (a1) at  (0,5) {$(e_{1})$};
\node (a2) at (-1,4) {$(e_{2})$};
\node (a3) at (0,4) {$(e_{-1})$};
\node (a4) at (1,4) {$(e_{-2}) $};
\node (a5) at (0,3) {$ (\hat{e}_{1})$};

\draw [dashed] [blue] (a1)--(a3);
\draw [blue] (a1) -- (a2) node[midway, above] {$\alpha$};
\draw [blue] (a1) -- (a4) node[midway, above] {$\beta$};
\draw[dashed] [red] (a2) -- (a5) ;
\draw [red] (a3) -- (a5) node[near end, above] {$\alpha, \beta$};
\draw[dashed] [red] (a4) -- (a5);

\end{tikzpicture}
\end{center}

We observe that the orbit $B_{\fso(3)}\cdot (e_{-1})$ is minimal 
with respect to the right monoid action and one dimensional.   Thus, the left monoid action is also necessary in the orthogonal setting for Theorem \ref{thm:bigthm}
to hold.  
\end{exam}
\begin{exam}\label{ex:bigorthoexam}
It follows from Part (2) of Theorem \ref{thm:orthocount} and Example \ref{ex:spils} that 
$|B_{\fso(4)}\backslash \B_{\fso(5)}|=17$.  We label the simple roots of $\fg=\fso(5)$ 
as $\alpha=\alpha_{1}=\eps_{1}-\eps_{2}$ and $\beta=\alpha_{2}=\eps_{2}$.  We label the simple
roots of $\fk=\fso(4)$ as $\tilde{\alpha}_{1}=\eps_{1}-\eps_{2}$ and $\tilde{\beta}=\eps_{1}+\eps_{2}$.

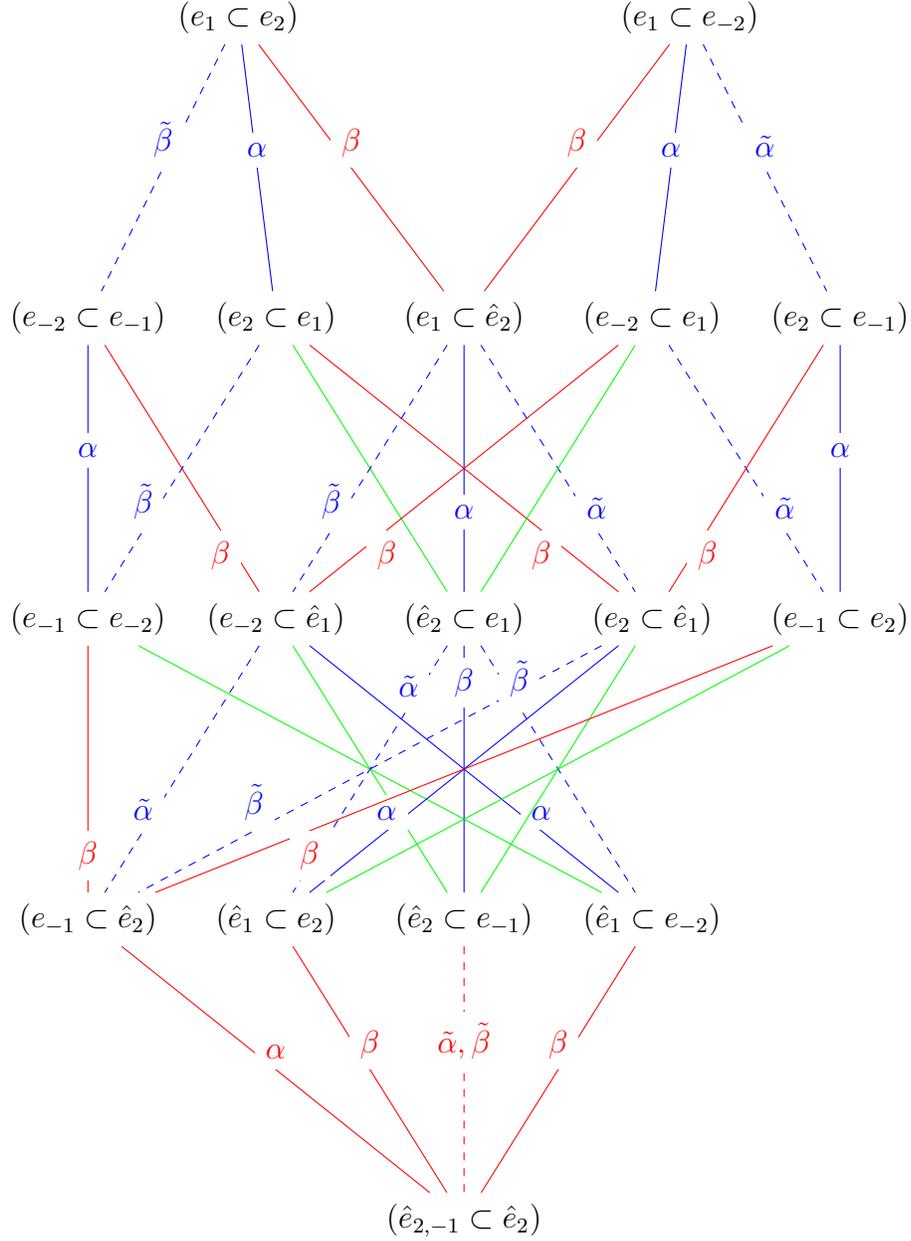
\begin{center}
\begin{tikzpicture}  
 [scale=2.0,auto=center,every node/.style={rectangle,fill=white!20}] 
\node (a1) at  (0,5) {$(e_{1}\subset e_{2})$};
\node (a2) at (3,5) {$(e_{1}\subset e_{-2})$};
\node (a3) at (-1,3) {$(e_{-2}\subset e_{-1} )$};
\node (a4) at (0.25,3) {$(e_{2}\subset e_{1})$};
\node (a5) at (1.5,3) {$(e_{1}\subset \hat{e}_{2})$};
\node (a6) at (2.75,3) {$(e_{-2}\subset e_{1})$};
\node (a7) at (4,3) {$(e_{2}\subset e_{-1})$};
\node (a8) at (-1,1) {$(e_{-1}\subset e_{-2})$};
\node (a9) at (0.25,1) {$(e_{-2}\subset \hat{e}_{1})$};
\node (a10) at (1.5,1) {$(\hat{e}_{2}\subset e_{1})$};
\node (a11) at (2.75,1){$(e_{2}\subset \hat{e}_{1})$} ;
\node (a12) at (4,1) {$(e_{-1}\subset e_{2})$} ;
\node (a13) at (-1,-1) {$(e_{-1}\subset \hat{e}_{2})$};
\node (a14) at (0.25,-1) {$(\hat{e}_{1}\subset e_{2})$};
\node (a15) at (1.5, -1) {$(\hat{e}_{2}\subset e_{-1})$}; 
\node (a16) at (2.75, -1) {$(\hat{e}_{1}\subset e_{-2})$}; 
\node (a17) at (1.5,-3) {$(\hat{e}_{2,-1}\subset \hat{e}_{2})$}; 

\draw [dashed][blue] (a1)--(a3) node[midway, above]{$\tilde{\beta}$}; 
\draw [blue] (a1)--(a4) node[midway, above]{$\alpha$}; 
\draw [red] (a1)--(a5) node[midway, above]{$\beta$};
\draw [red] (a2)--(a5) node[midway, above]{$\beta$};
\draw [blue] (a2)--(a6) node[midway, above]{$\alpha$}; 
\draw [dashed][blue] (a2)--(a7) node[midway, above]{$\tilde{\alpha}$}; 
\draw [blue] (a3)--(a8) node[midway, above]{$\alpha$}; 
\draw [red] (a3)--(a9) node[near end, below]{$\beta$};
\draw [dashed][blue] (a4)--(a8) node[near end, above]{$\tilde{\beta}$}; 
\draw [red] (a4)--(a11) node[near end, below]{$\beta$};
\draw [green] (a4)--(a10) ;
\draw [dashed][blue] (a5)--(a9) node[near end, above]{$\tilde{\beta}$}; 
\draw [dashed][blue] (a5)--(a11) node[near end, above]{$\tilde{\alpha}$}; 
\draw [blue] (a5)--(a10) node[near end, above]{$\alpha$}; 
\draw [dashed] [blue] (a6)--(a12) node[near end, above]{$\tilde{\alpha}$}; 
\draw[green] (a6)--(a10); 
\draw [red] (a6)--(a9) node[near end, below]{$\beta$};
\draw [red] (a7)--(a11) node[near end, below]{$\beta$};
\draw [blue] (a7)--(a12) node[midway, above]{$\alpha$}; 
\draw [red] (a8)--(a13) node[near end, below]{$\beta$};
\draw [green] (a8)--(a16);
\draw [dashed][blue] (a9)--(a13) node[near end, above]{$\tilde{\alpha}$}; 
\draw [blue] (a9)--(a16) node[near end, above]{$\alpha$}; 
\draw[green] (a9)--(a15); 
\draw [dashed][blue] (a11)--(a13) node[near end, above]{$\tilde{\beta}$}; 
\draw [blue] (a11)--(a14) node[near end, above]{$\alpha$};
\draw [green] (a11)--(a15);  
\draw [dashed][blue] (a10)--(a14) node[near start, above]{$\tilde{\alpha}$}; 
\draw [blue] (a10)--(a15) node[near start, above]{$\beta$};
\draw [dashed][blue] (a10)--(a16) node[near start, above]{$\tilde{\beta}$}; 
\draw [red] (a12)--(a13) node[near end, below]{$\beta$};
\draw[green] (a12)--(a14);
\draw [red] (a13)--(a17) node[midway, above]{$\alpha$}; 
\draw [red] (a14)--(a17) node[midway, above]{$\beta$}; 
\draw [dashed][red] (a15)--(a17) node[midway, above]{$\tilde{\alpha},\tilde{\beta}$}; 
\draw [red] (a16)--(a17) node[midway, above]{$\beta$}; 
\end{tikzpicture}
\end{center}

Consider the orbits $Q=B_{\fso(4)}\cdot (\hat{e}_{1}\subset e_{-2})$ and 
$Q^{\prime}=B_{\fso(4)}\cdot (e_{-1}\subset e_{-2})$ (the last orbit in the fourth row and the first orbit in the third row respectively).  We claim 
$Q^{\prime}\preceq Q$ in the standard order, but $Q^{\prime}\not \leq_{w} Q$.  
Indeed, if we let $Q_{+}=B_{\fso(4)}\cdot (e_{1}\subset e_{2})$, then 
$Q^{\prime}=m(\vec{s})*Q_{+}$, where $\vec{s}$ is the sequence 
$\vec{s}=(s_{\alpha},\, s_{\tilde{\beta}})$ (see Equation (\ref{eq:Ssequence})).  
The orbit $Q$ is $Q=m(\vec{s})*m(s_{\beta})*Q_{+}$.  So it follows
from the definition of the standard order (\ref{eq:stdconditions}) that 
$Q^{\prime}\preceq Q$, but from the diagram it is clear that $Q^{\prime}\not\leq_{w} Q$.  


Consider the subgraph in the bottom right corner:  
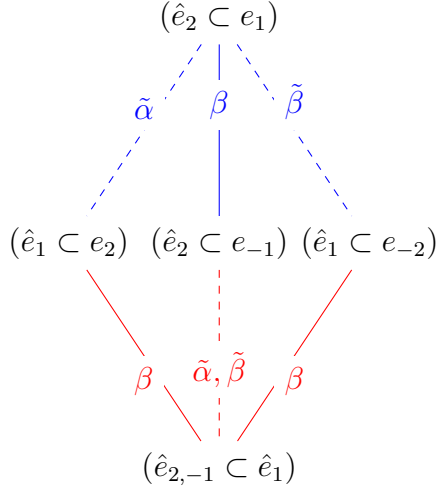
\begin{center}
\begin{tikzpicture}
[scale=2.0,auto=center,every node/.style={rectangle,fill=white!20}]
\node (a1) at  (0,5) {$(\hat{e}_{2}\subset e_{1})$};
\node (a2) at (-1,3.5) {$(\hat{e}_{1}\subset e_{2})$};
\node (a3) at (0,3.5) {$(\hat{e}_{2}\subset e_{-1})$};
\node (a4) at (1,3.5) {$(\hat{e}_{1}\subset e_{-2})$};
\node (a5) at (0,2) {$ (\hat{e}_{2,-1}\subset \hat{e}_{1})$};

\draw [dashed] [blue] (a1) -- (a2) node[midway, above] {$\tilde{\alpha}$};
\draw  [blue] (a1)--(a3) node[midway, above] {$\beta$};
\draw [dashed] [blue] (a1) -- (a4) node[midway, above] {$\tilde{\beta}$};
\draw [red] (a2) -- (a5) node[midway, below] {$\beta$};
\draw [dashed] [red] (a3) -- (a5) node[near end, above] {$\tilde{\alpha},\tilde{ \beta}$};
\draw  [red] (a4) -- (a5) node[midway, below] {$\beta$};
\end{tikzpicture}
\end{center}
This graph is the graph of the weak order 
for $B_{\fso(4)}$-orbits contained in the open $K=SO(4)$-orbit on $\B_{\fso(5)}$.  
By Remark \ref{r:opencorrespondence}, these orbits are in one-to-one correspondence with  
$B_{\fso(3)}\backslash \B_{\fso(4)}$ whose Bruhat graph is given in Example \ref{ex:smallone}.  
But notice that the right monoid actions in Example \ref{ex:smallone} correspond to left monoid actions here, 
and the left monoid actions in Example \ref{ex:smallone} correspond to right monoid actions in this graph.  
This is precisely the phenomenon described in Equations (\ref{eq:transfer}) and (\ref{eq:lefttoright}) where in this example $\fl=\fg$ and the element $\tilde{w}\in W$ in (\ref{eq:guyinW}) is the identity.  Note that we also see the shift from the root $\eps_{1}-\eps_{2}$ of $\fso(3)$ in Example \ref{ex:smallone} to the root $\beta$ here described in Equation (\ref{eq:transfer}).

\end{exam}

\bibliographystyle{amsalpha.bst}

\bibliography{bibliography-1}

\end{document}